\documentclass[onefignum,onetabnum]{siamart190516}
\usepackage{epsfig,amssymb,hyperref,enumitem, ulem, wrapfig, centernot}
\usepackage{bm,bbm,mathtools}
\usepackage[scr=boondoxo]{mathalfa}
\usepackage{caption}
\usepackage{subcaption}
\usepackage{todonotes}

\setlist[enumerate]{itemsep=5pt, topsep=5pt,after=\vspace{3pt}}
\setlist[itemize]{itemsep=5pt, topsep=5pt,after=\vspace{3pt}}
\newtheorem{property}{Numerically Verified Properties}
\theoremstyle{definition}
\newtheorem{remark}[theorem]{Remark}
\newtheorem{assump}[property]{Assumptions}

\newtheorem{conjecture}{Conjecture}

\renewcommand{\l}{\left}
\renewcommand{\r}{\right}

\newcommand{\mbf}{\mathbf}

\newcommand{\bz}{{\bf z}}

\newcommand{\E}{\mathbb E\,}
\newcommand{\Z}{\mathbb Z}
\newcommand{\ol}{\overline}
\newcommand{\bml}{{\bm\lambda}}

\newcommand{\calZ}{\mathcal Z}

\newcommand{\R}{\mathbb R}

\renewcommand{\L}{\mathcal L}

\newcommand{\noin}{\noindent}
\newcommand{\e}{\,\mathrm{exp}}

\newcommand{\Var}{\mathrm{Var}}
\newcommand{\Cov}{\mathrm{Cov}}
\newcommand{\Corr}{\mathrm{Corr}}

\newcommand{\loc}{{\idxsetx}}
\newcommand{\Law}{{\mathrm{Law}}}
\newcommand{\eV}{\eqref{V}}
\newcommand{\eE}{\eqref{E}}
\newcommand{\eEp}{\eqref{Ep}}
\newcommand{\eEf}{\eqref{Ef}}
\newcommand{\eEfp}{\eqref{Efp}}

\newcommand{\beq}{\begin{equation}}
\newcommand{\eeq}{\end{equation}}
\newcommand{\beqn}{\begin{equation*}}
\newcommand{\eeqn}{\end{equation*}}
\def\beqs#1\eeqs{%
    \begin{equation}\begin{split}%
    #1%
    \end{split}\end{equation}%
}
\def\beqsn#1\eeqsn{%
    \begin{equation*}\begin{split}%
    #1%
    \end{split}\end{equation*}%
}
\newcommand{\unit}{{[0,1]}}
\newcommand{\cts}{C(\unit)}

\newcommand{\idxsetx}{{N(x\pm\epsilon)}}

\newcommand{\argidxsetx}[1]{{#1\in \idxsetx}}

\newcommand{\muN}{\mu^N}
\newcommand{\rhoN}{\rho^N}
\newcommand{\muinfty}{\mu^\infty_K}

\newcommand{\bardlamloc}{\ol{\bm\omega}_\idxsetx}

\newcommand{\gibbs}[2]{\rho_{#1}[#2]}
\newcommand{\mugibbs}[3]{\mu_{#1}[#2, #3]}

\newcommand{\Nepslim}{N\to\infty,\,\epsilon\to0}
\newcommand{\limNeps}{\lim_{\epsilon\to0}\lim_{N\to\infty}}

\renewcommand{\v}{v}
\newcommand{\h}{h}

\newcommand{\w}{w}
\newcommand{\vN}{\mathbf v_N}
\newcommand{\wN}{\mathbf w_N}
\newcommand{\zN}{\mathbf z_N}
\newcommand{\hN}{\mathbf h_N}
\newcommand{ \vNtild}{\tilde{\mathbf v}_N}
\newcommand{\hNtild}{\tilde{\mathbf h}_N}
\newcommand{\zNtild}{\tilde{\mathbf z}_N}
\newcommand{ \wNtild}{\tilde{\mathbf w}_N}
\newcommand{\LN}{\L_N}
\title{The Local Equilibrium State of a Crystal Surface Jump Process in the Rough Scaling Regime\thanks{\funding{This material is based upon work supported by the U.S. Department of Energy, Office of Science, Office of Advanced Scientific Computing Research, Department of Energy Computational Science Graduate Fellowship under Award Number DE-FG02-97ER25308}}}
\author{Anya Katsevich\thanks{Department of Mathematics, Courant Institute of Mathematical Sciences, New York University. \email{katsevich@cims.nyu.edu}}}
\begin{document}

\normalem
\maketitle
\begin{abstract}
We investigate the local equilibrium (LE) distribution of a crystal surface jump process as it approaches its hydrodynamic (continuum) limit in a nonstandard scaling regime introduced by Marzuola and Weare. The atypical scaling leads to a local equilibrium state whose structure is novel, to the best of our knowledge. The distinguishing characteristic of the new, \emph{rough} LE state is that the ensemble average of single lattice site observables do not vary smoothly across lattice sites. We investigate numerically and analytically how the rough LE state affects the convergence mechanism via three key limits, and show that by comparison, more standard, ``smooth" LE states satisfy stronger versions of these limits. 
\end{abstract}



\section{Introduction} 

Consider an evolving interacting particle system ${\bf w}_N(t)=(w_{1,N}(t),\dots, w_{N,N}(t))$ on a lattice, with $w_{j,N}(t)$ particles at lattice site $j$.  The system evolves through particles randomly jumping from site to site with certain jump rates, which are functions of the local particle configuration. Now imagine these interactions occur on the unit interval, with lattice site $j$ corresponding to the interval $[j/N, (j+1)/N)$. If we also speed up time by an appropriate power of $N$ and take $N\to\infty$, we can hope to get a \emph{hydrodynamic limit}: a function $\w(t,x)$ with a nontrivial macroscopic dynamics, governed by a PDE, emerging from the microscopic dynamics as $N\to\infty$. The hydrodynamic limit problem is to prove that a microscopic process ${\bf w}_N(t)$ has a macroscopic limit $w(t,x)$ under a certain scaling of space and time, and to determine the PDE governing this macroscopic limit. 
 
The two leading proof methods which have been successfully adapted to a variety of rigorous hydrodynamic limit proofs, are the entropy production approach~\cite{varadhanI} and the relative entropy approach~\cite{Yau1991} . For a review of these approaches and their application to proving hydodynamic limits of jump processes (our interest here), see~\cite{kipnisbook}. Both methods quantify in some sense the proximity of the distribution of the microscopic process to a certain local equilibrium (LE) state. Generally speaking, this is the key ingredient in the analysis: showing that following an initial burn-in time, the distribution of the microscopic process adheres closely to a certain LE state when $N\gg1$. The LE state is the fundamental object which enables macroscopic motion to emerge out of the large number of microscopic interactions.

Here, we study the hydrodynamic limit of an interacting particle system (IPS) modeling crystal surface relaxation in a nonstandard scaling regime. We will show that the IPS has a wholly novel kind of LE state, which compels us to rethink the convergence mechanism of this system to its macroscopic limit. Our goal is to describe this new kind of LE state and to explain the implications it has for convergence to the macroscopic limit. To achieve this goal, we will study the convergence problem through the lens of three limits. The first two limits are necessary for the convergence of $\wN$ to $w$ in a sense that is physically intuitive and related to the kind of convergence shown in rigorous hydrodynamic limit proofs. The third limit is a key ingredient in deriving the PDE governing the limit $w$. Understanding how the limits follow from properties of the LE state will give us valuable intuition.

For concreteness, let us informally describe these limits, which we call~\eV,~\eE, and~\eEf, condensing them into two. Let $w_i = w_{i,N}(N^\alpha t)$ for an appropriate time scaling $N^\alpha$. The limits~\eqref{V} and~\eqref{E} express that a mesoscopic average of the $w_i$, over $i$ such that $i/N\in B_\epsilon(x)=(x-\epsilon,x+\epsilon)$, converges to $\w(t,x)$, for a continuous function $w:\unit\to\R$.
\beqn 
\text{(V, E):}\quad \limNeps \frac{1}{2N\epsilon}\sum_{\frac iN\in B_\epsilon(x)} w_i = \w(t,x).
\eeqn
The limit~\eqref{Ef} states that there is a function $\hat r$, independent of $t$ and $x$, such that
\beqn
 \text{(Ef)} \quad\frac{1}{2N\epsilon}\sum_{\frac iN\in B_\epsilon(x)} r(w_i)\stackrel{N,\epsilon}{\approx} \hat r\bigg(\frac{1}{2N\epsilon}\sum_{\frac iN\in B_\epsilon(x)}w_i\bigg),
\eeqn
i.e. the difference between the left- and righthand sides goes to zero as $N\to\infty$ and then $\epsilon\to0$. The terms $r(w_i)$ are the jump rates at site $i$, and we can think of the expression on the left as an average jump rate in the region $|i-Nx|\leq N\epsilon$. To motivate studying this third limit~\eqref{Ef}, we show informally that given~\eqref{V} and~\eqref{E}, the PDE ``nearly" follows from the additional limit~\eqref{Ef}.

If the distribution of $\wN$ is given exactly by a prototypical LE state (such as a slowly varying local Gibbs distribution), then the expectations $\E[w_i]$ vary smoothly over lattice sites $i$, and the marginals $\Law(w_i)$ belong to a family of measures which can be parameterized by their means. Thanks to these and a few other properties,~\eV,~\eE, and~\eEf~are easily satisfied. In fact,~\eV~and stronger limits~\eEp~and \eEfp~hold for the prototypical LE. Although for finite $N$, $\Law(\wN)$ is never exactly a prototypical LE state, we show numerically for a simple, zero range process that $\Law(\wN)$ does satisfy the stronger limits. We should note that investigation of standard LE states is a secondary focus; the numerical results are shown primarily for the purpose of contrasting the standard LE to that of our crystal surface process. 

The essential difference is that in our case the $\Law(w_i)$ do \emph{not} belong to a mean-parameterized measure family, and $\E w_i$ do \emph{not} vary smoothly with $i$, nor do most observables $\E f(w_i)$. It would be unsurprising if the random $w_i$ varied roughly (a term we will define formally), but the rough variation of expectations is surprising and novel. The expectations $\E w_i$ are rough even if the initial condition $\wN(0)$ is chosen so that $\E w_i(0) = w_0(i/N)$ for a smooth function $w_0(x)$. 

Our first main contribution is \textbf{to show numerically that due to the rough behavior, the stronger limits are not satisfied, but despite the rough behavior,~\eV,~\eE, and~\eEf~are still satisfied.} We call this new kind of LE a ``rough" LE state. Our second main contribution is theoretical. Using the numerically verified~\eE, and the additional numerically verified fact that $\Law(\wN)$ is induced by a local Gibbs measure, \textbf{we rigorously prove that mesoscopic averages of the distributions $\Law(w_i)$ are measures which do belong to a certain mean-parameterized family, unlike the $\Law(w_i)$ themselves.} This gives valuable insight about why~\eEf~holds for a rough LE state. 

We now give some background on the crystal surface process. Let ${\bf h}_N(t)=(h_{1,N}(t), \dots, h_{N,N}(t))$, where $h_{i,N}$ is the (discrete) height of a column, or number of particles, at site $i$. The process evolves via jumps between neighboring lattice sites which tend to lower the surface energy on average. These jumps occur with so-called ``Arrhenius" transition rates. For a review of the physics of crystal surfaces, including the Arrhenius rates, see~\cite{zangwill1988physics}. The process $\wN(t)$ studied here tracks the ``curvature" of this crystal surface height profile, or more precisely, the second order finite differences $w_{i,N} = h_{i+1,N}-2h_{i,N} + h_{i-1,N}$. This curvature process is expected to have a nontrivial hydrodynamic limit provided ${\bf h}_N$ is scaled to have order $O(N^2)$ heights (see Definition~\ref{def:init},~\ref{def:hydro} for precise statements). Marzuola and Weare introduced this scaling in~\cite{mw-krug}, naming it the ``rough" scaling regime for reasons independent of the rough LE discovered here. The authors use probabilistic but heuristic arguments to show that the hydrodynamic limit of $N^{-2}{\bf h}_N$ is a fourth order diffusion with exponential nonlinearity. Linearizing this exponential gives the PDE limit of the height process in the more standard, \emph{smooth} scaling regime, first derived by Krug, Dobbs, and Majaniemi using physical arguments~\cite{krug1995adatom}. Recently, the rough scaling limit of an Arrhenius rate process with additional deposition and evaporation has also been derived~\cite{gao2020_arrPDE}, yielding a nonlinear fourth order PDE with a second order correction.

So far, none of these hydrodynamic limits have been proven rigorously, due in part to technical obstacles posed by the discrete heights. On the other hand, there are several continuous height variants of the crystal surface model for which rigorous hydrodynamic limits have been carried out. In~\cite{funaki1997motion}, Funaki and Spohn prove the second order hydrodynamic limit of the Ginzburg-Landau $\nabla\phi$ interface model without conservation of mass. Using our notation, this means $\sum_ih_{i,N}(t)$ can vary. In~\cite{Nishikawa}, Nishikawa proves the fourth order limit of the analogous dynamics with $\sum_ih_{i,N}(t)$ conserved. In~\cite{savu}, Savu derives the hydrodynamic limit of a ``nongradient" form of the continuous height model. 

 
The dynamics of the $\wN$ curvature process can be formulated independently of the $\hN$ height process, so that studying the hydrodynamic limit of $\wN$ independently of $\hN$ is legitimate. In some sense in fact, the $\hN$ limit is the one which depends on the $\wN$ limit. 
As we argue in Section~\ref{subsec:arr-discush}, the rough LE state of $\wN$ explains an interesting difference noted in~\cite{mw-krug} between the rough and smooth PDE limits of the crystal surface height process: the differential operator of the smooth scaling PDE involves a function which ``sees'' the discreteness of the microscopic system, while the rough scaling differential operator involves an analogous function which does not. \emph{These and other insights significantly expand the current understanding of the convergence mechanism for the crystal surface process under rough scaling.} In particular, we correct some of the claims in~\cite{mw-krug} used to derive the PDE limit of the height process.

Finally, we note that this new, ``rough" LE is not limited to the curvature process of the Arrhenius rate crystal surface dynamics in the rough scaling regime. In the conclusion, we mention a second process with Metropolis-type jump rates, that has the same kind of LE structure. This is the subject of an upcoming paper~\cite{metrop_pde}.

\subsection{Organization} In Section~\ref{sec:model}, we introduce the crystal surface model, define the notions of scaling regime and hydrodynamic limit, and review existing results on the crystal PDE limits. In Section~\ref{sec:reduc}, we explain our focus on three key limits. We then review the structure of standard LE states in Section~\ref{sec:smooth-LE}, and give numerical evidence for the hypothesis that they satisfy the stronger limits~\eEp~and~\eEfp. Section~\ref{sec:rough-LE} contains our numerical case study of $\wN$'s rough LE state. In Section~\ref{sec:theory}, we first state and numerically verify the explicit form of $\Law(w_i)$. We then use this form, in addition to the numerically confirmed~\eE, to rigorously prove that~\eEf~holds thanks to mesoscopic averaging of $\Law(w_i)$. We make some concluding remarks in Section~\ref{sec:conclude}. The proofs and details of numerical simulations can be found in the appendix.

For maximal clarity, let us summarize the type of result (informal, numerical, or rigorous) presented in each section. Section~\ref{sec:model} is a review which includes informal claims. The derivation of the PDE in Section~\ref{sec:reduc} is also non-rigorous. It is intended only to motivate our study of~\eEf~in addition to~\eV~and~\eE. All claims made in Sections~\ref{sec:smooth-LE} and~\ref{sec:rough-LE} are confirmed numerically. All statements in Sections~\ref{subsec:omega-tilde-w} and~\ref{subsec:marg-av} are proven rigorously, using numerically verified assumptions stated explicitly in Section~\ref{subsec:arr-prop}. Section~\ref{subsec:arr-discush} is again informal. 

\subsection{Notation}
Given a sequence of vectors $\mbf w_N\in\Z^N$, $N=1,2,\dots$, we will suppress the dependence of the entries of $\mbf w_N$ on $N$, writing $\mbf w_N = (w_1,\dots, w_N)$. For a probability mass function (pmf) $\mu$ over the integers, we write $\mu(f)$ to denote $\sum_nf(n)\mu(n)$, with the exception of $f(n)=n$. Instead we denote the first moment of $\mu$ by $m_1(\mu)$; that is, $m_1(\mu) =\sum_n n\mu(n)$. We will denote the parameter $\lambda$ of a pmf in a parameterized family using square brackets:  e.g. $\mu[\lambda]\in \{\mu[\lambda]\mid\lambda\in\R\}$. The probability of $n$ under $\mu[\lambda]$ is then $\mu[\lambda](n)$. We let $\unit$ denote the unit torus, and write $i\in\idxsetx$ to denote an integer $i$ such that $|i-Nx|<N\epsilon$. Also, $\cts$ is the space of continuous functions on the unit torus. Further notation will be introduced as needed. 

\section*{Acknowledgments} I would like to thank Jonathan Weare for his guidance, patience, and insight throughout the research and writing process. I would also like to thank Jeremy Marzuola, who introduced me to the crystal surface problem and with whom I have had many useful discussions throughout the years. Finally, I am grateful for the computing resources that NYU High Performance Computing made available to me, and for the support of the DOE CSGF.

\section{Microscopic Model and Hydrodynamic Limit}\label{sec:model} In~\ref{subsec:dynamics}, we introduce the crystal surface height, slope, and curvature process, describe the generators of these processes, and the invariant measure of the slope process. In~\ref{sec:hydro}, we define the notions of scaling regime and hydrodynamic limit, and review existing results on the PDE limits of the smooth and rough scaled crystal surface process.

\subsection{Microscopic Model}\label{subsec:dynamics} Consider a crystal surface as an arrangement of particles on a periodic lattice, as shown schematically in Figure~\ref{fig:schematic}. We represent the crystal surface height profile by a Markov jump process $(\hNtild(t))_{t\geq0} = (\tilde h_1(t),\dots, \tilde h_N(t))_{t\geq0}$ on $\Z^N$. If $\tilde h_i\geq 0$ then $\tilde h_i$ represents the number of particles stacked above level zero at lattice site $i$, and if $\tilde h_i<0$ then $|\tilde h_i|$ represents the number of particles ``missing" below level zero at lattice site $i$, shown as black dotted circles in the figure. Since the lattice is periodic, we will use $\bmod$ $N$ indexing throughout. 

The surface evolves through particle jumps between neighboring columns. Suppose a particle at the top of column $i$ jumps to a neighboring column $j=i\pm 1$, as depicted in the figure (the red dotted circles denote the possible locations of the particle post-jump). If the particle configuration pre-jump is $h$, then the particle configuration post-jump is $h^{i,j}$ defined as
\beq
(h^{i,j})_k = \begin{cases}
h_i-1,\quad &k=i,\\
h_j+1,\quad &k=j,\\
h_k,\quad&\text{otherwise.}
\end{cases}
\eeq


\begin{wrapfigure}{r}{0.5\textwidth}
  \begin{center}
    \includegraphics[width=0.5\textwidth]{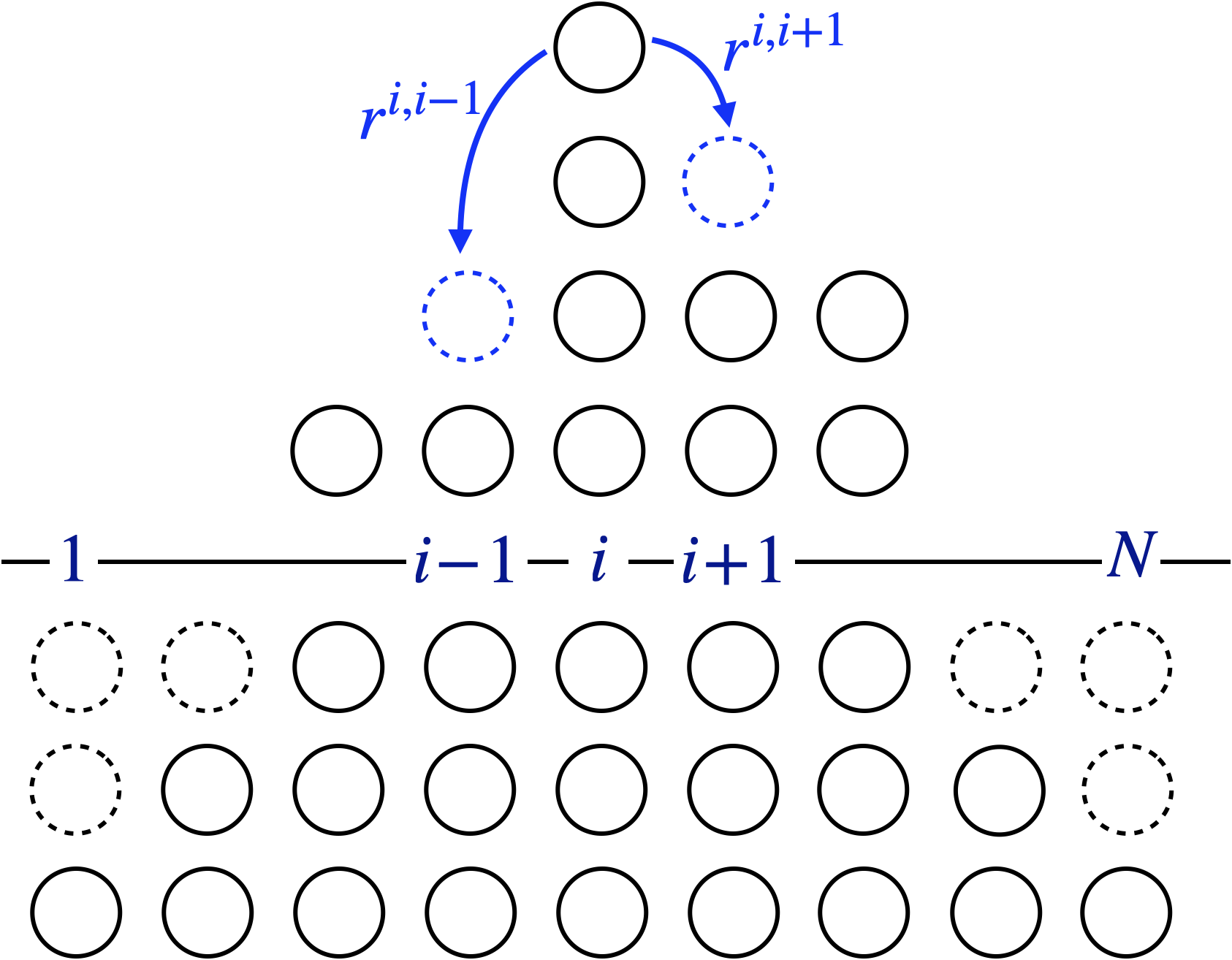}
  \end{center}
 \caption{Schematic of crystal surface height process $\hNtild$.}
 \label{fig:schematic}
 \vspace{-0.4cm}
\end{wrapfigure}

These transitions $h\mapsto h^{i,j}$ are the only transitions which occur in the dynamics. The slope process is given by $\zNtild(t) = (\tilde z_1(t),\dots, \tilde z_N(t))$, where $\tilde z_i = \tilde h_{i+1}-\tilde h_i$. In this paper, we study the ``curvature" process $\wNtild(t) = (\tilde w_1(t),\dots, \tilde w_N(t)),$ where $$\tilde w_i = \tilde z_i-\tilde z_{i-1}= \tilde h_{i+1}-2\tilde h_i+\tilde h_{i-1}.$$ 

If $\mbf h$ jumps to $\mbf h^{i,j}$, then we write that $\mbf z$ jumps to $\mbf z^{i,j}$, and $\mbf w$ jumps to $\mbf w^{i,j}$. Here, $\mbf z^{i,j}$ ($\mbf w^{i,j}$) denotes the slope profile (curvature profile) induced by $\mbf h^{i,j}$, where $\mbf z$ is the slope of $\mbf h$ ($\mbf w$ is the ``curvature" of $\mbf h$). It is straightforward to see that $\mbf z^{i,j}$ and $\mbf w^{i,j}$ depend only on $\mbf z$ and $\mbf w$, respectively. For example, 
\beq\label{w-jump}(w^{i,i+1})_k = 
\begin{cases}
w_{k}-1,\quad &k=i-1,\\
w_k+3,\quad &k=i,\\
w_k-3,\quad &k=i+1,\\
w_k+1, \quad &k=i+2,\\
w_k,\quad &\text{otherwise,}
\end{cases}
\quad (w^{i+1,i})_k = 
\begin{cases}
w_{k}+1,\quad &k=i-1,\\
w_k-3,\quad &k=i,\\
w_k+3,\quad &k=i+1,\\
w_k-1, \quad &k=i+2,\\
w_k,\quad&\text{otherwise.}
\end{cases}
\eeq
Now, to each height profile we associate an energy, or Hamiltonian, that depends only on the corresponding slope profile,
$$H(\mbf h) = H(\mbf z) = \sum_{i=1}^Nz_i^2.$$ The quadratic potential is a modeling choice that makes the subsequent analysis easier. Another common choice is the absolute-value potential. Let us now define the Arrhenius transition rates. The rate of a jump from $i$ to $i+1$ is equal to the rate of a jump from $i$ to $i-1$, and is given by
\beq\label{arr}r_i(\mbf h) =\e\l(-K\l[H(J_i\mbf h) - H(\mbf h)\r]\r),\eeq where $K$ is a system parameter which can be thought of as inverse temperature, and $J_i\mbf h$ is obtained from $\mbf h$ by decreasing the height at site $i$ by 1. The physical intuition underlying these rates is that in order for a particle at site $i$ to jump to a neighboring column $j=i\pm 1$, it first breaks the bonds holding it in place, and then jumps left or right with equal probability. The amount of energy needed to break the bonds is the energy change induced by removing the particle, i.e. $H(J_i\mbf h) - H(\mbf h)$. Note that $\mbf h\mapsto J_i\mbf h$ is not itself an allowable transition under the dynamics. Define \beq\label{r-of-w} r(w)=e^{-2K-2Kw}.\eeq  Working out the transition rates explicitly by expanding the Hamiltonians, we get $$r_i(\mbf h) = e^{-2K-2K(z_i-z_{i-1})} = r(w_i).$$ 

Since the jumps $\mbf h\mapsto \mbf h^{i,i\pm1}$ occur with rates $r_i(\mbf h)=r(w_i)$,  depending only on the curvatures $w_i$, it follows that $\zNtild(t)$ and $\wNtild(t)$ are both also Markov jump processes which can be defined independently of $\hNtild(t)$. For example the dynamics of $\wNtild(t)$ is given by the Markov jump process with jumps $\mbf w\mapsto \mbf w^{i,i\pm1}$ occurring at rates $r(w_i)$, $i=1,\dots, N$. More formally, the dynamics of $\wNtild$ is given by the generator
\beqs\label{w-gen}(\LN f)(\mbf w) &=\sum_{i=1}^Nr(w_i)\l(\l[f(\mbf w^{i,i+1})-f(\mbf w)\r]+\l[f(\mbf w^{i,i-1})-f(\mbf w)\r]\r).\eeqs For example, let $\pi_j:(w_1,\dots, w_N)\mapsto w_j$ denote the projection onto the $j$th coordinate. Recalling the transition operators $\mbf w\mapsto \mbf w^{i,j}$ in~\eqref{w-jump}, we compute that
\beq\label{L-pi-w}
(\LN\pi_j)(\mbf w) = r(w_{j-2})-4r(w_{j-1}) +6r(w_j)- 4r(w_{j+1}) + r(w_{j+2}).
\eeq

The generator determines the evolution of an initial measure. For example, if $\wNtild(0) \sim \muN_0$ then 
$$\wNtild(t) \sim\muN_0 e^{t\LN},$$ where multiplication on the right denotes application of the adjoint operator. The process $\zNtild$ has the special property --- not shared by $\hNtild$ and $\wNtild$ --- that it is invariant with respect to a product measure, specifically the Gibbs measure
\beq\label{Phi}\Phi_N(\mbf z) =\prod_{i=1}^N\frac{\e\l(-Kz_i^2\r)}{\calZ} = \frac{e^{-KH(\mbf z)}}{\calZ^N},\quad \calZ = \sum_{n\in \Z}\e\l(-Kn^2\r).\eeq The invariance of $\Phi_N$ under the $\zNtild$ dynamics is a consequence of detailed balance:
$$r_{i}(\mbf z)\Phi_N(\mbf z) = r_{i+1}(\mbf z^{i,i+1})\Phi_N(\mbf z^{i,i+1})\quad\forall i=1,\dots,N,\,\forall \mbf z\in\Z^N.$$ That the Arrhenius rates satisfy detailed balance is straightforward to see from their original formulation~\eqref{arr} in terms of Hamiltonians. Also, one can show that the transitions $\mbf z\mapsto \mbf z^{i,j}$ do not change the values of $S_0(\mbf z)=\sum_kz_k$ and $S_1(\mbf z)=\sum_kkz_k$, so the ergodic measures are the restrictions of $\Phi_N$ to $(S_0,S_1)$ level sets.

\subsection{Hydrodynamic Limit}\label{sec:hydro}
Let us take a step back to discuss the connected notions of scaling regime and hydrodynamic limit for a sequence of microscopic processes $ \vNtild$ defined on $\Z^N$. The letter $v$ will always denote a generic process. 

The idea of a hydrodynamic limit is that, if we zoom out from a microscopic process, a macroscopic dynamics will emerge. How to zoom out depends on several characteristic scales of the microscopic process. For stochastic lattice gases such as the crystal surface model, one of these scales is spatial: the lattice width $N$. A second scale to consider is temporal: as $N$ increases, it takes longer for local changes to have a global effect on the stochastic lattice gas. We therefore associate to macroscopic times $\Delta t$ the microscopic time $N^\alpha\Delta t$, for some $\alpha>0$. A third characteristic scale, relevant to unbounded microscopic processes such as the crystal surface, is the amplitude of the process itself. For example, consider the slope process $\zNtild$ associated to a crystal surface height process. The ``smoothness" of the crystal can be measured by the magnitude $N^\beta$ of these slopes. Rougher crystals have larger $\beta$. The transition rates will apply a greater force (so to speak) to smoothen rougher surfaces. 

To obtain a hydrodynamic limit, we must first rescale the process $ \vN$. Let $$\vN(t) = \vNtild(N^\alpha t).$$ The spatial rescaling occurs by interpreting $\vN(t)$ as a random step function on the unit interval, or more formally, a random measure:
\beq\label{empirical}v_N(t,dx) = \frac1N\sum_{i=1}^N v_{i}(t)\delta\l(x-\frac iN\r).\eeq Finally, we encode the amplitude scaling into the $ \vN$ initial condition, and into the definition of hydrodynamic limit. 
\begin{definition}\label{def:init}
We say the sequence of measures $\muN_0$ on $\Z^N$ is associated with a profile $\v_0(x)$ under amplitude scaling $N^\beta$ if
\beq\label{v-init}\muN_0\l((v_1,\dots, v_N)\; : \; \l|\frac1N\sum_{i=1}^N\phi\l(\frac iN\r)\l[N^{-\beta}v_i\r] -\int_\unit\phi(x)\v_0(x)dx\r|>\delta\r)\to0\eeq as $N\to\infty$ for all $\phi\in \cts$ and $\delta>0$. 
\end{definition}
As an example, the product measure $\muN_0$ with marginals $$v_i\;\sim\;\lfloor N^\beta v_0(i/N)\rfloor + \mathrm{Bernoulli}(1-\{N^\beta v_0(i/N)\})$$ satisfies~\eqref{v-init}. Here, $\{q\}=q-\lfloor q\rfloor $ denotes the fractional part of $q$.
\begin{definition}[Hydrodynamic Limit]\label{def:hydro} Let $ \vNtild$ be a sequence of random processes on $\Z^N$. 
We say the sequence $\vN(t)= \vNtild(N^\alpha t)$ has a \emph{hydrodynamic limit} $\v:[0,T]\times\unit\to\R$ under amplitude scaling $N^{\beta}$ if for each $t\in[0,T]$, $\phi\in \cts$, and $\delta>0$ we have
\beq\label{hydro-lim}\mathbb P\l(\l|\frac1N\sum_{i=1}^N\phi\l(\frac iN\r)\l[N^{-\beta}v_i(t)\r] -\int_\unit\phi(u)\v(t,u)du\r|>\delta\r)\to0,\quad N\to\infty.\eeq \end{definition} In~\eqref{hydro-lim}, the probability distribution on $\vN(t)$ is given by
$$\mathbb P(\vN(t) \in A) = \mathbb P( \vNtild(N^\alpha t) \in A) = \muN_{t}(A),$$ where $\muN_t := \muN_0\e(N^\alpha t\LN).$
Given $\alpha,\beta$, and a sequence of processes $\vNtild(t)$ with generators $\LN$, the main goals are (1) to prove that $\vN(t)=\vNtild(N^\alpha t)$ has a hydrodynamic limit $\v$ under amplitude scaling $N^\beta$, and (2) to determine the PDE governing the evolution of $\v$. Typically, these goals are intertwined and achieved together in rigorous proofs. Of course, we only expect there to be a nontrivial PDE limit if $\vNtild$ is initialized appropriately, i.e. if $\vNtild(0)\sim\muN_0$ is associated to a smooth $\v_0$ under amplitude scaling $N^\beta$. The function $\v_0$ is then the initial condition of the PDE.

Let us now return to the crystal surface models. In~\cite{krug1995adatom} and~\cite{mw-krug}, the authors study the Arrhenius rate discrete height process $\hNtild$ in the \emph{smooth} scaling regime (the name was coined in the latter paper). In this regime, one takes the hydrodynamic limit of $\hN(t) = \hNtild(N^4t)$ under amplitude scaling $N^1$. Note that the amplitude scaling has a nonlinear effect on the dynamics, i.e. $(\LN\pi_i)(Nh) \neq N(\LN\pi_i)(h)$ since the rates are exponential in $w_i = h_{i+1}-2h_i + h_{i-1}$. These papers show using nonrigorous arguments that in the smooth scaling regime, the hydrodynamic limit $\h$ solves 
\beq\label{smooth-h-PDE}\partial_t\h = -\partial_{xxx}\lambda_D(\partial_x\h).\eeq Here, $\lambda_D(u)=u+\lambda_o(u)$, where $\lambda_o$ is a smooth periodic perturbation with period one (the function $\lambda_D$ is denoted $\sigma_D$ in~\cite{mw-krug}). We will discuss this function and the smooth scaling PDE in Section~\ref{sec:theory}. Marzuola and Weare also propose a non-standard, ``rough" scaling regime in~\cite{mw-krug}, which is our interest here. In the rough scaling regime, one takes the hydrodynamic limit of $\hN(t) = \hNtild(N^4t)$ under amplitude scaling $N^2$. Through a heuristic argument supported by numerical simulations~\cite{mw-krug}, the authors conclude that the rough scaling limit $h$ of $\hN$ solves
\beq\label{PDE}
\begin{cases}
\h_t = \partial_{xx}\,e^{-2K\h_{xx}},\quad &t>0,\\
\h(0,x) =\h_0(x),\quad &x\in\unit.
\end{cases}
\eeq Existence, uniqueness, and regularity of solutions to this PDE and variations of it have been studied e.g. in~\cite{liu_arrPDE, gao2019_arrPDE, liu2018global, granero2018global, ambrose2019radius}. Liu and Xu~\cite{liu_arrPDE}, and Gao, Liu, and Lu~\cite{gao2019_arrPDE} show that singularities can form in $h_{xx}(t,\cdot)$ even when $h_0$ is smooth. However, Ambrose proves that if $\h_0$ is analytic, and if $\|\h_0''\|_{\mathbb A}$ is sufficiently small, then there exists a strong, analytic solution $h(t,x)$ to~\eqref{PDE} with growing radius of analyticity~\cite{ambrose2019radius}. Here, $\|u\|_{\mathbb A} = \sum_{k\in\Z}|\hat u(k)|$, where $\hat u(k)$ is the $k$th Fourier coefficient of $\hat u$. 

In this paper, we study the hydrodynamic limit of $\wN(t)=\wNtild(N^4t)$ under amplitude scaling $N^0$. As discussed, this process can be defined independently of $\hNtild(t)$. This scaling corresponds to the rough scaling regime for $\hNtild$. 
\begin{remark}\label{remark:hz-from-w}Note that if $\hN(t) = \hNtild(N^4t)$ has a hydrodynamic limit under amplitude scaling $N^2$, this does not imply that $\wN(t)$ has a hydrodynamic limit under amplitude scaling $N^0=1$. Meanwhile, the converse is likely true under some technical assumptions. This is due to the fact that if $\phi\in C(\unit)$ then $\phi''$ need not belong to $C(\unit)$ (or even exist), while there is always a periodic $\Phi\in C(\unit)$ such that $\Phi''=\phi$. In general, we expect the hydrodynamic limit of an antiderivative to follow from the hydrodynamic limit of the derivative, not vice versa.\end{remark}
The above remark shows we cannot rigorously deduce the limit of $\wN$ from the limit of $\hN$. Nevertheless, based on our knowledge of the hydrodynamic limit of $\hN$, we anticipate that the hydrodynamic limit of $\wN(t) =  \wNtild(N^4t)$ is the weak solution to
\beq\label{PDE-w}
\begin{cases}
\w_t = \partial_{xxxx}\,e^{-2K\w},\quad &t>0,\\
\w(0,x) =\w_0(x),\quad &x\in \unit
\end{cases}
\eeq 
with periodic boundary. For future reference, we define
\beq \hat r(w) = e^{-2Kw},\eeq the nonlinear function in the PDE~\eqref{PDE-w}. 

In our numerical simulations, we will always take $w_0 = h_0''$ for some analytic $h_0$ with $\|w_0\|_{\mathbb A}$ small, so that we can use the analyticity result of~\cite{ambrose2019radius} to conclude that there is an analytic, strong solution $h(t,x)$ to~\eqref{PDE} and hence $\partial_{xx}h(t,x)$ is an analytic strong solution to~\eqref{PDE-w}. Although other non-analytic weak solutions to~\eqref{PDE-w} could exist, we find numerically that $\wN(t)$ always converges to a smooth function $w(t,\cdot)$. 


The reason to study $\wN$  --- specifically, the \emph{probability distribution} of $\wN(t)$ once ``local equilibrium" has been reached ---  is that doing so will uncover very interesting new phenomena overlooked in a study of $\hN$ alone. The most striking phenomenon is the roughness of the $(\E w_i)_{i=1}^N$ profile mentioned in the introduction. Note that this roughness in the microscopic system is \emph{not} related to any potential singularity formation in solutions to the PDE~\eqref{PDE-w}, since we take initial data to ensure the PDE has a strong analytic solution.


\section{Three Key Limits within Hydrodynamic Convergence}\label{sec:reduc}
In this section, we introduce the three key limits we focus on in this paper, which expose important properties of the process. The first two limits express that $\wN$ converges to a macroscopic $w$. The third limit helps determine the PDE governing $\w$. To motivate this third limit, we will sketch the PDE derivation assuming the first two limits hold. It is important to emphasize that this PDE derivation is informal, and primarily serves as inspiration to study the third limit in addition to the first two. 

Before proceeding, we introduce some notation. For $\phi\in L^\infty(\unit)$ and $w_N(t, dx)$ given by the random measure intepretation of $\wN(t)$ (recall~\eqref{empirical}), we define
\beq\label{conv}
(\phi\ast w_N(t))(x) = \int_\unit \phi(x-y)w_N(t,dy).
\eeq Also, for a vector $\mbf v=(v_1,\dots, v_N)$ and a function $f:\R\to\R$, define
\beqsn
\ol{\bf v}_{\idxsetx}&=\frac{1}{2N\epsilon}\sum_{\argidxsetx i} v_i, \\
\bar {f}({\bf v}_{\idxsetx}) &= \frac{1}{2N\epsilon}\sum_{\argidxsetx i}f(v_i).\eeqsn

Now, instead of the usual convergence~\eqref{hydro-lim} given in the hydrodynamic limit definition, we will consider a slightly more physically intuitive convergence. Namely, we will study whether the following holds:
\beq\label{convergence}\limNeps\E\bigg|\frac{1}{2N\epsilon}\sum_{i\in\idxsetx} w_i(t) - \w(t,x)\bigg|^2=0,\quad\forall t>0,x\in\unit.\eeq 
~\eqref{convergence} is somewhat stronger than~\eqref{hydro-lim}, but not significantly. Suppose~\eqref{hydro-lim} holds, and consider taking $\phi$ to be $\phi(u)=\hat\phi_\epsilon(x-u)$, where $\hat\phi_\epsilon$ is an even bump function centered at $0$ with support in $(-\epsilon,\epsilon)$. By~\eqref{hydro-lim}, we have that $( \wN(t)\ast\hat\phi_\epsilon)(x)$ converges as $N\to\infty$ to $(\w(t)\ast\hat\phi_\epsilon)(x)$ in probability for each $x\in\unit$. If $\w$ is continuous in $x$, then it is not hard to see that $( \wN(t)\ast\hat\phi_\epsilon)(x)$ converges as $N\to\infty$ and then $\epsilon\to0$ to $\w(t,x)$ (this means the probability that the two differ by more than $\delta$ goes to zero in the double limit). In~\eqref{convergence}, we ask for a slightly stronger convergence than this, replacing the smooth function $\hat\phi_\epsilon(u)$ by the indicator $\phi_\epsilon(u)=\frac{1}{2\epsilon}\mathbbm{1}_{(-\epsilon,\epsilon)}(u)$ and convergence in probability by convergence in $L^2$. 

The advantage of studying the convergence~\eqref{convergence} is two-fold. First, it allows us to focus on properties of the process in local, \emph{mesoscopic} regions of the lattice. ``Mesoscopic" regions are intervals $(x-\epsilon,x+\epsilon)$ which have macroscopic length $2\epsilon\ll1$, but which, for any fixed $\epsilon$, contain a growing number of lattice sites $2N\epsilon$, $N\to\infty$. Second, convergence in $L^2$ can be conveniently separated into convergence of expectations and vanishing variance. Namely, the convergence~\eqref{convergence} is equivalent to 
\begin{subequations}\label{one}
\begin{align}
&\limNeps\mathrm{Var}\l(\ol{\bf w}_{\idxsetx}(t) \r)= 0,\label{one-Var}\\
&\limNeps\E\ol{\bf w}_{\idxsetx}(t) = \w(t,x).\label{one-Exp}
\end{align}
\end{subequations}
These are our first two limits of interest. The equation ~\eqref{one-Var} is satisfied if e.g. the variances $\Var(w_i$) remain bounded while the correlations $\Corr( w_i,  w_j)$ decrease with $N$, so that by averaging increasingly many $ w_i$ ($2N\epsilon$ of them, with $N\to\infty$), the variance of the average goes to zero. The equation~\eqref{one-Exp} expresses that the mesoscopic spatial average of expectations $\E[ w_i]$ converges to the limiting profile $\w(t,x)$. 

Now, assume there exists a limit $\w(t,x)$ such that~\eqref{one-Exp} holds. We will show that~\eqref{one-Exp}, combined with one more limit, is ``nearly" sufficient to deduce that $\w$ is a weak solution to~\eqref{PDE-w}. We define ``weak solution" to mean that for all smooth compactly supported functions $\psi$ and for all $t>0$ we have
\beq\label{weak}\int_\unit\psi(x)[\w(t,x)dx -\w_0(x)]dx = \int_0^t\int_\unit\psi^{(4)}(x)\hat r(\w(s,x))dsdx.\eeq The term ``nearly" is a catch-all for the non-rigorous statements made below (again, the following derivation is only for the purpose of motivating the third limit). Using~\eqref{one-Exp} and the fact that $\ol{\bf w}_{\idxsetx}(t)=(\phi_\epsilon\ast w_N(t))(x)$, we have
\beqs\int_\unit\psi(x)\w(t,x)dx &= \limNeps \int_\unit\psi(x)\E\l[(\phi_\epsilon\ast w_N(t))(x)\r]dx \\&= \limNeps \int_\unit(\psi\ast\phi_\epsilon)(x)\E[ w_N(t, dx)],\eeqs so that
\beqs\label{intermediate}
\int_\unit\psi(x)[\w(t,x)dx &-\w_0(x)]dx \\&= \limNeps \int_\unit(\psi\ast\phi_\epsilon)(x)\E[ w_N(t, dx)- w_N(0,dx)]\\
&= \frac1N\sum_{i=1}^N(\psi\ast\phi_\epsilon)(i/N)\E[ w_i(t)- w_i(0)].\eeqs
Now, by definition of the generator $\LN$ and the fact that $\tilde w_i = \pi_i(\wNtild)$, we have
\beqs
\E[ w_i(t)-\E  w_i(0)] = \E\l[\tilde w_i(N^4t) - \tilde w_i(0)\r] 
= \int_0^tN^4\E[(\LN \pi_i)(\wN(s))]ds.\eeqs 
We replaced $\wNtild(N^4s)$ with $\wN(s)$ in the last integral. Recall from~\eqref{L-pi-w} that $\LN\pi_i$ is a fourth order finite difference of the rates. Therefore, we can write~\eqref{intermediate} as
\beqs\label{pre-PDE}
\int_\unit\psi(x)[\w(t,&x)dx -\w_0(x)]dx \\= \lim_{\epsilon\to0}&\lim_{N\to\infty} \frac1N\sum_{i=1}^N(\psi\ast\phi_\epsilon)(i/N)\times\\&\int_0^tN^4\E\l[(r_{i-2}-4r_{i-1}+6r_i -4r_{i+1}+r_{i+2})( \wN(s))\r]ds\\
= \lim_{\epsilon\to0}&\lim_{N\to\infty}\int_0^t\frac1N\sum_{i=1}^N(\psi^{(4)}\ast\phi_\epsilon)(i/N)\E\l[r_i(\wN(s))\r]ds\\
=\int_0^t&\int_\unit\psi^{(4)}(x)\limNeps \E\l[\bar r({\bf w}_{\idxsetx}(s))\r] dsdx
\eeqs
To get the last line, define $\eta_s=\sum_{i=1}^N\E[r_i(\wN(s))]\delta(x-i/N)$, and note that the integral in the second to last line can be written as $\int_0^t\int_0^1 (\psi^{(4)}\ast\phi_\epsilon)(x)\eta_s(dx)ds$. Now, move $\phi_\epsilon$ onto $\eta_s$ and note that $(\phi_\epsilon\ast\eta_s)(x) = \E\l[\bar r({\bf w}_{\idxsetx}(s))\r].$ 

Comparing~\eqref{pre-PDE} with the weak formulation~\eqref{weak} of the PDE, we see that it remains to show that
\beq\label{two}\limNeps \l|\E\bar r({\bf w}_{\idxsetx}(s))- \hat r\l(\w(s,x)\r)\r|=0\eeq for all $s,x$. This limit should hold uniformly in $s$ and $x$, but we are primarily interested in local properties of $\wN$, so we will not study whether uniform convergence over time and space holds. The limit~\eqref{two} is the key reason a PDE emerges. It allows us to close the macroscopic equation by replacing a mesoscopic average of nonlinear observables of the microscopic $ w_i$ with a function of the limiting value $\w$ in that mesoscopic neighborhood. The fact that the PDE derivation reduces to~\eqref{two} crucially depended on $\E{\bf w}_{\idxsetx}(s)$ converging to $\w(t,x)$. As such, it is more appropriate to rewrite~\eqref{two} as follows, which is equivalent to~\eqref{two} thanks to the continuity of $\hat r(w)=e^{-2Kw}$:
\beq\label{two-true}\limNeps \bigg|\E\bar r({\bf w}_{\idxsetx}(s))- \hat r\l(\E{\bf w}_{\idxsetx}(s)\r)\bigg|=0\eeq This is the third limit of interest. 
\begin{remark} A feature of the $\wN$ dynamics which vastly simplifies the above arguments is the fact that the generator applied to $\pi_i$ involves four finite differences, which can all be moved onto the smooth test function to cancel the $N^4$ factor coming from the time scaling. In this respect, the $\wN$ process is similar to IPS satisfying the so-called gradient condition. This term is primarily used in the context of IPS leading to second order diffusions. However, the simplifying feature is the same as in our case: the generator applied to $\pi_i$ has exactly $\alpha$ finite differences, where $N^\alpha$ is the characteristic time scale. See~\cite{varadhan1996nongradient} for a review of this concept, and e.g.~\cite{wick1989hydrodynamic, kipnis1994hydrodynamical, varadhanII} for examples of nongradient convergence proofs.  \end{remark}

The limit~\eqref{two-true} is similar to the Replacement Lemma (RL), a key intermediate result in many rigorous hydrodynamic limit proofs. See for example Theorem 3.1 in~\cite{hardcore}, Lemma 1.10 in Chapter 5 of~\cite{kipnisbook}, and Lemma 3.5 and preceding arguments in Part II of~\cite{spohnbook}. The RL differs somewhat from~\eqref{two-true} in several respects. First, the RL is a statement about convergence in $L^1$ rather than convergence in expectation. The second more significant difference is that in the RL, the measure with respect to which $L^1$ convergence is shown is a spatially and temporally, \emph{globally averaged} measure. For the hydrodynamic limit proof it is sufficient to study this globally averaged measure, which has some very convenient properties.

However, one can obtain much more insight about the closure of the macroscopic equation by studying local properties of a process $\vN$; specifically, the marginal distributions $\Law(\{ v_{i}(s)\}_{\argidxsetx i}).$ In the next section, we will review what is known about the structure of these distributions for standard IPS. The structure will shed light on the three limits~\eqref{one-Exp},~\eqref{one-Var}, and~\eqref{two}, which we will now rewrite slightly for a general process $ \vN$. Both here and later on in the text, dependence on $t$ is implied whenever it is not explicitly indicated.
\begin{align*}
\label{V}&\lim_{\epsilon\to0}\lim_{N\to\infty}\mathrm{Var}\l(\ol{\bf v}_{\idxsetx }\r)= 0\quad\forall x\in\unit, t>0\tag{V}\\ \\
\label{E}&\text{There exists a continuous $v:\unit\to\R$ such that }\tag{E}\\
&\lim_{\epsilon\to0}\lim_{N\to\infty}\E\ol{\bf v}_{\idxsetx }= \v(t,x), \quad \forall x\in\unit, t>0\\ \\
&\label{Ef}\text{For all ``suitable" }f,\text{ there exists }\hat f\text{ such that  }\tag{Ef}\\ &\E\bar {f}({\bf v}_{\idxsetx })\stackrel{N,\epsilon}{\approx}\hat f\l(\E\ol{\bf v}_{\idxsetx }\r), \quad \forall x\in\unit, t>0.
\end{align*}
~\eqref{Ef} should be understood to mean that the difference between the two sides of the approximate equality goes to zero in the $N\to\infty$, $\epsilon\to0$ double limit. (We find the approximate equality notation simpler to interpret; this is purely an aesthetic decision). The word ``suitable" is a stand-in for a certain class of functions to be identified later. For this class of functions $f$, the corresponding $\hat f$ will be guaranteed to be continuous, so that given~\eqref{E}, the limit~\eqref{Ef} is equivalent to the limit $\E\bar {f}({\bf v}_{\idxsetx })\to\hat f(\v(t,x))$. 

\section{Local Equilibrium}\label{sec:smooth-LE} This section will set the stage for our presentation of the fascinating and anomalous properties of $\wN$'s local equilibrium, by drawing a comparison to the LE state of more standard processes $\vN(t)\in\Z^N$ for which hydrodynamic limits have been shown to exist.  We will review what is known in the literature about these processes' LE states, and supplement the review with a few illuminating numerical simulations. The numerical simulations give stronger results than does the existing theory, showing that in addition to~\eqref{V},  stronger versions of~\eqref{E} and~\eqref{Ef} are satisfied. But first, what \emph{is} an LE state? 
\subsection{Informal Definition and Prototypical Form of LE State}\label{subsec:LE-def} 
\begin{definition}[LE State, Informal]\label{def:LE}A \emph{\textbf{local equilibrium state}} is a sequence of random vectors $\vN$ (equivalently, measures $\muN=\Law(\vN)$) on $\Z^N$, $N=1,2,\dots$ with the following property: there is a continuous function $v:\unit\to\R$, such that for each $x$ the marginal distributions $\Law(\{v_i\}_{i\in\idxsetx})$ are asymptotically ``fully determined" (through some parameterization) by $v(x)$. 
A \emph{\textbf{global equilibrium}} state is a local equilibrium state in which $v(x)\equiv v_0$ is constant, so that the entire distribution of $(v_1,\dots, v_N)$ is parameterized by the single number $v_0$. \end{definition}

To connect these notions with time, typically, interacting particle systems $\vN(t)$ assume an LE state instantaneously on the macroscopic time scale (i.e. $\Law(\vN(t))$, $N=1,2,\dots$ is an LE state for each $t>t_N$, where $t_N$ to $0$ as $N\to\infty$), and the function $v=v(t,x)$ varies smoothly with $t$, i.e. slowly on the microscopic time scale. Meanwhile, the invariant, long-time distribution of the process is given by a global equilibrium state, parameterized by a single constant $v_0$. 


As we will see, equilibration of $\vN(t)$ to either a ``rough" or ``smooth" LE state (defined below), is the key reason we expect the function $\hat f$ of~\eqref{Ef} to exist. And as the previous section shows,~\eqref{Ef} (together with~\eqref{E}) is a crucial ingredient in the emergence of a PDE governing the macroscopic dynamics.

\subsection{Smooth Local Equilibrium}\label{smooth-LE-def} The prototypical LE state takes the form
\beq\label{local-eq}\muN=\muN_{v(\cdot)} = \bigotimes_{i=1}^N\mu\l[v\l(\frac iN\r)\r],\eeq where $\otimes$ denotes the measure product, and $\mu[\cdot]$ is a family of mean-parameterized measures on $\Z$. By mean-parameterized, we mean that $v=m_1(\mu[v])$. To see that this is an LE state, note that if $(v_1,\dots, v_N)\sim\muN_{v(\cdot)}$, then the $v_i$, $i\in\idxsetx$ are independent and approximately identically distributed, with $\Law(v_i)\approx \mu[v(x)]$ provided $N\gg1$, $\epsilon\ll1$. Thus, the parameter $v(x)$, via $\mu[v(x)]$, fully determines the joint distribution of $v_i$, $i\in\idxsetx$, as in our definition. 

Suppose $\vN\sim\muN_{v(\cdot)}$, so in particular $\Law(v_i)= \mu[v(i/N)]$. Since $\mu$ is mean-parameterized, we have
\beqs\label{smooth-fundamental}\E v_i &= v(i/N),\\
\E[f(v_i)] &=  \mu[\E v_i](f) = \hat f(\E v_i),\quad\text{where}\quad\hat f(v) := \mu[v](f).\eeqs Equation~\eqref{smooth-fundamental} expresses the powerful yet simple idea that for mean-parameterized families, the expectation of any observable $f$ is a fixed function $\hat f$ of the mean.  The above calculations motivate us to define smooth and rough LE states as follows:
\begin{definition}[Smooth, Rough Local Equilibrium]
We say a process $ \vN(t)$ has a \emph{\textbf{smooth}} LE state if for each $t>0$,~\eqref{V} holds, as well as 
\begin{align*}
\label{Ep}&\text{There exists a continuous $v:\unit\to\R$ such that}\tag{E$\,'$}\\
&\lim_{N\to\infty}\E[ v_{Nx+k}]= \v(t,x),\quad\forall \,x\in\unit,\, \forall \,k=0,1,2,\dots\text{fixed}\\ 
\label{Efp}&\text{For all ``suitable" }f,\text{ there exists }\hat f\text{ such that  }\tag{Ef$\,'$}\\ &\E f( v_{Nx+k})\stackrel{N}{\approx}\hat f\l(\E v_{Nx+k}\r), \quad \forall \,x\in\unit,\,\forall\,k=0,1,2,\dots\text{ fixed.}
\end{align*}
We say $ \vN$ has a \emph{\textbf{rough}} LE state if~\eqref{V},~\eqref{E}, and~\eqref{Ef} are satisfied, but \eqref{Ep} and \eqref{Efp} are violated. 
\end{definition}
Like~\eqref{Ef}, we interpret \eqref{Efp} to mean the difference between the left- and righthand sides goes to zero with $N$. Note that \eqref{Ep} and \eqref{Efp} are pointwise versions of~\eqref{E} and~\eqref{Ef}, respectively. To describe one of the essential differences between rough and smooth LE states, we introduce the following terminology.
\begin{definition}\label{smth-rgh}
Let ${\bf v}_N = (v_1,\dots, v_N)$, $N=1,2,\dots$ be a sequence of vectors. We say ${\bf v}_N$ is \emph{\textbf{smoothly}} varying in a neighborhood of $x$ if for any finite $R\geq 0$ we have
$$\lim_{N\to\infty}\max_{|i-Nx|\leq R}|v_{i+1}-v_i| = 0.$$ Otherwise, ${\bf v}_N $ is \emph{\textbf{roughly}} varying. 
\end{definition}
It is straightforward to see that whenever \eqref{Ep} holds, the expectations $\E[ \vN(t)]$ are smoothly varying near every $x$. On the other hand,~\eqref{E} does not guarantee this. As we will see in the next section, $\E[ \wN(t)]$ (the expectation of the crystal curvature process in the rough scaling regime) is roughly varying in a very interesting way. Nevertheless, for rough LE states $\vN$,~\eqref{E} is still satisfied for a continuous function $v(t,x)$. Thus, the roughness of $\E[v_i]$ must smooth out upon mesoscopic averaging.

Our prototypical example of a smooth LE state is the measure~\eqref{local-eq}. For such a measure, let us specify the ``suitable" class of functions $f$. First, assume the family $\mu[\cdot]$ is continuous and majorized by some pmf $p$:
\begin{definition}\label{family}
We say the family $\{\mu[v]\mid v\in\R\}$ is \emph{\textbf{continuous}} if $\lambda\mapsto\mu[\lambda](n)$ is continuous for each $n\in\Z$. We say the pmf $p$ \emph{\textbf{majorizes}} the family if $p$ has a finite first moment and
\beq\label{majorize} \forall M>0\quad\exists C(M)>0\text{ s.t. }\sup_{|v|\leq M}\mu[v](n)\leq C(M)p(n),\quad\forall n\in\Z.\eeq
\end{definition}
For a continuous measure family majorized by $p$, the suitable class of functions $f$ are the $f\in L^1(p)$. Indeed, if  $f\in L^1(p)$ then the expectations $\hat f(v):=\mu[v](f)$ are finite near any fixed $v_0$. Moreover, by the Lebesgue Dominated Convergence Theorem, $\hat f$ is continuous at any fixed $v_0$. 
 
Of course, it is unrealistic to expect $\Law(\vN(t))=\muN_{\v(t,\cdot)}$ exactly in hydrodynamic convergence arguments. But if $\Law(\vN)$ is sufficiently close in some sense to $\muN_{\v(t,\cdot)}$ as $N\to\infty$, then we might still expect \eqref{Ep} and \eqref{Efp} to hold. Spohn postulates this in Conjecture 3.2 of Part II of~\cite{spohnbook}. However, rigorous results on the proximity of $\Law(\vN(t))$ to the local equilibrium state $\muN_{\v(t,\cdot)}$ of~\eqref{local-eq} are insufficient to show that \eqref{Ep} and \eqref{Efp} hold. We review what is known theoretically about a few standard processes and show numerically that \eqref{Ep} and \eqref{Efp} are in fact satisfied.
\begin{remark}
Note that~\eqref{Ef} and \eqref{Efp} only consider single site observables $f$, i.e. $f:\R\to\R$. We expect analogous statements can be formulated for general cylinder functions $f:\R^{N}\to\R$. For example, for the general statement of \eqref{Efp}, see the aforementioned Conjecture 3.2 in~\cite{spohnbook}. We choose to focus on single site observables for simplicity and because the key observable of interest for the $\wN$ process is of this form, namely, $f(\wN)=r_i(\wN)=r(w_i)$. 
\end{remark}

\subsection{Example of Smooth LE State: Zero Range Process} Let $S=\mathbb N\cup \{0\}$, and $\vNtild(t)\in S^N$ be a jump process in which $\tilde v_i(t)$ represents the number of particles at lattice site $i$ at time $t$. The dynamics of the zero range process is determined by a pmf $p$ supported on $\{-R,\dots, R\}$ and a rate function $g:\{0,1,\dots\}\to [0,\infty)$. At each lattice site $i$, a Poisson clock with rate $g(v_i)$ determines when one of the $v_i$ particles will jump away from site $i$ (with $g(0)=0$, to preclude jumps from a lattice site containing no particles). The pmf $p$ determines the probability that the particle will jump to site $i+k$, $k=-R,\dots, R$. The overall transition rate $\vN\mapsto \vN^{i,i+k}$ is therefore $r_{i,i+k}(\vN) = g(v_i)p(k)$. The process is named ``zero range" due to the fact that the jump rate from site $j$ only depends on the number of particles at site $j$ itself. 

Kipnis and Landim prove that under certain conditions on $g$, the zero range process has a hydrodynamic limit $v(t,x)$ under diffusive time scaling, i.e. $\vN(t) = \vNtild(N^2t)$, with no amplitude scaling~\cite{kipnisbook}. Their proof uses the relative entropy method introduced in~\cite{Yau1991}, in which one obtains an estimate of the relative entropy between $\Law(\vN(t))$ and an appropriate local equilibrium state $\muN$. For the zero range process, the LE state of $\vN$ takes the prototypical form~\eqref{local-eq}. The mean-parameterized measure family $\{\mu[v]\mid v\in\R_+\}$ for this process is described in Appendix~\ref{app:ZR-LE}. Specifically,~\cite{kipnisbook} shows that \beq\label{rel-ent}H(\Law(\vN(t))\mid \muN_{v(t,\cdot)}) = o(N),\quad N\to\infty,\eeq where $v(t,x)$ is the hydrodynamic limit and $H$ is the relative entropy, 
$$H(\mu\;\big\vert\;\nu) = \sum_{\mbf v\in\mathrm{supp}(\nu)}\mu({\bf v})\log(\mu({\bf v})/\nu({\bf v})).$$ However,~\eqref{rel-ent} does not imply $H\l(\Law(v_{Nx+R})\;\big\vert\; \mu\l[v\l(x\r)\r]\r)=o(1)$ for any fixed $x$, so~\eqref{rel-ent} is insufficient to prove \eqref{Ep} and \eqref{Efp}. 
\begin{remark}
An example of a result one can show using~\eqref{rel-ent}, which is somewhat similar to~\eqref{Ef}, is
$$\lim_{\ell\to\infty}\lim_{N\to\infty}\frac1N\sum_{i=1}^N\bigg|\E\bigg[\frac{1}{2\ell}\sum_{|j-i|\leq\ell}f(v_i(t))\bigg] -\hat f(v(t,x))\bigg|=0.$$
This follows from the proof of Corollary 1.3 in Chapter 6 of~\cite{kipnisbook}.\end{remark}
Using numerical simulation, one can get much stronger results. Indeed, we show that the symmetric nearest neighbor zero range process has a smooth LE state, i.e. it satisfies~\eqref{V}, \eqref{Ep}, and \eqref{Efp}. For the rate function $g$ of the zero range process, we take $g(k) = k + k^{1/4}$. The reason for this choice of $g$ is that it satisfies the condition of linear growth, assumed in the proof of~\eqref{rel-ent} given in~\cite{kipnisbook}. Moreover, we choose $g$ to be nonlinear, in order to make the existence of $\hat g$ nontrivial (when verifying \eqref{Efp} with $f=g$). For our initial macroscopic condition we take $\v_0(x) = 5\sin(2\pi x)^2$. See Appendix~\ref{app:KMC} for the details of the numerical simulations. The left panel of Figure~\ref{fig:Eprime} depicts the $i/N\mapsto \E[ v_i(t)]$ profile, and the right panel depicts $\max_i\Var(\ol{\bf v}_{i\pm N\epsilon})$, for $N=500,1000,2000$.
\begin{figure}
\centering
\vspace{-0.6cm}
\includegraphics[width=0.7\textwidth]{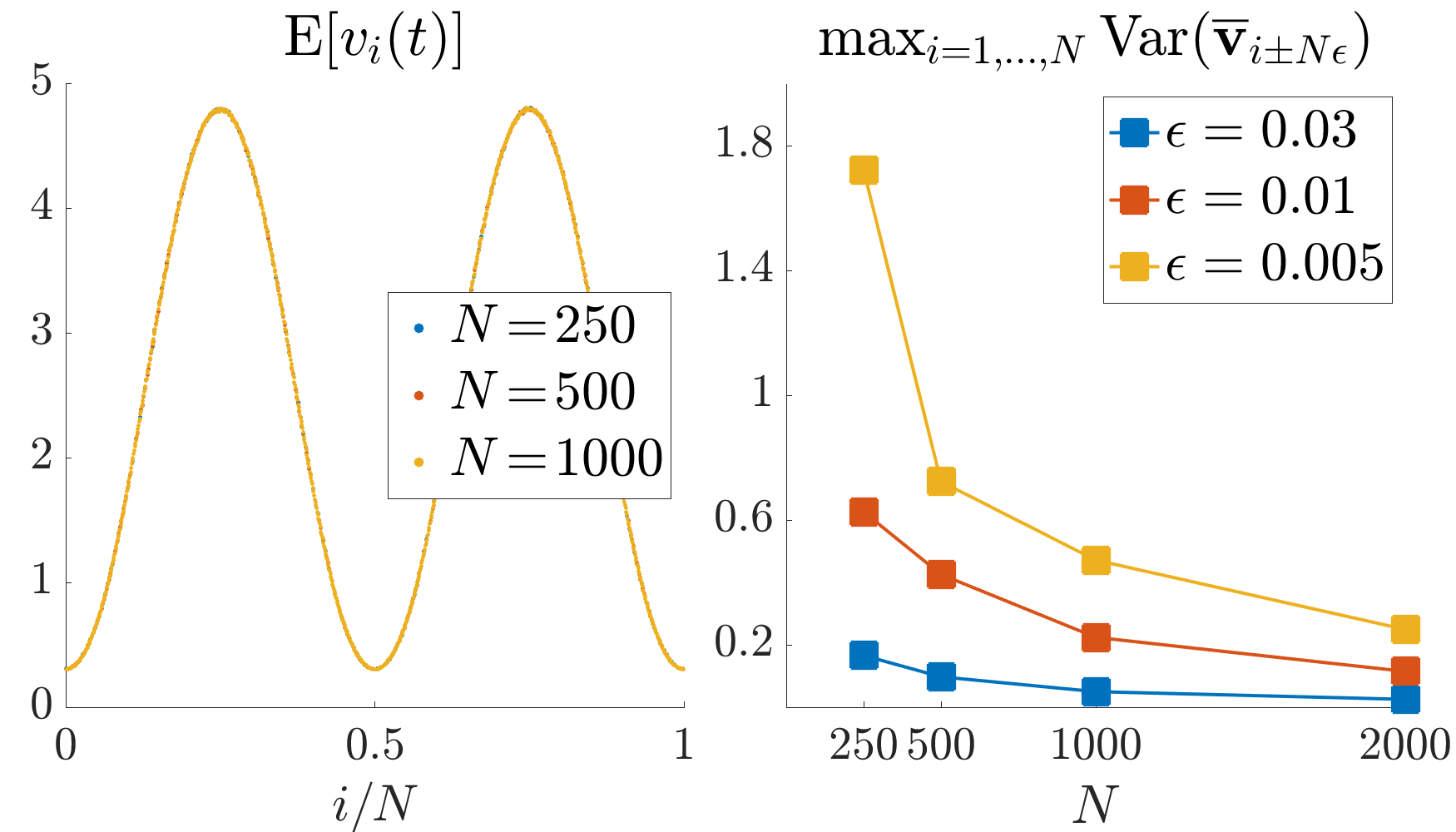}
\caption{Verification of \eqref{Ep} and~\eqref{V}. For each $N$, the expectations $\E[ v_i(t)]$ and variances $\Var(\ol{\bf v}_{i\pm N\epsilon})$ are estimated from an average of $2\times 10^5$ samples.} 
\label{fig:Eprime}
\end{figure} 
It is clear from the left panel that there is a smooth function $v(t,x)$ such that $\E[ v_{Nx+R}(t)]$ converges to $v(t,x)$ as $N\to\infty$ for any fixed $R$. This confirms \eqref{Ep}. The right panel also clearly confirms~\eqref{V}. Next, we verify \eqref{Efp} for two functions $f$: $f(v)=g(v)-v = v^{1/4}$, where $g$ is the rate function, and $f(v)=ve^{-v}$. To check that this limit holds, we plot pairs $(\E  v_i, \E f( v_i))$ to see if they line up on the curve $(v,\hat f(v))$ as $N$ increases. This is shown in Figure~\ref{fig:Efprime}. The function $\hat f$ is computed explicitly using our explicit knowledge of $\mu[\cdot]$. See Appendix~\ref{app:ZR-LE} for the details. 
\begin{figure}
\centering
\vspace{-0.6cm}
\includegraphics[width=0.8\textwidth]{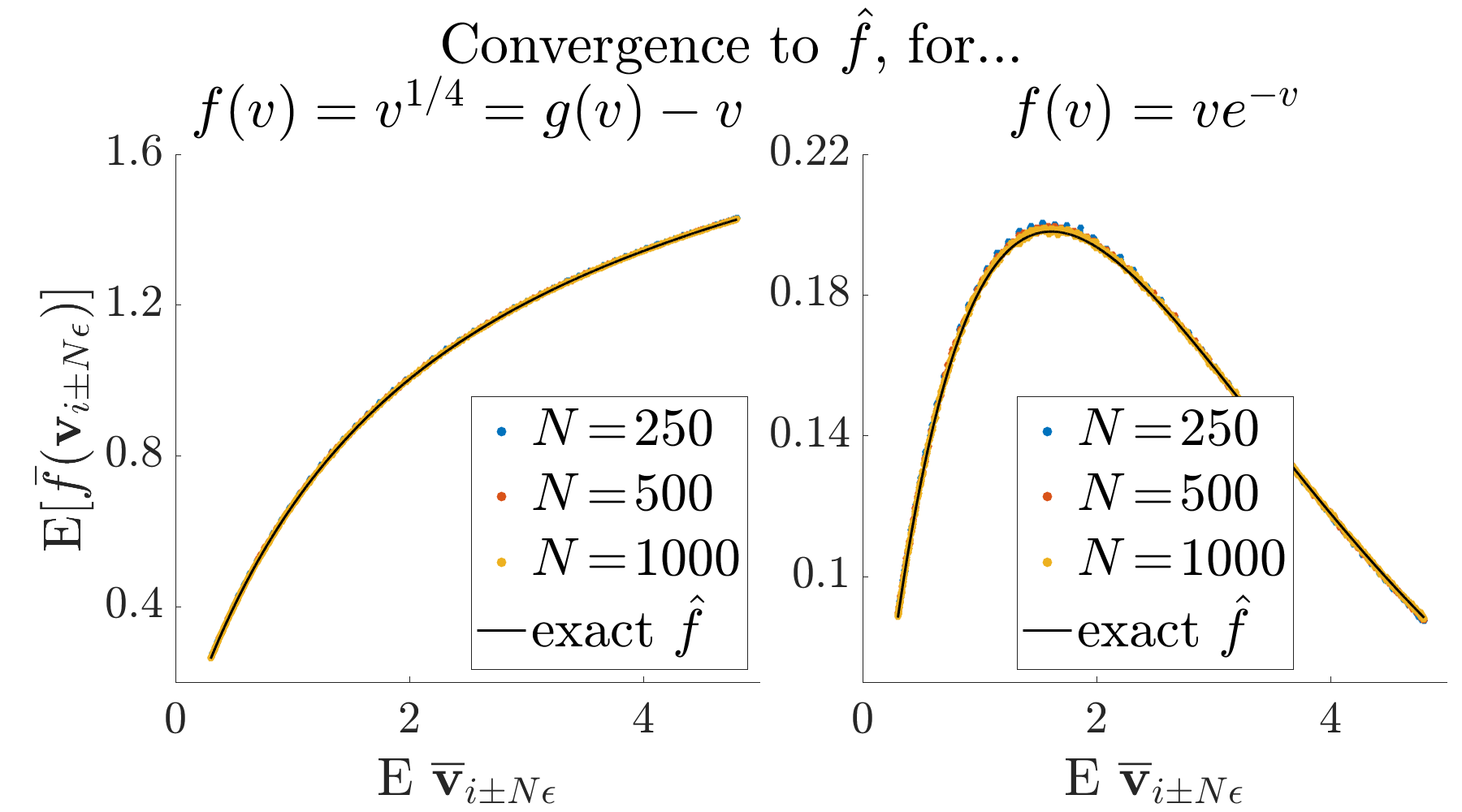}
\caption{Verification of \eqref{Efp} for two functions $f$.} 
\vspace{-0.5cm}
\label{fig:Efprime}
\end{figure} 

\subsection{Hard Core Exclusion; Crystal Slope Process}
We confirm the distinguishing feature of smooth LEs, that $\E[v_i(t)]$ varies smoothly across lattice sites $i$, for two more processes. First, we simulate a hard core exclusion process satisfying the conditions of~\cite{hardcore}. Under these conditions, the authors prove that $\Law(\vN(t))\approx\muN_{v(t,\cdot)}$, where $\muN_{v(t,\cdot)}$ takes the form~\eqref{local-eq} and $\mu[v]=\mathrm{Bernoulli}(v)$. See the paper for the precise statement which, as for the zero range process, is insufficient to show that the limits \eqref{Ep} or \eqref{Efp} hold. The plot of $\E v_i$, which we omit for brevity, clearly shows the $\E v_i$ are smoothly varying. 

Second, we consider the crystal slope process in the smooth scaling regime, i.e. we take $\zN(t) = \zNtild(N^4t)$, where $z_i = h_{i+1} - h_{i}$ for a height process initialized with amplitude scaling $N^1$. For this process, there are currently no rigorous results on the form of $\Law(\zN)$. However, as argued in~\cite{mw-krug}, $\zN$ is expected to have a smoothly varying local Gibbs LE state (for more on the local Gibbs distribution, see Section~\ref{subsec:arr-prop}). The left and right panels of Figure~\ref{fig:ASEP-crystal} show $\E [z_i(t)]$ and $N(\E[z_i(t)]-\E[z_{i-1}(t)])$, respectively, as functions of $i/N$. We take $N=1000$ and $t=10^{-8}$. As expected, we see that $\E[z_i(t)]$ look smoothly varying, and interestingly, so do the finite differences. The oscillations in both plots are smooth, and not an artifact. They are a consequence of the function $\lambda_D$ appearing in the PDE (for a plot of this function, see Figure~\ref{fig:u-lambda}).

\begin{figure}
\centering
\includegraphics[width=0.42\textwidth]{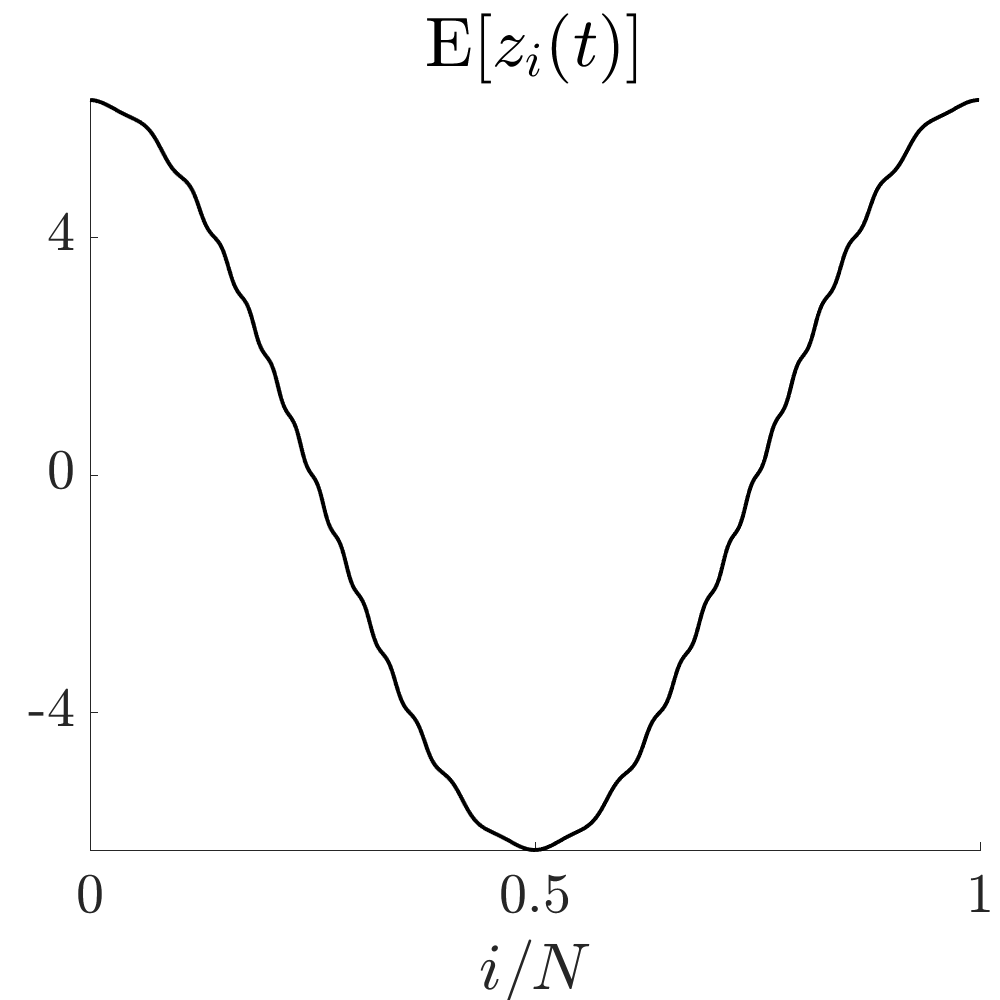}
\includegraphics[width=0.42\textwidth]{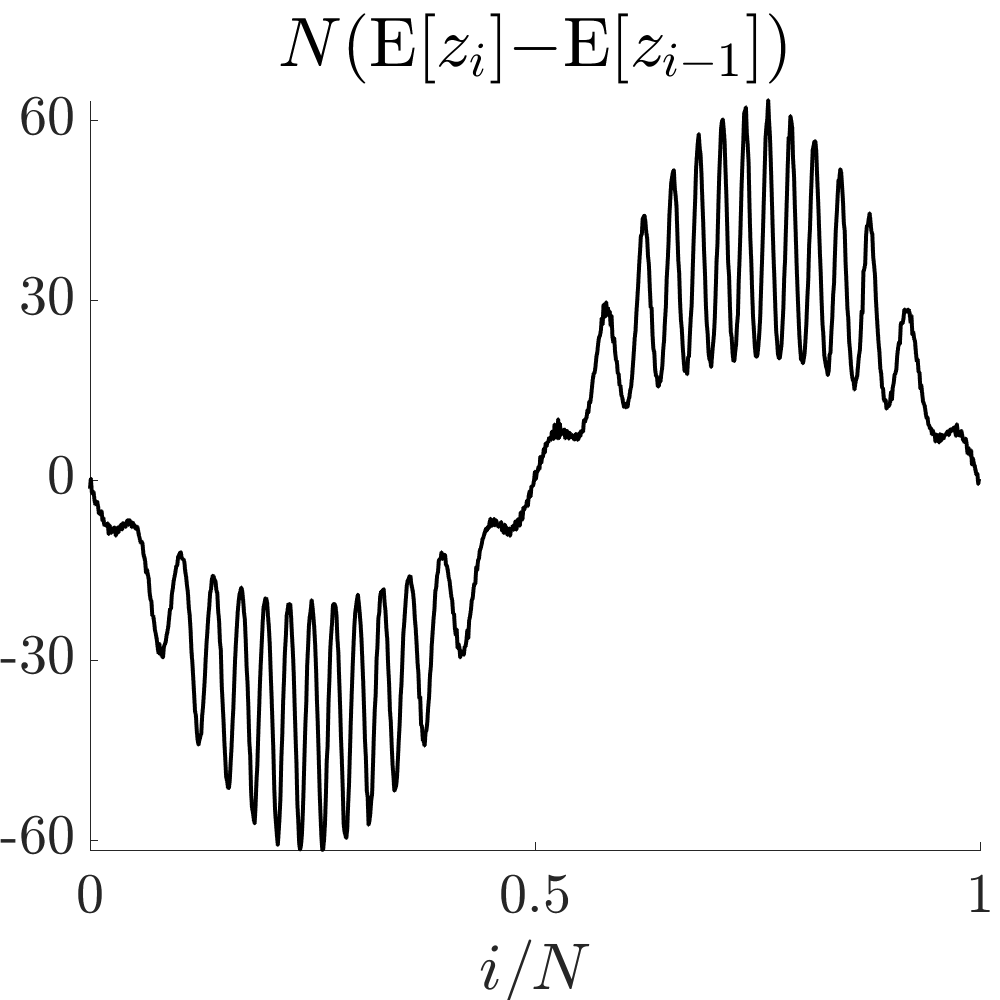}
\caption{Left: $\E[z_i(t)]$, for the time-rescaled crystal slope process in the smooth scaling regime. Right: $N(\E[z_i(t)]-\E[z_{i-1}(t)])$. Both seem to be smoothly varying. The oscillations are not an artifact; they are a result of the function $\lambda_D$ appearing in the PDE. $N=1000$ and $t=10^{-8}$. }
\label{fig:ASEP-crystal}
\end{figure}

\section{Rough Local Equilibrium: A Numerical Study}\label{sec:rough-LE} We now return to the crystal process $ \wN$. Our numerical study of the distribution of $ \wN(t)$, starting in Section~\ref{subsec:not-smooth}, will show that it is nothing like the prototypical LE state~\eqref{local-eq}. However, we show in Section~\ref{subsec:rough-LE} that it is nevertheless a \emph{rough} LE, satisfying~\eqref{E},~\eqref{Ef}, and~\eqref{V}. Section~\ref{subsec:qual} concludes with some qualitative observations. But first let us mention a few numerical considerations. 
\subsection{Numerical Considerations} Let $\E^n[f( \wN(t))]$ denote the average of \\$f(\wN^{(k)}(t))$, $k=1,\dots,n$, computed from $n$ independent runs of the process. We will also consider time averages
\beqs\label{time-av}\E_\Delta[f( \wN(t))]:=\frac1\Delta\int_{I_t}\E[f( \wN(s))]ds,\eeqs where $I_t$ is any length $\Delta$ interval containing $t$. Let $\E_\Delta^n[f(\wN(t))]$ denote sample average estimates of $\E_\Delta[f( \wN(t))]$. For details on how we compute these quantities, see Appendix~\ref{app:KMC}.


For precise confirmation of~\eqref{E},~\eqref{Ef}, and~\eqref{V}, --- which are statements about the instantaneous-time measure $\Law(\wN(t))$ --- we use the sample-only, instantaneous estimator $\E^n[f(\wN(t))]$. The only exception is in computing the spatially averaged rate expectation $\E[\bar r({\bf w}_\idxsetx)]$ needed to check~\eqref{Ef} for $f=r$. The rate observable has a very large variance (even upon spatial averaging) so we are forced to rely on the sample-and-time average~\eqref{time-av}, which has significantly lower variance. We justify our use of the time average to estimate the instantaneous-time rate expectation in Appendix~\ref{app:time-av}. We also use time averaging to compute expectations of other observables which help us understand additional properties of the LE state, which are not involved in our formal verification of~\eqref{E},~\eqref{Ef}, and~\eqref{V}. In this case, it is simply convenient to have very low variance estimators since they give us a clearer picture of some qualitative features of the LE state. The left panel of Figure~\ref{fig:gazon-time-av} shows that time averaging negligibly affects the $\E[ w_i(t)]$ profile, and for other observables the discrepancy is also negligible. We will not further comment on time averaging, but we do indicate when we have time-averaged by writing $\E^n_\Delta[f( \wN(t)]$ in the plot titles. 

We fix $K=3$ in all of the simulation results shown in this section. 


\subsection{No Smooth LE}\label{subsec:not-smooth}
We begin by plotting $i/N\mapsto \E^n[ w_i(t)]$ to show it is roughly varying, which automatically implies \eqref{Ep} is not satisfied. Consider Figure~\ref{fig:gazon-time-av}. The black line depicts the expectations $\E^n[ w_i(0)]$ at initialization, which are chosen so that $\E[ w_i(0)] = \w_0(i/N)$ exactly. We take $N=6000$. Remarkably, in a time interval of macroscopic length $t=2\times 10^{-14}$, the process evolves from this smooth profile to the configuration shown by the blue dots, which are the points $(i/N,\E^n[ w_i(t)])$. We call these points a ``cloud" rather than a curve, since they are so roughly varying that they do not seem to pass the straight line test. Interestingly, the cloud seems to narrows to a point at integer values of the range. This will be discussed in Section~\ref{subsec:omega-tilde-w}.

Due to the chaotic nature of the cloud in space, one might be led to believe that it is similarly chaotic in time. However, an interesting feature of the $\E^n[ w_i(t)]$ process is that following a vanishingly small time $t_N$, the process evolves slowly. The red points in Figure~\ref{fig:gazon-time-av} depict the values of the process averaged over the macroscopic time interval $[1\times 10^{-14}, 2\times 10^{-14}]$. We see that these red points, $(i/N, \E^n_\Delta[ w_i(t)])$ are also roughly varying, and close to $\E^n[ w_i(t)]$. The fact that the chaotic cloud does not smooth out upon time averaging is a clear sign that it evolves slowly in time. 
\begin{figure}
\centering
\includegraphics[width=0.49\textwidth]{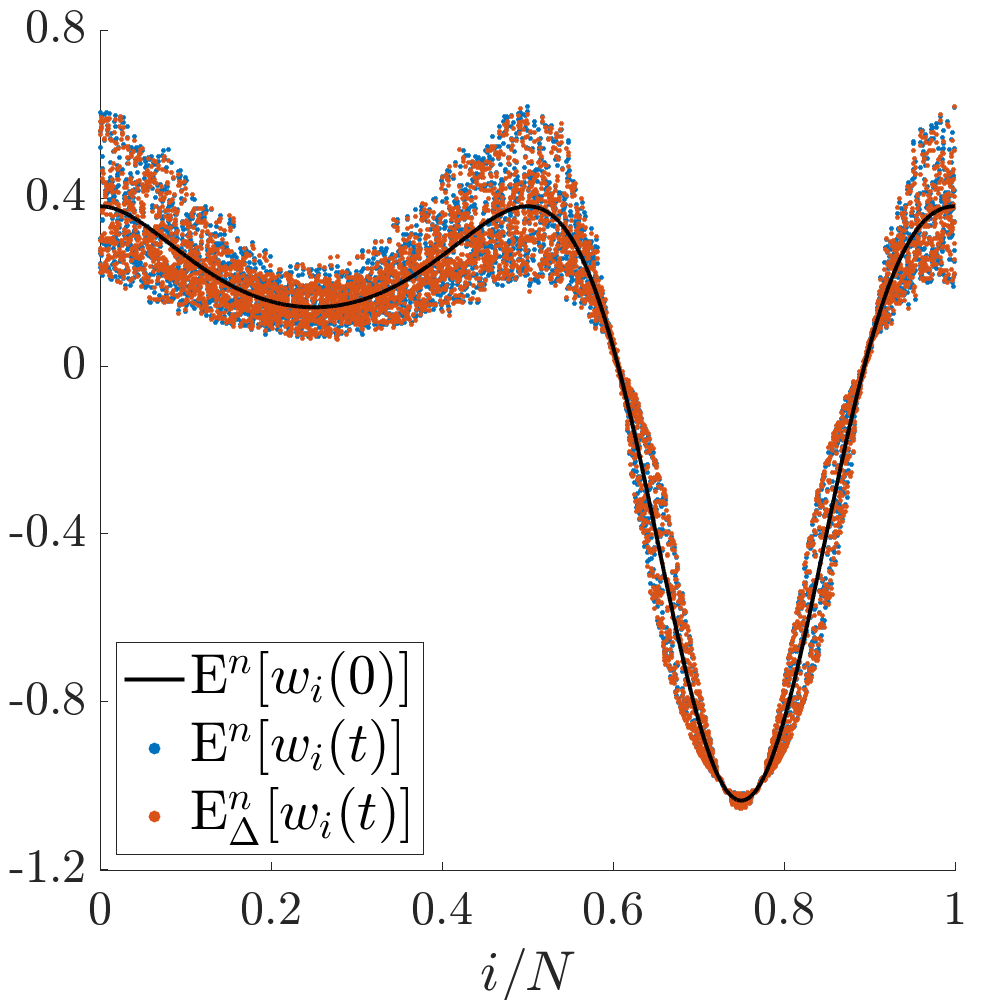}
\includegraphics[width=0.49\textwidth]{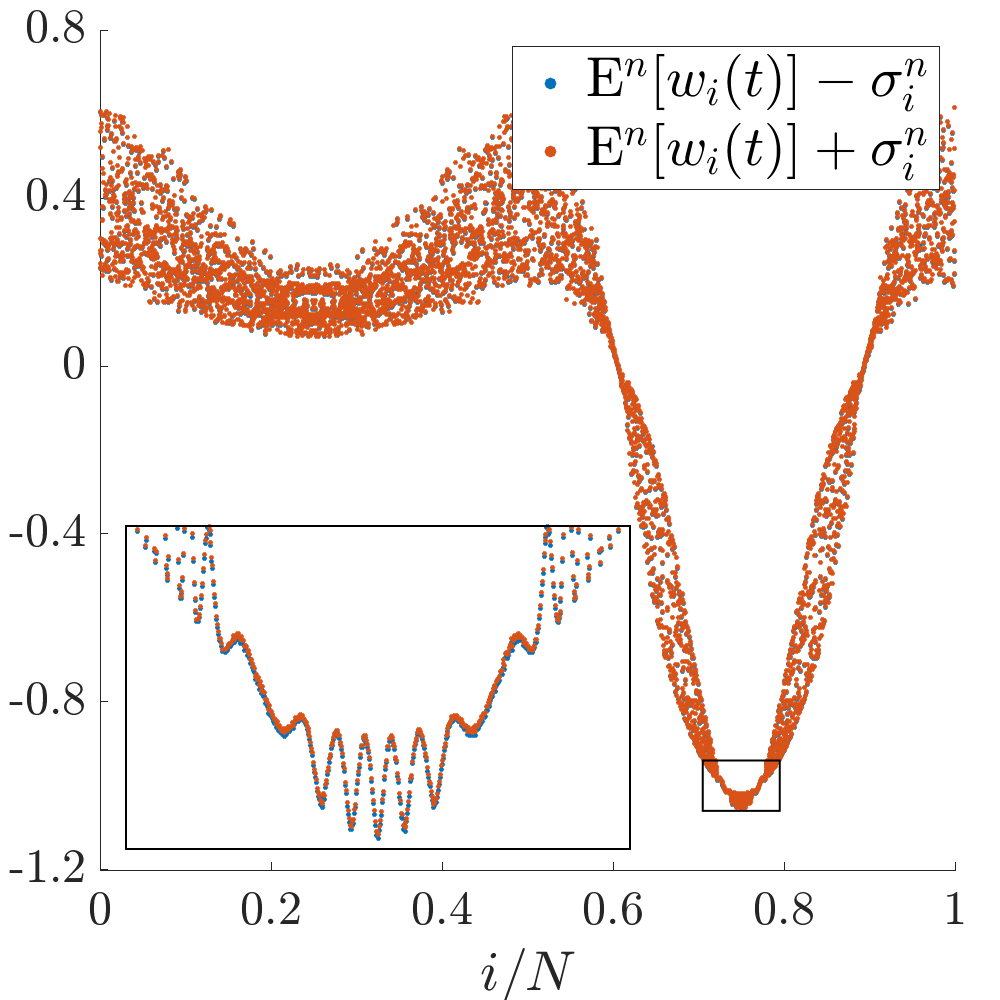}
\caption{$N=6000$. Left: the roughly varying $\E^n[ w_i(t)]$, and the time average thereof. Right: Profiles one standard deviation, denoted $\sigma^i_n$, away from the sample average $\E^n[w_i(t)]$, with $n=10^6$.} 
\label{fig:gazon-time-av}
\end{figure} 

We offer three more tests to confirm that the $i/N\mapsto \E^n[ w_i(t)]$ cloud is a real and ``permanent" phenomenon. First, we show in the right panel of Figure~\ref{fig:gazon-time-av} that the cloud is not due to randomness of the sample average approximation to $\E[ w_i(t)]$. The plot depicts a confidence interval of one standard deviation, denoted $\sigma_i^n$, around the sample mean $\E^n[w_i(t)]$, with $n=10^6$. The second and third tests are to verify that the cloud persists in time and with increasing $N$.  Figure~\ref{fig:gazon-evoln} depicts the evolution of $\E^n_\Delta[ w_i(t)]$ in time starting from two smooth initial conditions (ICs), which we call IC 1 and IC 2. They are denoted $w_0(x)=h_0''(x)$ in the plot because we evolve forward the height process to get the curvature process. Clearly, the cloud persists in time for both ICs. 
\begin{figure}
\vspace{-2cm}
\centering
\begin{subfigure}[b]{0.8\textwidth}
\includegraphics[width=\textwidth]{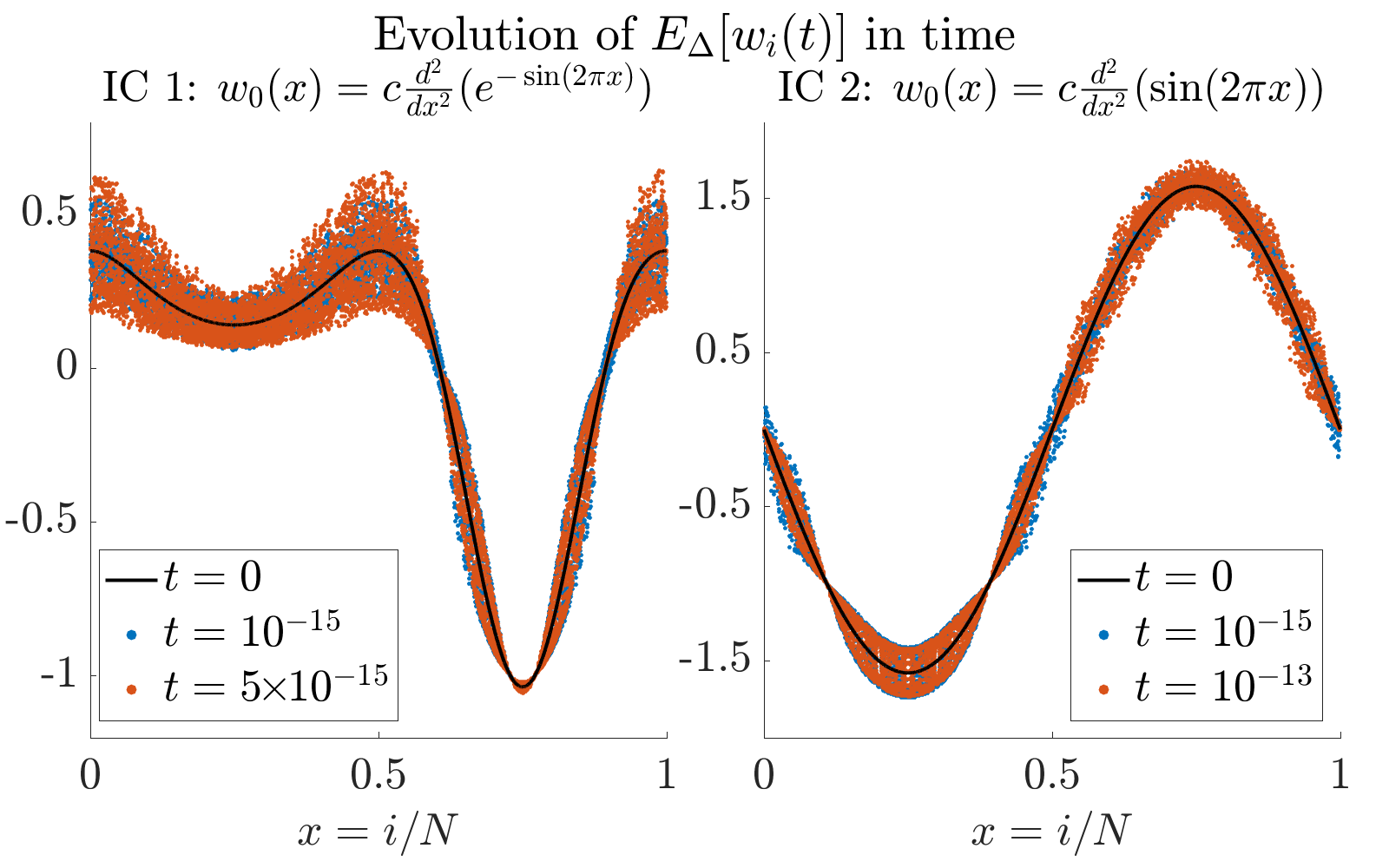}
\caption{Evolution in time of $\E^n_\Delta[ w_i(t)]$ from two initial conditions.} \label{fig:gazon-evoln}
\end{subfigure} 
\begin{subfigure}[b]{0.9\textwidth}
\centering
\includegraphics[width=\textwidth]{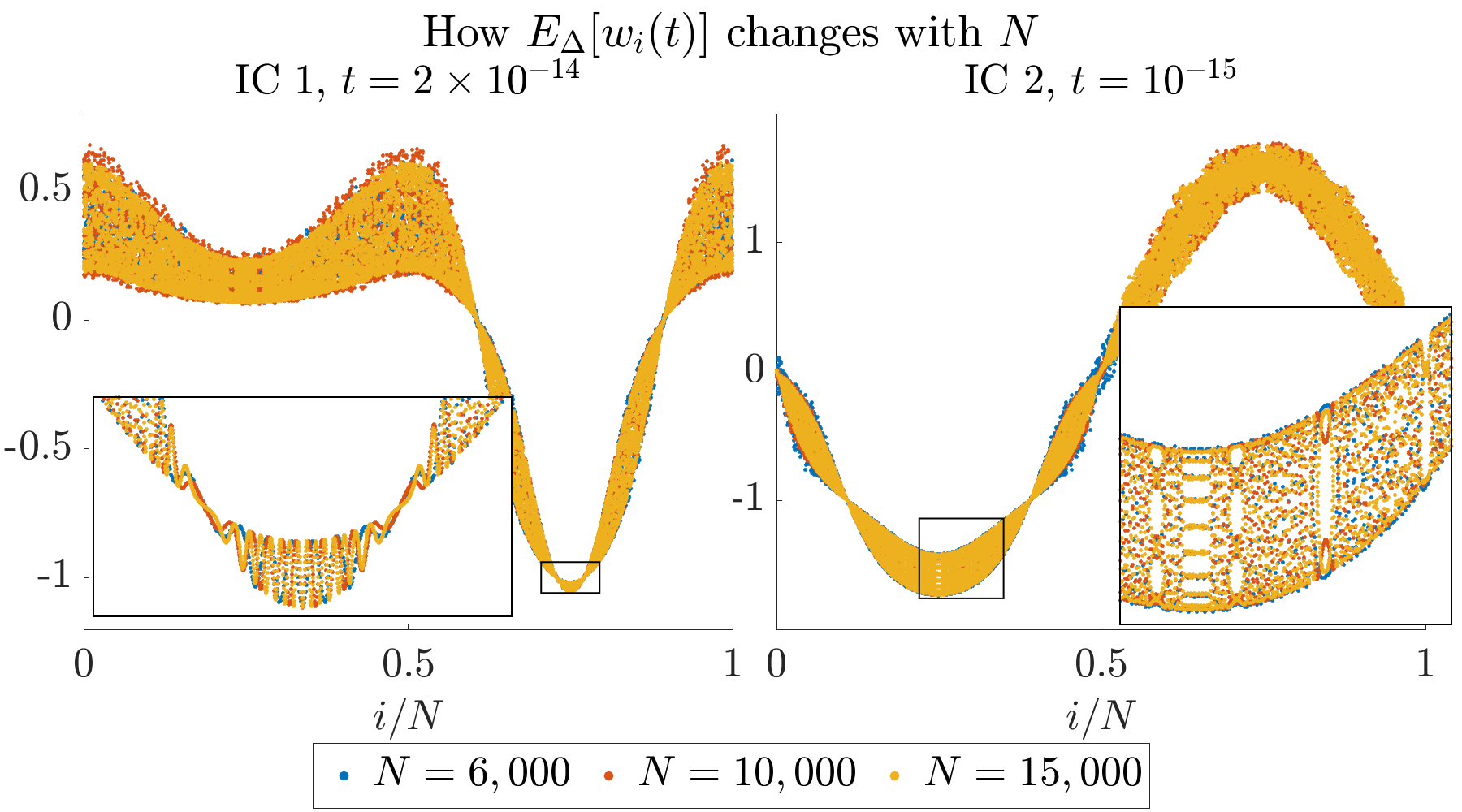}
\caption{Form of expectations $\E^n_\Delta[ w_i(t)]$ for several values of $N$. }
\label{fig:gazon-N}
\end{subfigure} 
\end{figure}
Meanwhile, Figure~\ref{fig:gazon-N} confirms that the cloud persists with increasing $N$, for both of these ICs. The figure also shows that $\E^n_\Delta[ w_i(t)]$ does not converge with $N$. This is particularly evident in the insets. 

The fact that $\E[ w_i]$ varies roughly does not rule out the possibility that the single site marginals $\mathrm{Law}( w_i(t))$ belong to a mean-parameterized measure family. However, Figure~\ref{fig:no-Efprime} shows that \eqref{Efp} is not satisfied: there is no function$\hat f$ taking $\E w_i$ to $\E f( w_i)$ in the $N\to\infty$ limit. As a result, $\mathrm{Law}( w_i(t))$ cannot belong to any mean-parameterized measure family.
\begin{figure}
\centering
\vspace{-0.4cm}
\includegraphics[width=0.7\textwidth]{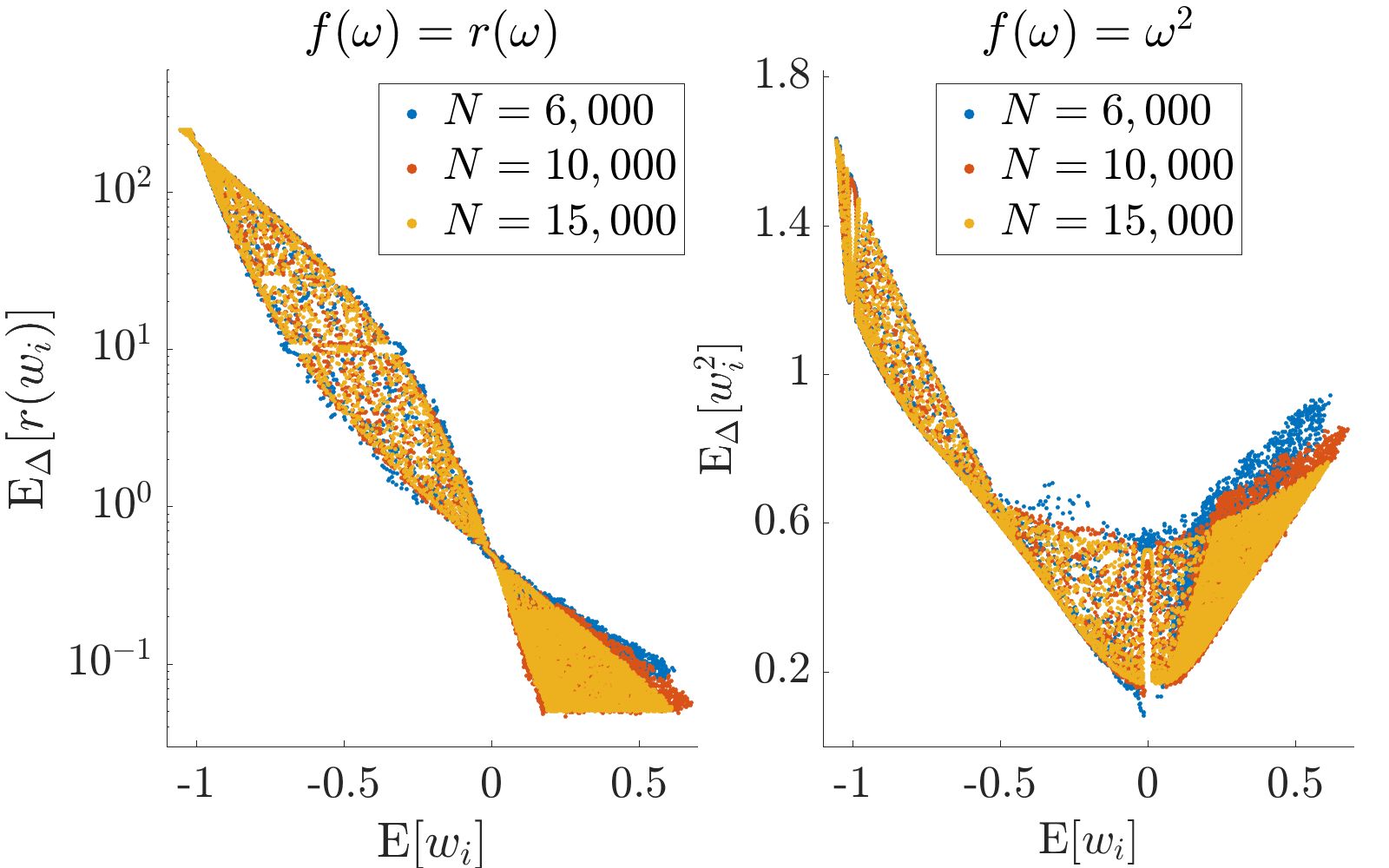}
\caption{There is no function $\hat f$ taking $\E w_i$ to $\E f( w_i)$ in the $N\to\infty$ limit. Therefore, \eqref{Efp} is not satisfied, and $\mathrm{Law}( w_i(t))$ cannot belong to any mean-parameterized measure family. Data shown from process generated by IC 1.}
\label{fig:no-Efprime}
\end{figure} 


\subsection{Verification of Rough LE}\label{subsec:rough-LE} We now confirm the limits~\eqref{V},~\eqref{E}, and~\eqref{Ef} hold, and that the limit $\w(t,x)$ in~\eqref{E} is continuous in $x$. Figure~\ref{fig:V-rough} confirms~\eqref{V}. In fact, the figure shows that 1) the window average variances go to zero uniformly across the lattice, and 2) the second $\epsilon\to0$ limit is unnecessary. For the variance to go to zero, we simply need to average over increasingly large window sizes. Indeed, keeping $\epsilon>0$ fixed, the variance decays as the window size $2N\epsilon$ increases. 
\begin{figure}
\centering
\vspace{-0.9cm}
\includegraphics[width=0.6\textwidth]{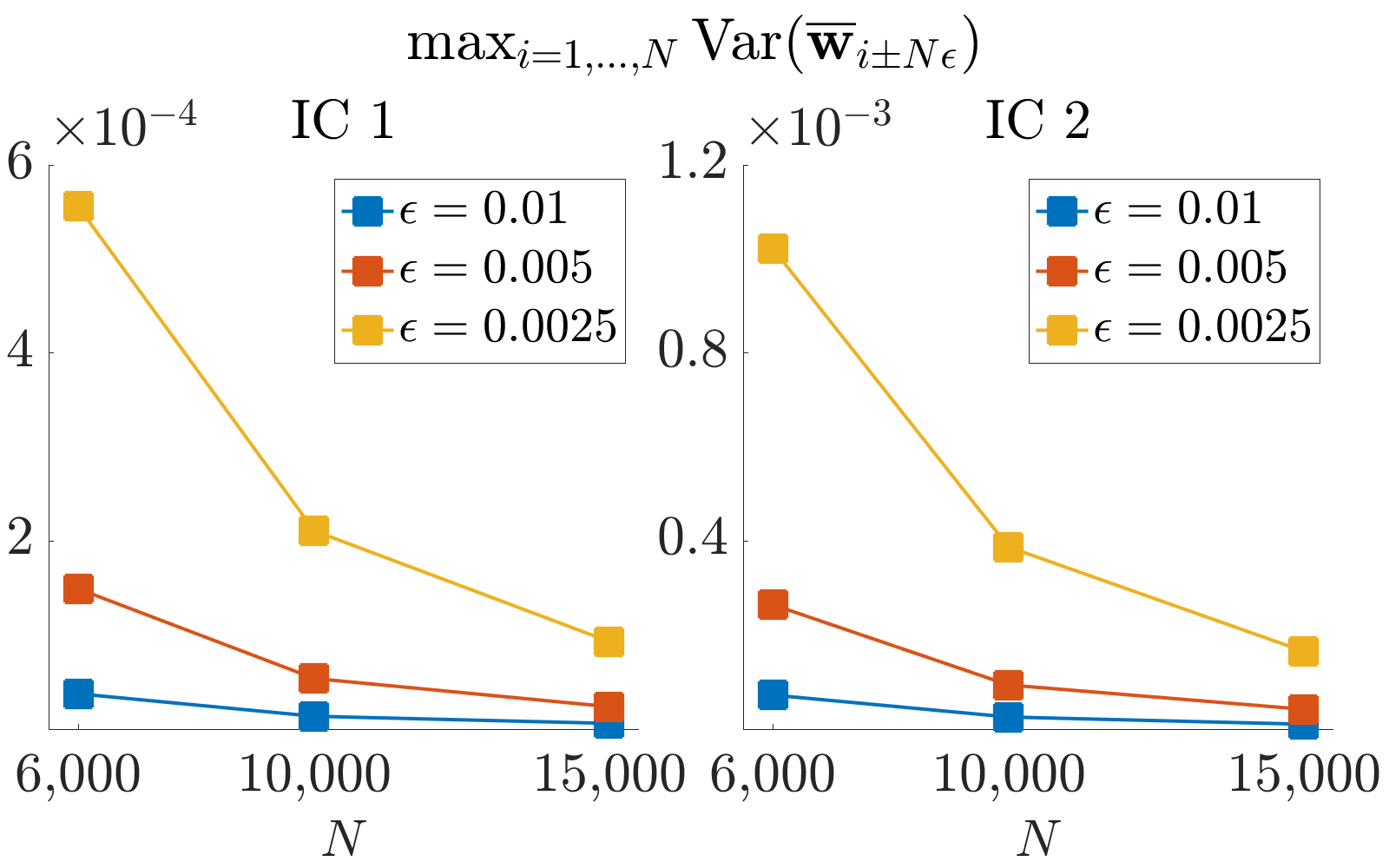}
\caption{Confirmation of~\eqref{V}. In fact, these plots suggest that the limit holds uniformly in $x$. We also see that the outer $\epsilon\to0$ limit is unnecessary: for each fixed $\epsilon>0$, the variances decay as $N$ increases.}
\label{fig:V-rough}
\vspace{-0.5cm}
\end{figure}

Next, we check~\eqref{E}. Doing so numerically is somewhat delicate, since for any finite $N$, if we take $\epsilon$ small enough the window average $\E\l[\ol{\bf w}_{\idxsetx }\r]$ will necessarily revert back to being rough and non-converging. Since we cannot take $N\to\infty$ numerically, we use the following proxy for convergence as $N\to\infty,\epsilon\to0$. For each $N$, we choose a ``good'' $\epsilon(N)$ (see details in Appendix~\ref{app:epsilon-of-N}), and then check that $\E[\ol{\bf  w}_{N(x\pm\epsilon(N))}]$ converges as $N\to\infty$. Figure~\ref{fig:E-rough} confirms that this convergence indeed holds. The continuity of the limit is clear from the plot. 
\begin{figure}
\centering
\includegraphics[width=0.75\textwidth]{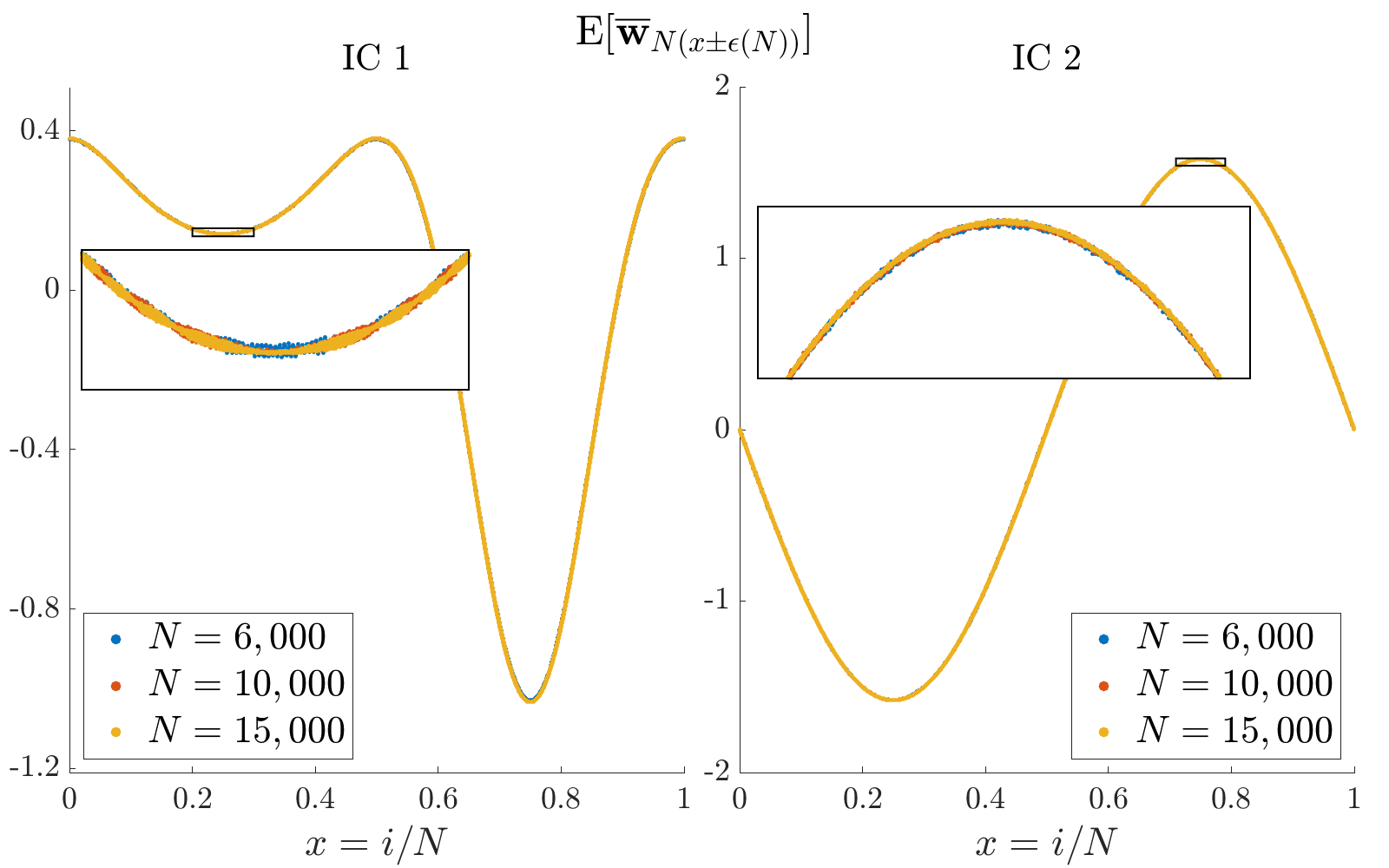}
\caption{Confirmation of~\eqref{E}. We take the limit by choosing a ``good" $\epsilon(N)$ for each $N$, and taking $N\to\infty$. For the procedure to choose $\epsilon(N)$, see Appendix~\ref{app:epsilon-of-N}.}
\label{fig:E-rough}
\vspace{-0.45cm}
\end{figure}

Finally, let us check~\eqref{Ef}. If~\eqref{Ef} holds, then for all initial conditions, for all $i=1,\dots, N$, and for all $t>0$, the pairs $(\E \ol{\bf w}_{i\pm N\epsilon},\; \E \bar f({\bf w}_{i\pm N\epsilon}))$ should approximately lie on the curve $\{(\omega, \hat f(\omega)\mid\omega\in\R\}$ for some function $\hat f$. As $N$ increases and $\epsilon$ decreases, these points should adhere to this curve more closely. Figure~\ref{fig:Ef-rough} verifies~\eqref{Ef} for $f(\omega)=r(\omega) = e^{-2K-2K\omega}$ (the rate function), and $f(\omega)=\omega^2$. As in the verification of~\eqref{E}, we choose an appropriate $\epsilon(N)$ for each $N$. To emphasize the universality of $\hat f$ across space, time, and initial conditions, we plot all the following points, for a given $N$, in the same color:
$$\l\{\bigg(\E \ol{\bf w}^{(k)}_{i\pm N\epsilon}(t_k),\; \E \bar f\big({\bf w}^{(k)}_{i\pm N\epsilon}(t_k)\big)\bigg)\;\bigg\vert\; i=1,\dots, N;\; k=1,2\r\},$$ where $w^{(1)}$ and $w^{(2)}$ denote the processes generated from ICs 1 and 2, respectively, and $t_1 =2\times 10^{-14}$, $t_2=10^{-15}$.
\begin{figure}
\centering\vspace{-1.1cm}
\includegraphics[width=\textwidth]{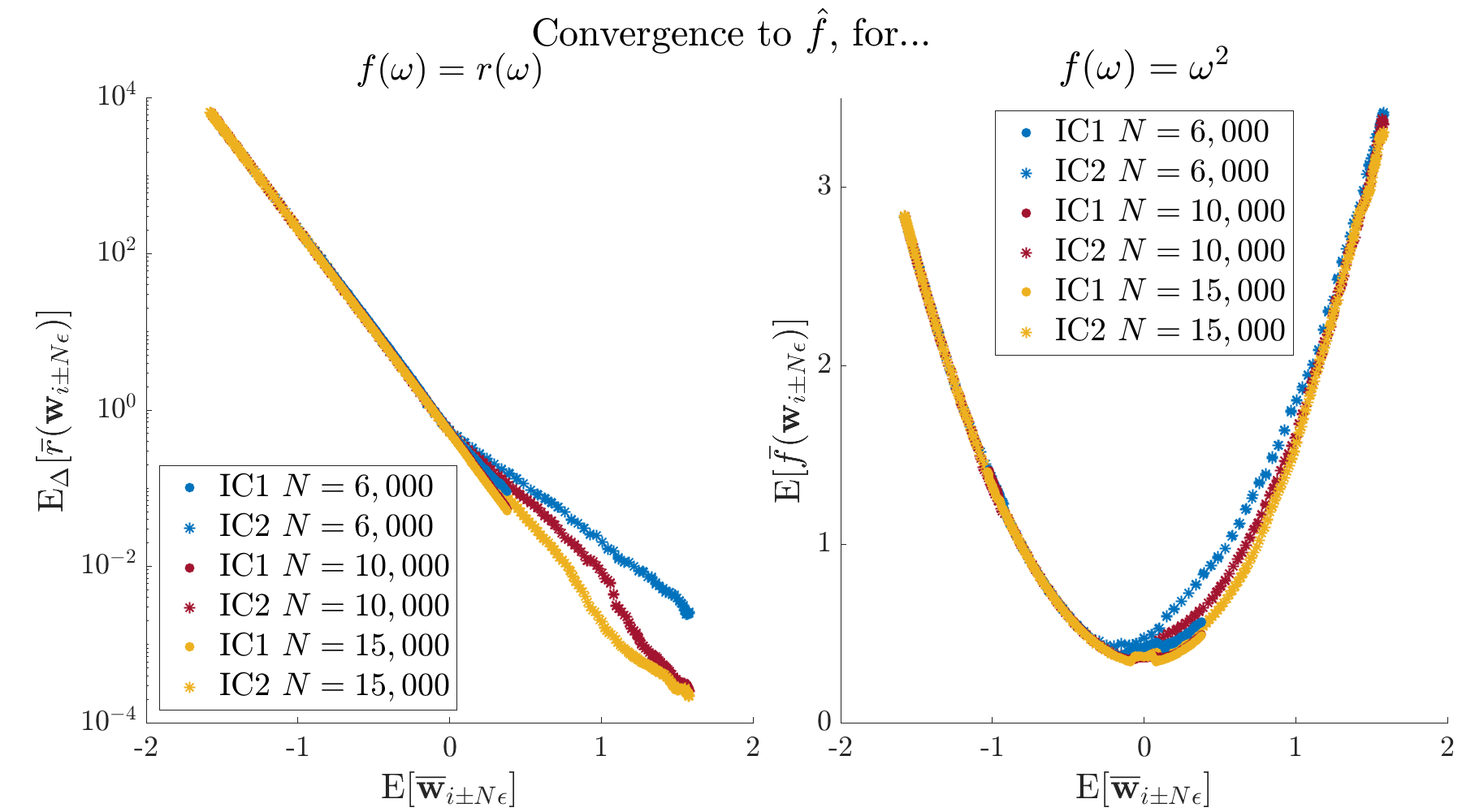}
\caption{Confirmation of~\eqref{Ef} for $f(\omega)=r(\omega)$ and $f(\omega)=\omega^2$.}
\vspace{-0.45cm}
\label{fig:Ef-rough}
\end{figure}
We see from the figure that as $N$ increases, the points corresponding to a given $N$ do in fact line up more closely on a single curve. Recall from Figure~\ref{fig:no-Efprime} that no universal function $\hat f$ sending $\E[ w_i]$ to $\E[f( w_i)]$ exists. It is therefore remarkable that upon window averaging the domain and the range there \emph{does} exist a universal function, i.e. which maps $\E \ol{\bf w}_{i\pm N\epsilon}$ to $\E \bar f({\bf w}_{i\pm N\epsilon})$. 

Note that the points to the left of zero, i.e. for which $\E\ol{\bf w}_{i\pm N\epsilon}<0$, have lined up much more neatly on a single curve than points to the right of zero, for which $\E \ol{\bf w}_{i\pm N\epsilon}>0$. This indicates that the rate of LE onset, i.e. the rate at which the finite $N$ quantities in~\eqref{V},~\eqref{E}, and~\eqref{Ef} converge to their limits for a given time $t$, is not uniform across $x$. In particular, the rate of LE onset at a point $x$ seems to depend on the value of $\E  \ol{\bf w}_{\idxsetx }(t)$ (i.e. on the value of $\w(t,x)$, to which $\E  \ol{\bf w}_{\idxsetx }(t)$ converges). There is a straightforward informal explanation for this observation. It is intuitively clear that local equilibration (i.e. mixing) requires movement of particles, and there is little movement in regions where the rates are very small. From the left panel of Figure~\ref{fig:Ef-rough}, we see that on average, the rates become exponentially smaller as $\E\ol{\bf w}_{\idxsetx}$ increases. Therefore, the larger the value of $\E \ol{\bf w}_{\idxsetx }$, the slower the local equilibration.

\subsection{Qualitative Study}\label{subsec:qual}
We conclude the section with some further qualitative observations about the $ \wN$ process's local equilibrium state, based on the form of $i/N\mapsto \E[f_\Delta( w_i)]$ profiles. Figure~\ref{fig:qual} depicts the expectations $\E_\Delta[ w_i]$, $\E_\Delta[ w_i^2]$ and $\E_\Delta[r( w_i)]$ for the processes generated from the two ICs, and $N=15,000$. 
\begin{figure}
\centering
\includegraphics[width=0.8\textwidth]{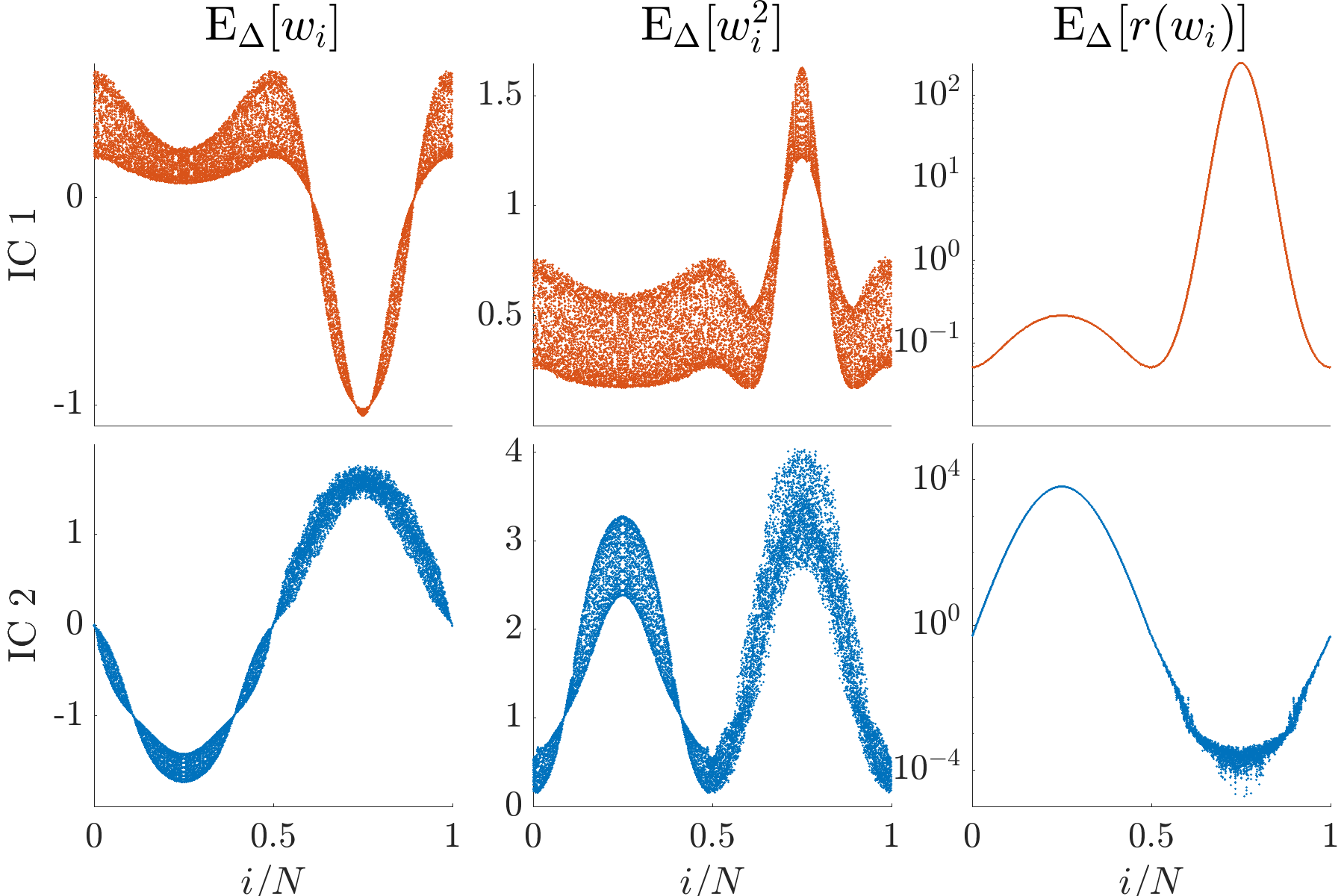}
\caption{Observable expectations}
\label{fig:qual}
\end{figure}

First, consider the bottom row, associated to IC 2. Note that for each of the three observables, there is a qualitative difference between the form of the profile in the region $(0.6, 0.9)$ (approximately) and the form elsewhere. For example, outside this region we observe that the $(i/N,\E_\Delta [w_i])$ cloud narrows to a point when $\E_\Delta[ w_i]$ takes the value 0 or -1 (an integer) But inside this region, the cloud does not narrow to a point near integer values of the range (i.e. where $\E_\Delta[ w_i]=1$). 

This has to do with the fact that the rate expectations in the region $(0.6,0.9)$ are on the order of $10^{-4}$, which is extremely small relative to the largest rate expectation on the order of $10^4$. This is shown in the bottom rightmost panel. As we discussed above, we expect LE onset to be slowest where the rates are smallest, so the qualitative form of observable expectations in $(0.6, 0.9)$ is \emph{not} characteristic of the LE state. Discounting this region for the second process, we make the following three observations about the form of these expectations. 
\begin{enumerate}
\item \textbf{The rate expectations $\E[r( w_i)]$ are smoothly varying.} This is not true of all observable expectations, since
\item \textbf{The means $\E w_i$ are roughly varying,} as are the second moments.
\item \textbf{The $(i/N,\E w_i)$ cloud narrows to a point near integer values of the range.}
\end{enumerate}
The smoothness of the rate expectations is surprising given that the first and second moments are so roughly varying. In the next section, we will give a partial explanation for this by showing that $\E[ w_i]$ are roughly varying \emph{because} $\E[r( w_i)]$ are smoothly varying. We will also explain the third observation.  

In the present section, we gleaned useful information about the LE state by probing~\eqref{V},~\eqref{E} and~\eqref{Ef} numerically, without ever referring to the explicit form of the distribution. In the next section, we will obtain a deeper understanding of the rough LE state using the explicit, limiting form of $\mathrm{Law}( w_i)$ (as confirmed by numerics, but without proof). Knowing the LE distribution explicitly is a rare privilege; for many interacting particle systems, there is no closed form expression for the LE state. 

When the distribution is unknown, studying~\eqref{V},~\eqref{E}, and~\eqref{Ef} numerically is a useful tool to understand essential properties of the LE state. In fact, numerical simulations like the ones shown here can be used to derive the PDE limit for IPS in which the LE state, and therefore also the PDE, have no explicit analytic representation. We do this in~\cite{metrop_pde} for a crystal surface dynamics with Metropolis-type rates.

\section{A Look under the Hood at the Rough LE State}\label{sec:theory} 
Although the previous section shows that the crucial limits~\eqref{V},~\eqref{E}, and~\eqref{Ef} are satisfied, several questions remain unresolved. The key question we address in this section is the following: given that $\mathrm{Law}( w_i)$ do not lie in a mean-parameterized measure family, where does the function 
$$\hat f: \E  \ol{\bf w}_{i\pm N\epsilon}\mapsto \E \bar f({\bf w}_{i\pm N\epsilon})$$ come from? To answer this and other questions, we begin in Section~\ref{subsec:arr-prop} by writing down the form of $\Law(w_i)$ explicitly, and confirming this form numerically. At the end of~\ref{subsec:arr-prop}, we summarize all assumptions used in the rigorous proofs in Sections~\ref{subsec:omega-tilde-w} and~\ref{subsec:marg-av}. These include the numerically confirmed limit~\eE, and a slightly stronger version of what we can confirm numerically about $\Law(w_i)$. In Section~\ref{subsec:omega-tilde-w}, we rigorously prove further properties of the LE state using the assumptions. In particular, we address Observations 1,2,3 of Section~\ref{subsec:qual}, explaining how the intriguing shape of the $i/N\mapsto \E w_i$ profile arises, and how this shape is related to the smoothness of the rate expectations. The heart of this section is~\ref{subsec:marg-av}, in which we rigorously prove that~\eEf~is satisfied thanks to mesoscopic averaging of $\Law(w_i)$. We conclude in Section~\ref{subsec:arr-discush} by synthesizing how the rough LE arises, and how the rough LE affects the PDE limit. 

\subsection{The form of Law$(\wN)$, via Law$(\zN)$}\label{subsec:arr-prop} Although we have been studying $\wN$ independently of $\zN$ and $\hN$, we will now reintroduce the $\zN$ process. This will allow us to deduce the form of $\Law(\wN)$ from $\Law(\zN)$, which we know more about. Let $\zN(t) = (z_1(t),\dots,  z_N(t))$ be any process satisfying $z_i(t)-z_{i-1}(t)=w_i(t)$. 
This process $\zN$ will have generator $\LN$ defined as in~\eqref{w-gen}, but with $\mbf w, \mbf w^{i,j}$ replaced by $\mbf z,\mbf z^{i,j}$, respectively. 
Recall from Section~\ref{subsec:dynamics} that the ergodic invariant measures of the $\zN$ process (i.e. the global equilibrium measures) take the form $\Phi_N(\mbf z\mid \sum_iz_i=S_0, \sum_iiz_i=S_1)$, where $\Phi_N(\bz) \propto e^{-KH(\bz)} = \e(-K\sum_{i=1}^Nz_i^2).$ By a conditional limit theorem, we should have 
\beqs\label{rho-approx-0}
\Phi_N(\mbf z\mid &\sum_iz_i=S_0, \sum_iiz_i=S_1)\\
&\approx \frac{1}{\calZ_{a_0,a_1}}\e\bigg(-K\sum_{i}z_i^2+2Ka_0\sum_iz_i + 2Ka_1\sum_iiz_i\bigg)\eeqs for $N\gg 1$, where $a_0$ and $a_1$ are functions of $S_0$ and $S_1$. Define the one-parameter family $\rho_K[\cdot]$ by \beq\label{gibbs-dist}\gibbs K \lambda (n) = \frac{e^{-K(n-\lambda)^2}}{\calZ_K(\lambda)},\quad \quad\calZ_K(\lambda)=\sum_{m=-\infty}^\infty e^{-K(m-\lambda)^2}.\eeq  By completing the square in the exponent in~\eqref{rho-approx-0}, we see that the measure on the second line takes the form $\rhoN_{\bml_N}= \otimes_{i=1}^N\rho_K[a_0+ia_1]$. Thus, in global equilibrium, the parameters $\lambda_i=a_0+ia_1$ associated to $z_i$ form a single line, determined by the two global parameters $a_0$ and $a_1$. In \emph{local} equilibrium, we expect the distribution of $\zN(t)=\zNtild(N^4t)$ to also take the form $\otimes_{i=1}^N\rho_K[\lambda_i]$, where $\lambda_i\approx a_0(x)+ia_1(x)$ are linear in each mesoscopic region $i\in\idxsetx$. We make the following conjecture which, for now, does not assume the $\lambda_i$ have any special structure.
\begin{conjecture}[Unverified]\label{conj:z} For each $t$ there exist $\bml_N(t) = (\lambda_1(t),\dots, \lambda_N(t))$ such that as $N\to\infty$, the below approximation becomes accurate with respect to some divergence (e.g. relative entropy):  
\beq\label{rho-approx}\Law(\zN(t))\approx \rhoN_{\bml_N(t)}(\bz) \propto \prod_{i=1}^N\rho_K[\lambda_i(t)](z_i),\quad N\to\infty.\eeq 
\end{conjecture}
Note that we can think of $\gibbs K \lambda$ as the discrete (integer-supported) version of the normal distribution $\mathcal N(\lambda, 1/2K)$. Define
$$u_D(\lambda)=m_1(\rho_K[\lambda])=\sum_{n=-\infty}^\infty ne^{-K(n-\lambda)^2}\bigg/\calZ_K(\lambda)$$ the first moment of $\gibbs K\lambda$, and let $\lambda_D =u_D^{-1}$. For comparison, we let $u_C(\lambda)=\lambda$, the first moment of the distribution $\mathcal N(\lambda, 1/2K)$, and $\lambda_C=u_C^{-1}$, also the identity function. 

\begin{wrapfigure}{r}{0.5\textwidth}
\centering
\vspace{-0.8cm}
\includegraphics[width=0.5\textwidth]{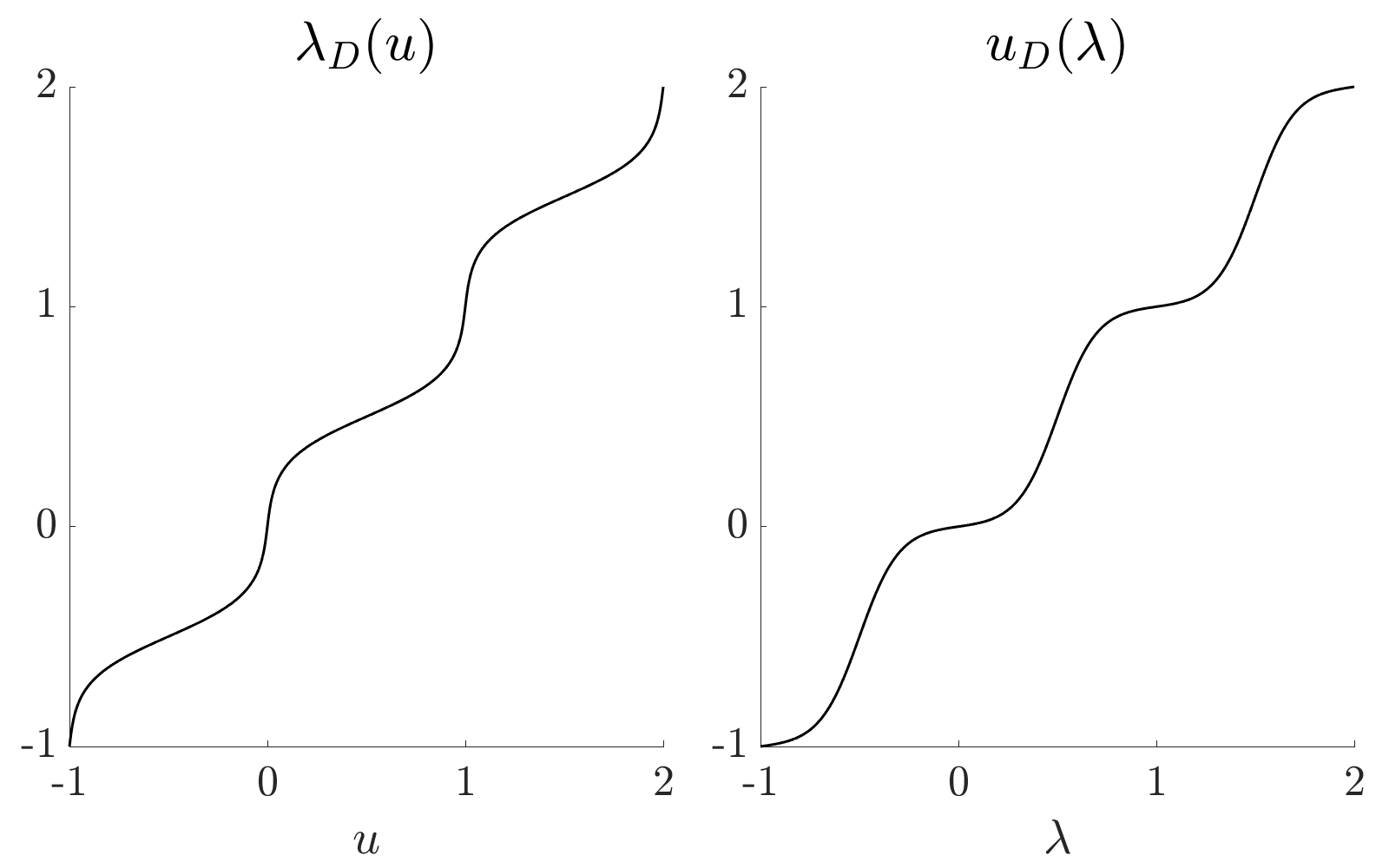}
\vspace{-0.6cm}
\caption{$u_D(\lambda) = m_1(\rho_K[\lambda])$, and $\lambda_D=u_D^{-1}$, when $K=5$. The oscillations are due to the discreteness of the support of $\rho_K[\lambda]$.}
\label{fig:u-lambda}
\vspace{-0.7cm}
\end{wrapfigure} 

Figure~\ref{fig:u-lambda} shows a plot of $u_D(\lambda)$ and $\lambda_D(u)$. Both functions implicitly depend on $K$. The oscillations of $u_D$ and $\lambda_D$ around the identity functions $u_C$ and $\lambda_C$ are a consequence of the restricted support of $\gibbs K \lambda$ (restricted compared to the continuous normal distribution).

We will now partially confirm that the $ \wN$ distribution induced by the conjectured $\zN$ distribution~\eqref{rho-approx} is correct. Since we are interested in $ \wN$ through the lens of~\eqref{V},~\eqref{E}, and~\eqref{Ef}, we do not need to know the full, $N$ site distribution $\Law( \wN)$. Indeed, the limit~\eqref{V} only depends on the two site distributions $\Law( w_i,  w_j)$ through $\Cov( w_i,  w_j)$, and the limits~\eqref{E} and~\eqref{Ef} only depend on the one site distributions $\Law( w_i)$ through $\E[ w_i]$, $\E[f(w_i)]$. 

That $\mathrm{Law}(\zN)$ is given by a product measure in the $N\to\infty$ limit implies that we should have $\Corr( w_{Nx+i},  w_{Nx+i+k})\to0$ as $N\to\infty$ for all $i$ and lag $k\geq2$. Figure~\ref{fig:corrs-rough} confirms the correlation decay: the plots depict $\max_{2\leq k\leq 9}\l|\Corr( w_{Nx},  w_{Nx+k})\r|$ as a function of $x=i/N$ with increasing $N$, for the two processes generated from IC 1 and IC 2.  In addition to confirming the correlation structure predicted by~\eqref{rho-approx},  Figure~\ref{fig:corrs-rough} also shows that $\max_{k\geq 2}\l|\Corr( w_{Nx},  w_{Nx+k})\r|$ is a useful metric for the rate of convergence to LE with $N$ in different spatial regions. Indeed, in the left panel we see that for the IC 1 - generated process, the correlations are lowest in the region $i/N\in (0.6, 0.9)$, where the rates are highest (the rates are shown in the upper rightmost panel in Figure~\ref{fig:qual}). This is in line with the intuition that regions with large jump rates are expected to equilibrate more quickly. For IC 2, the correlations are lowest in the region $i/N\in (0.1, 0.4)$, also where the rates are highest (shown in the lower rightmost panel in Figure~\ref{fig:qual}).
\begin{figure}
\centering
\includegraphics[width=0.7\textwidth]{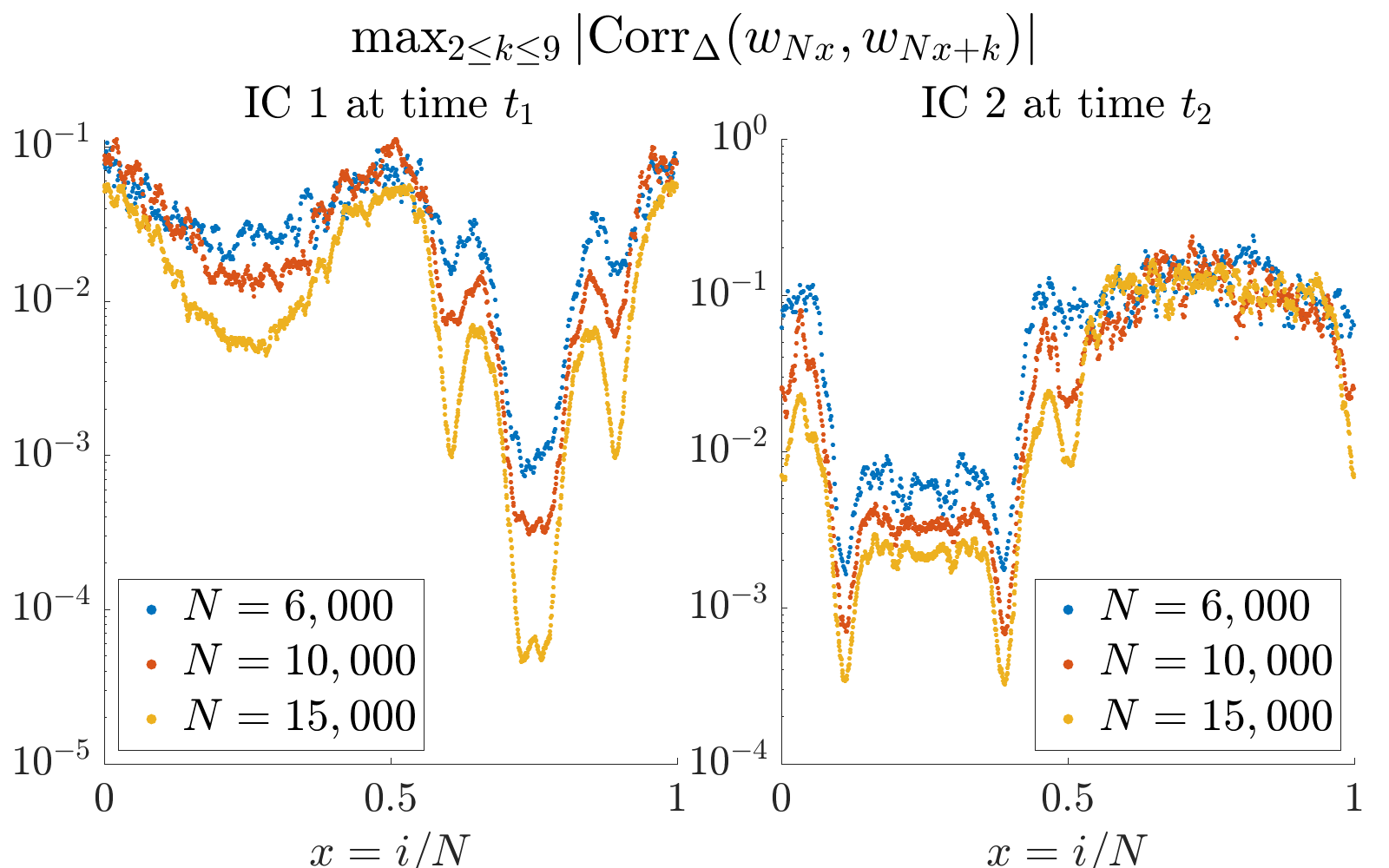}
\caption{Confirming that the correlations $\Corr( w_i,  w_{i+k})$ decay for all $k\geq 2$, as predicted by~\eqref{rho-approx}. The figure also reveals that correlations are a useful metric for the rate of convergence of different spatial regions to LE.} 
\label{fig:corrs-rough}
\vspace{-0.85cm}
\end{figure}
This is all we have to say about pair distributions $\Law( w_i,  w_{i+k})$. From now on, we are exclusively interested in $\Law( w_i)$. 

Let $\mu_K[\omega,\lambda]$ denote the distribution on $Z-Z'$ induced by $(Z', Z)\sim \rho_K[\lambda-\omega]\otimes\rho_K[\lambda]$, so that
\beqs\label{muK-approx}
\Law(z_{i-1}, z_i)&\approx \rho_K[\lambda_{i-1}]\otimes\rho_K[\lambda_i]\\\Rightarrow \Law( w_i) &\approx \mu_K[\omega_i, \lambda_i],\eeqs where $$\omega_i = \lambda_i- \lambda_{i-1}.$$
Note that if $Z'$ and $Z$ were independent \emph{continuous} normals $\mathcal N(\lambda-\omega, 1/2K)$ and $\mathcal N(\lambda, 1/2K)$, then $\Law(Z-Z')$ would be $\mathcal N(\omega, 1/K)$, depending only on one parameter $\omega$ (we are treating $K$ as fixed). However, due to the discreteness of the support of $\rho_K[\lambda]$, the distribution $\mu_K[\omega,\lambda]=\Law(Z-Z')$ belongs to a two-parameter family, depending on both $\omega$ and $\lambda$. We have
\beqs\mu_K[\omega,\lambda](n)&=\mathbb P(Z-Z' = n)\\ &= \sum_{m=-\infty}^\infty\gibbs K {\lambda-\omega}(m)\gibbs K {\lambda}(m+n)\\
&=\sum_{m=-\infty}^\infty\frac{\e\l(-K(m-\lambda+\omega)^2-K(m+n-\lambda)^2\r)}{\calZ_K(\lambda-\omega)\calZ_K(\lambda)}
\eeqs
Expanding the squares and rearranging terms, we get
\beqs\label{mu-K}\mugibbs K \omega \lambda(n) =\e\l(-\frac K2(n-\omega)^2\r) Q_K\l(n,\omega,\lambda\r),\eeqs
where
\beq\label{Q}
Q_K\l(n,\omega,\lambda\r)=\frac{\calZ_{2K}\bigg(\lambda-\frac12\bigg[\omega+n\bigg]\bigg)}{\calZ_K\big(\lambda-\omega\big)\calZ_K\big(\lambda\big)}.
\eeq
Note that $Q_K$ depends on $\lambda$ through $\lambda\bmod 1$, since the normalization constant functions $\calZ_K$ and $\calZ_{2K}$ are one-periodic. For the first moment of $\mu_K[\omega,\lambda]$ we have \beq\label{m1muK}m_1(\mu_K[\omega, \lambda]) = \E Z - \E Z' = u_D(\lambda) - u_D(\lambda - \omega).\eeq This first moment depends in a non-trivial way on both $\omega$ and $\lambda$, implying that in general, expectations of observables under this distribution are not functions of the first moment alone. The same was observed to be true of the $ w_i$ in Figure~\ref{fig:no-Efprime} of Section~\ref{subsec:not-smooth}: there is no function $\hat f:\E w_i\mapsto \E[f( w_i)]$, indicating that the $ w_i$ are not mean-parameterized. 

For future reference, let us also compute the rate expectation with respect to $\mu_K[\omega,\lambda]$. Recall that $r(w) = e^{-2K - 2Kw}$. Using that $\mu_K[\omega,\lambda]=\Law(Z-Z')$ where $(Z',Z)\sim \gibbs K {\lambda-\omega}\otimes \gibbs K {\lambda}$, we get
\beqs\label{r-exp-I}
\mu_K[\omega,\lambda](r) = e^{-2K}\E\l[e^{2KZ'}\r]\E\l[e^{-2KZ}\r].
\eeqs
Now, we compute 
\beqs\label{exp-cK-lam}
\E_{Z\sim\rho_K[\lambda]}\e\l(cKZ\r) &= \frac{\sum_n \e\l(cKn\r)\e\l(-K(n-\lambda)^2\r)}{\calZ_K(\lambda)}\\
&= \e\l(cK\lambda+c^2K/4\r)\frac{\sum_n\e\l(-K(n-\lambda-c/2)^2\r)}{\calZ_K(\lambda)}\\
&=  \e\l(cK\lambda+c^2K/4\r)\frac{\calZ_K(\lambda+c/2)}{\calZ_K(\lambda)}.
\eeqs 
$\calZ_K(\lambda)$ has period 1, so when $c$ is an even integer, the ratio of $\calZ_K$'s is equal to 1. Using~\eqref{exp-cK-lam} in~\eqref{r-exp-I} with $c=\pm2$, we get
\beqs\label{r-exp}
 \mu_K[\omega,\lambda](r) =e^{-2K}\e\l(K+2K(\lambda-\omega)\r)\e\l(K-2K\lambda\r)=\e\l(-2K\omega\r).
\eeqs
We see that the rate observable has the rare property that its expectation with respect to $\mu_K[\omega,\lambda]$ does not depend on $\lambda$! We return to this point later on.

We will now show numerically that indeed, there exist $\lambda_1,\dots, \lambda_N$ such that the approximation $\Law( w_i)\approx \mu_K[\omega_i,\lambda_i]$ becomes increasingly accurate in the limit. This will be sufficient for our analysis, and we will not need that $\Law(z_{i-1}, z_i)\approx \rho_K[\lambda_{i-1}]\otimes\rho_K[\lambda_i]$. Note that in the above calculations, it was convenient to use the representation $\mu_K[\omega,\lambda] = \Law(Z-Z')$ for $(Z',Z)\sim\rho_K[\lambda-\omega]\otimes\rho_K[\lambda]$ to compute the rate expectation $\mu_K[\omega,\lambda](r)$ and first moment $m_1(\mu_K[\omega,\lambda])$. However, we could have also used the pmf~\eqref{mu-K} to do these computations. The following properties summarize our numerical findings, which include a crucial observation regarding the smoothness of the $\omega_i$.  
\begin{property}\label{gibbs-num} There exist $\lambda_1,\dots, \lambda_N$ such that...
\begin{itemize}
\item[1.1] The following limit holds: \vspace{-0.1cm}$$\max_{i=1,\dots, N}H\bigg(\mathrm{Law}( w_i)\,\bigg\vert\, \mu_K[\omega_i, \lambda_i]\bigg)\to0,\quad N\to\infty,$$ where $\omega_i=\lambda_{i}-\lambda_{i-1}$ and $\mu_K[\omega,\lambda]$ is the pmf defined in~\eqref{mu-K}.
\item[1.2] The $\omega_i$ vary smoothly in the neighborhood of any $x$ (recall Definition~\ref{smth-rgh}). In fact, they satisfy the stronger property that for each $t,x$, we have
$$\ol{\lim_{\epsilon\to0}}\;\ol{\lim_{N\to\infty}}\max_{\argidxsetx i}N|\omega_{i+1}-\omega_{i}|<C=C(t,x)<\infty.$$
\end{itemize}
\end{property}
Note that the smoothness of $\omega_i$ is in line with the claim above Conjecture~\ref{conj:z}, that in mesoscopic regions around $x$ we can make the approximation $\lambda_i=a_0(x)+ia_1(x)$. Figure~\ref{fig:prop} verifies these two properties numerically for the process generated from IC 1. For the details of how we found the $\lambda_i$, see Appendix~\ref{app:loc-gibbs} and Remark~\ref{remark:tautology}.
\begin{figure}
\centering
\vspace{-0.7cm}
\includegraphics[width=0.8\textwidth]{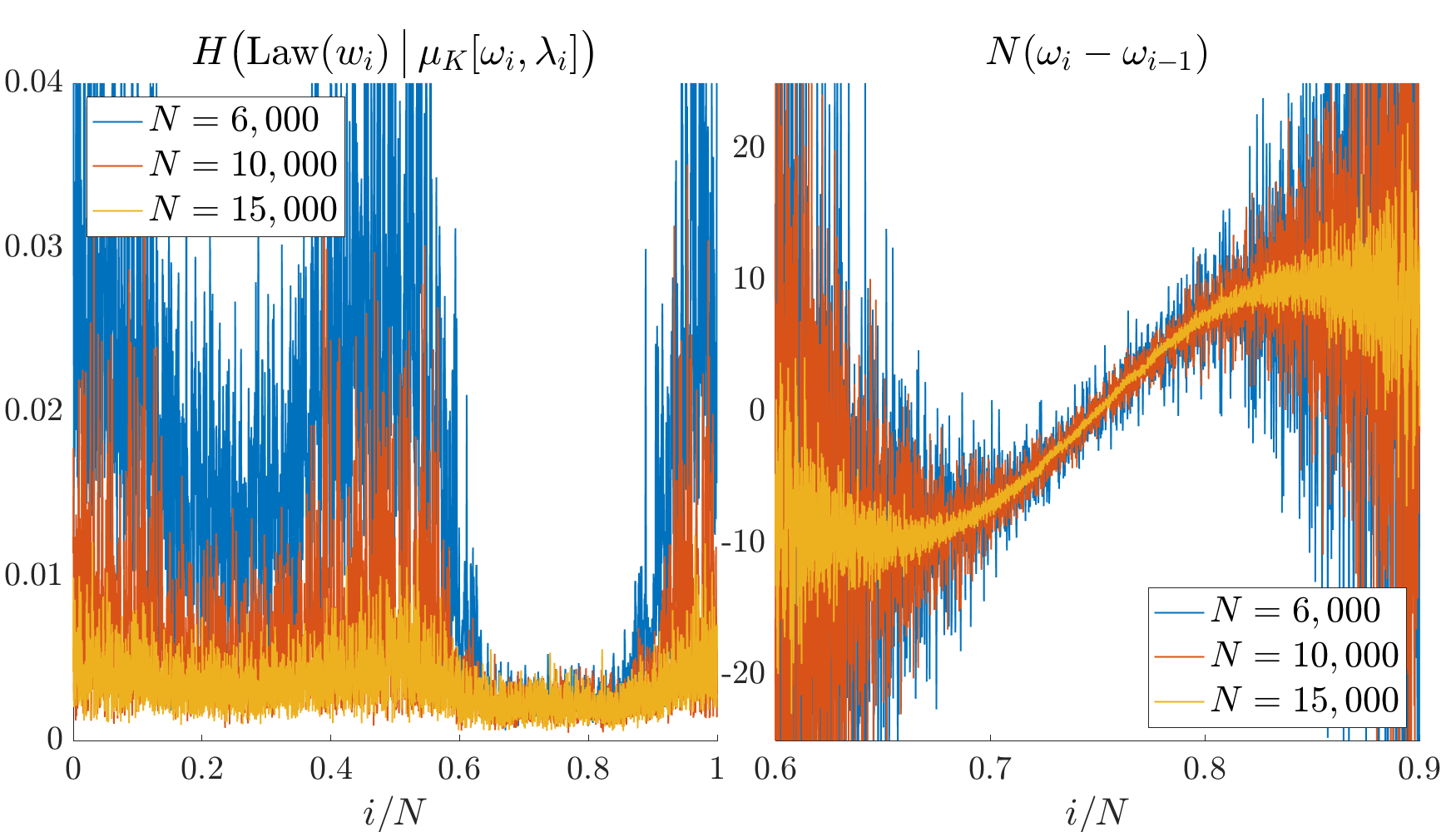}
\caption{Numerical Verification of Properties 1.1 and 1.2. For the details of how we found the $\lambda_i$, see Appendix~\ref{app:loc-gibbs}. The KL divergences are computed based on 1000 samples. The plot on the right confirms that the limit $N(\omega_i - \omega_{i-1})$ exists as $N\to\infty$, for $i/N$ in the most equilibrated region $(0.6, 0.9)$.}
\label{fig:prop}
\vspace{-0.5cm}
\end{figure}
Note that the KL divergence is lowest for $i/N\in (0.6, 0.9)$, which we identified as the region which is quickest to equilibrate, based on the correlation plot Figure~\ref{fig:corrs-rough} and the rate expectation plot in Figure~\ref{fig:qual}. The smoothness of the $\omega_i$ is very sensitive to LE onset, so we are only able to show Property 2 for $x\in (0.6, 0.9)$.

We will use these numerical results in Section~\ref{subsec:omega-tilde-w} to explain Observations 1-3 (made at the end of Section~\ref{subsec:qual}), and in Section~\ref{subsec:marg-av} to explain why~\eqref{Ef} is satisfied. For the sake of a clear presentation, let us list all the facts we will assume to be true from now on. Assumption 2.1 below is a stronger version of the numerically verified Property 1.1. Assumption 2.2 is the same as Property 2.2. The third assumption is that~\eqref{E} holds, which was shown numerically in Section~\ref{subsec:rough-LE}. 
\begin{assump}\label{gibbs-theory} \textcolor{white}{.}\\
\begin{itemize}
\item[2.1] For some $N_0$, and all $N>N_0$, there exist $\lambda_i$, $i=1,\dots, N$ such that $$\mathrm{Law}\big( w_i\big)= \mu_K[\omega_i, \lambda_i]\quad\forall i=1,\dots,N,$$ where $\omega_i:=\lambda_i-\lambda_{i-1}$ and $\mu_K[\omega,\lambda]$ is the pmf defined in~\eqref{mu-K}.
\item[2.2] The $\omega_i$ satisfy
$$\ol{\lim_{\epsilon\to0}}\;\ol{\lim_{N\to\infty}}\max_{\argidxsetx i}N|\omega_{i+1}-\omega_{i}|<C.$$
\item[2.3] There exists a continuous function $\w(t,\cdot)$ such that
$$\lim_{\epsilon\to0}\lim_{N\to\infty}\E\ol{\bf w}_{\idxsetx } = \w(t,x),$$ that is,~\eqref{E} holds.
\end{itemize}
\end{assump}
To be clear, Assumption 2.1 does not actually hold, but it will simplify the analysis significantly. Handling errors incurred by the fact that $\mathrm{Law}\big( w_i\big)\neq \mu_K[\omega_i, \lambda_i]$ for finite $N$ would not add any insight. It would only obscure our understanding of the answers to the questions outlined at the beginning of this section. 

 \subsection{Implications of the Smoothness of $\omega_i$}\label{subsec:omega-tilde-w}  Note that we can write
\beq\label{uD-lambda}u_D(\lambda) = u_C(\lambda) + u_o(\lambda) = \lambda + u_o(\lambda),\eeq where $u_o$ is bounded and has period 1 (``o'' for ``oscillating''), and is odd about $\lambda=1/2$. Using~\eqref{uD-lambda} and the fact that $\E[ w_i] = u_D(\lambda_i) - u_D(\lambda_i-\omega_i)$, we have 
\beq\label{EW-delam}\E w_i = \omega_i + u_o(\lambda_i) - u_o(\lambda_{i}-\omega_i).\eeq 
From here, we have the following result:
\begin{lemma}\label{dlam-converge} Let $\omega_i$, $w(t,x)$ be as in Assumption 2. The following limits hold as $N\to\infty$ and then $\epsilon\to0$:
\beq\label{dlam-barw}\max_{\argidxsetx i}\l|\omega_i-\E \ol{\bf w}_{\idxsetx }\r|\to0,\quad \max_{\argidxsetx i}|\omega_i-\w(t,x)|\to0\eeq  
\end{lemma}
\begin{remark} The second limit in~\eqref{dlam-barw} implies \eqref{Ep} for the $\omega_i$, i.e. we have $\omega_{Nx+R}\to\w(t,x)$ for any $R$ fixed.
\end{remark}
\begin{proof}We take the window average over $\argidxsetx i$ on both sides of~\eqref{EW-delam}, then take advantage of telescoping and the boundedness of $u_o$ on $\R$. This gives
\beqs\label{dlam-II}\l|\E\ol {\bf w}_\idxsetx-\bardlamloc\r|\leq  2\|u_o\|_\infty/(2N\epsilon+1).\eeqs
Moreover, Assumption 2.2 gives
\beq\label{dlam-I}\max_{\argidxsetx j}|\omega_j - \bardlamloc|\leq \l(2\epsilon N+1\r)\max_{i\in\idxsetx}|\omega_{i+1}-\omega_i| \leq  \l(2\epsilon+\frac 1N\r)C.\eeq Combining~\eqref{dlam-I} and~\eqref{dlam-II} gives the first limit in the lemma. The second limit follows directly from the first and Assumption 2.3. 
\end{proof}
\begin{corollary}\label{corr-omega-Ew}
The $\lambda_i\bmod 1$ are roughly varying near any $x$ such that $\omega(t,x)\notin\Z$, and $\E[ w_i]$ is roughly varying near any $x$ such that $\omega(t,x)\bmod 1\neq 0,1/2$. 
\end{corollary}
\begin{proof}
To see that $\lambda_i\bmod1$ are roughly varying, note that using the second limit in~\eqref{dlam-barw} we have
$$\max_{|i-Nx|\leq R}|(\lambda_{i}-\lambda_{i-1})\bmod 1| = \max_{|i-Nx|\leq R}|\omega_i\bmod 1| \to \omega(t,x)\bmod 1\neq 0.$$ The proof that $\E[ w_i]$ varies roughly near any $x$ such that $\w:=\w(t,x)\bmod 1\neq 0,1/2$, relies on the following fact: for any $v\in (0,1)$, the function $c\mapsto f_v(c)=u_o(c+v)-2u_o(c) + u_o(c-v)$ takes the value zero only when $c=0$ and $c=1/2$. Now, using~\eqref{EW-delam}, we have
\beqs\label{w-u-omega}
\max_{|i-Nx|\leq R}\l|\E w_{i+1} - \E w_i\r| \geq & \max_{|i-Nx|\leq R}|u_o(\lambda_{i}+\omega_{i+1}) -2u_o(\lambda_i) + u_o(\lambda_{i}-\omega_{i})|\\& - \max_{|i-Nx|\leq R}|\omega_{i+1} - \omega_{i}|\eeqs The term on the second line goes to zero, so it remains to show the term on the first line on the righthand side does not go to zero. Consider two cases. First, assume that $\lambda_{\lfloor Nx\rfloor}\bmod 1$ either does not converge at all, or converges to $c\neq0,1/2$. In either case, we can use the compactness of $[0,1]$ to extract a subsequence converging to $c\neq 0,1/2$. Without loss of generality, assume the entire sequence $\lambda_{\lfloor Nx\rfloor}\bmod 1$ converges to $c$. Using this and the convergence of $\omega_i$ to $\w$ for $i=\lfloor Nx\rfloor$ and $i=\lfloor Nx\rfloor +1$, we get 
\beqs
\lim_{N\to\infty}|u_o&(\lambda_{\lfloor Nx\rfloor}+\omega_{\lfloor Nx\rfloor +1}) -2u_o(\lambda_{\lfloor Nx\rfloor}) + u_o(\lambda_{\lfloor Nx\rfloor}-\omega_{\lfloor Nx\rfloor})| \\&= |u_o(c+\w)-2u_o(c) + u_o(c-\w)|=f_\w(c).\eeqs Since $w\neq 0$ and $c\neq 0,1/2$, we can use the fact about the zeros of $f_v(c)$ for $v=\w$, to conclude that $f_\w(c)\neq 0$. Hence the term on the right in the first line of~\eqref{w-u-omega} does not converge to zero, so $\max_{|i-Nx|\leq R}\l|\E w_{i+1} - \E w_i\r|$ cannot converge to zero. 
In the second case, we assume $\lambda_{\lfloor Nx\rfloor}\bmod 1$ converges to $c=0$ or $c=1/2$. It then follows that $c'=c+\w\bmod 1$ is the limit of $\lambda_{\lfloor Nx\rfloor + 1}\bmod 1.$ Since $c=0$ or $c=1/2$ and $\w$ is neither 0 nor 1/2, it follows that $c'\bmod 1\neq 0,1/2$. Applying the same argument as above but now to the second order finite difference centered at $\lfloor Nx+1\rfloor$ instead of $\lfloor Nx\rfloor$, we have
\beqs
\lim_{N\to\infty}
|u_o&(\lambda_{\lfloor Nx\rfloor+1}+\omega_{\lfloor Nx\rfloor +2}) -2u_o(\lambda_{\lfloor Nx\rfloor+1}) + u_o(\lambda_{\lfloor Nx\rfloor+1}-\omega_{\lfloor Nx\rfloor+1})| \\&= |u_o(c'+\w)-2u_o(c') + u_o(c'-\w)|=f_\w(c')\neq 0.\eeqs
\end{proof}
\noin Let us now explain the three observations made at the end of Section~\ref{subsec:qual}. \\

\noin\textbf{1) The rate expectations $\bf\E [ r( w_i)]$ vary smoothly.} Recall that $\mu_K[\omega,\lambda](r) = e^{-2K\omega}$. Thus, $\E[r( w_i)]  = e^{-2K\omega_i}$, which is smoothly varying since it is a continuous function of the smoothly varying $\omega_i$. This is a rare property. Indeed, for most other observables $f$, the profile ($\E f( w_i))_{i=1}^N$ is likely to be roughly varying since the expectations are functions of the roughly varying $\lambda_i\bmod1$. 
\begin{remark}\label{remark:tautology}The fact that the $\E[r( w_i)]$ profile varies smoothly as a result of $\omega_i$ varying smoothly is a bit of a tautology. This is because our numerical verification of Property 2 regarding the smoothness of $\omega_i$ used the rate expectation formula~\eqref{r-exp}. In other words, in the righthand panel of Figure~\ref{fig:prop}, we computed $\omega_i$ as $\omega_i=\log(\E[r( w_i)])/(-2K)$. Explaining the smoothness of $\omega_i$ is an open question. \end{remark}

\noin\textbf{2) The expectations $\bf\E w_i$ vary roughly.} This was shown in the corollary using the fact that $\omega_i = \lambda_i-\lambda_{i-1}$ converge to $\w(t,x)$ in the sense of \eqref{Ep}. We note that if instead $\E w_i$ converged to $\w(t,x)$ in the sense of \eqref{Ep} (and were therefore smoothly varying), then the $\omega_i$ would vary roughly near any $x$ for which $\w(t,x)\bmod1\neq0,0.5$. Indeed, let $u_i = u_D(\lambda_i)$, so that $\lambda_i = \lambda_D(u_i)$. Write $\lambda_D(u) = u+\lambda_o(u)$, where $\lambda_o$ has the same properties as $u_o$. In particular, $\lambda_o$ is periodic, smooth, and hence bounded. We then have
\beq\label{omeg-thru-w}\omega_i = \lambda_i - \lambda_{i-1} = \E[ w_i] + \lambda_o(u_i) - \lambda_o(u_{i}-\E[ w_i]),\eeq since $\E[ w_i] = u_i - u_{i-1}$. From here, we obtain an estimate exactly analogous to~\eqref{w-u-omega}. The proof of the roughness of $\omega_i$ follows by swapping $\E[ w_i]$ and $\omega_i$ in the proof of Corollary~\ref{corr-omega-Ew}. In the rough scaling PDE derivation in~\cite{mw-krug}, the assumption $\Law(\zN)= \otimes_i\rho_K[\lambda_i]$ was used, so that $\lambda_i = \lambda_D(\E z_i)$.  It was also implicitly assumed that $\E z_i = Nh_x(t,i/N)+ o(1)$, where $h$ is the rough scaling limit of $\hN$. This implies $\E w_i = h_{xx}(t,i/N) + o(1)$, and hence $\E w_i$ is smoothly varying, implying $\omega_i = \lambda_i -\lambda_{i-1}$ is roughly varying. But in fact, we know the opposite to be true.\\

\noin\textbf{3) The $\bf(i/N,\E w_i)$ cloud narrows to a smooth curve near integer values of the range.} There is a straightforward explanation for this phenomenon. Let $\w=\w(t,x)$ be an integer, and recall that $\omega_{Nx+j}\approx \w$ for finite $j$ and $N$ large. We then have $\lambda_{Nx+j}\approx \lambda_{Nx} + j\omega \equiv \lambda_{Nx}\mod 1$. Hence the $u_o(\lambda_{Nx+j})-u_o(\lambda_{Nx+j-1})$ contribution to $\E w_{Nx+j}$ is smoothly varying for $j$ small (and the $\omega_i$ contribution is always smoothly varying). We state this observation more formally in the following proposition, and give a proof in Appendix~\ref{app:proofs}.
\begin{proposition}
If $\w(t,x)\in\Z$, then $\l(\E w_j\r)_{\argidxsetx j}$ is smoothly varying in a \emph{mesoscopic} neighborhood of $x$, meaning $$\ol{\lim_{\epsilon\to0}}\;\ol{\lim_{N\to\infty}}\max_{\argidxsetx j}|\E w_j -\E w_{j-1}| = 0.$$  
\end{proposition}

\subsection{Why~\eqref{Ef} holds: mesoscopic averaging of marginals}\label{subsec:marg-av} In this section, we will give a rigorous explanation for why the limit~\eqref{Ef} holds, which we repeat here for convenience.
\beqn\label{hatf-exist}\text{For all ``suitable" }f\text{ there exists }\hat f\text{ such that }\E \bar{f}( {\bf w}_\loc)\stackrel{N,\epsilon}{\approx} \hat f(\E\ol{\bf w}_\loc).\tag{Ef}\eeqn Let us give an informal overview first. Recall that we are assuming $\Law(w_i)= \mu_K[\omega_i, \lambda_i\bmod 1]$ exactly, and from Lemma~\ref{dlam-converge} we know that $\omega_i$ nears $\E\ol{\bf w}_\loc$ uniformly over $i\in\loc$. As such, we can write $$\Law(w_i)\approx\mu_K[\E\ol{\bf w}_\loc, \lambda_i\bmod 1].$$ We then have
\beqs\label{Ef-intro}
\E \bar{f}( {\bf w}_\loc)&=\frac{1}{2N\epsilon}\sum_{i\in\loc}\E f(w_i) \\
&\approx \frac{1}{2N\epsilon}\sum_{i\in\loc}\mu_K[\E\ol{\bf w}_\loc, \lambda_i\bmod 1](f)
\eeqs
We need to write the expression on the second line as a function of $\E\ol{\bf w}_\loc$ only, but it would seem that the $\lambda_i\bmod 1$ prevent us from doing so. However, it turns out that the empirical measure formed by the points $\lambda_i\bmod 1$, $i\in\loc$, converges to the uniform distribution on $\unit$ as $\Nepslim$! \textit{As a result, the $\lambda_i\bmod1$ ``integrate out", giving us our desired function $\hat f$}:
\beqs\label{Ef-intro-II}
\E \bar{f}( {\bf w}_\loc)&\approx \int_0^1\mu_K[\E\ol{\bf w}_\loc, \lambda](f)d\lambda \\
&= \muinfty[\E\ol{\bf w}_\loc](f)=\hat f(\E\ol{\bf w}_\loc).
\eeqs
Here, we have defined $\muinfty[\omega]$ as the measure given by the integral of $\mu_K[\omega,\lambda]$ over $\lambda\in\unit$, and $\hat f(\omega):=\muinfty[\omega](f)$. 
An alternative way to state this argument is to work directly with measures. Note that we can write the righthand side of~\eqref{Ef-intro} as the expectation of $f$ with respect to the averaged measure $\sum_{i\in\loc}\Law(w_i)/2N\epsilon$. Therefore,~\eqref{Ef-intro-II} is true because
\beq\label{rough-fundamental}
\frac1{2N\epsilon}\sum_{i\in\loc}\Law(w_i)\approx \muinfty[\E\ol{\bf w}_\loc].\eeq By linearity of expectation, the mean of the measure on the left in~\eqref{rough-fundamental} is $\E\ol{\bf w}_\loc$. Let us check that $\muinfty[\E\ol{\bf w}_\loc]$ also has mean $\E\ol{\bf w}_\loc$, or in other words, that $\muinfty[\cdot]$ is mean-parameterized. (In principle, the two measures in~\eqref{rough-fundamental} could be close but have different means for $N<\infty$ and $\epsilon>0$.) To show that $m_1(\muinfty[\omega]) = \omega$, we will use the observation that
$$u_C(\omega) = \omega=\int_{\omega-1/2}^{\omega+1/2}u_D(\lambda)d\lambda,$$ a consequence of $u_o$ being odd about 1/2 and periodic. This equation says that we can recover the mean of a continuous normal distribution $\mathcal N(\omega, \sigma^2)$ by integrating the mean of the discretely supported distributions $\mathcal N_D(\lambda, \sigma^2)$ over $\lambda\in [\omega-1/2,\omega+1/2]$. Using this result, we have
 \beqs \label{mu-infty-mean}
m_1(\muinfty[\omega]) = \int_{0}^1m_1(\mu_K[\omega, \lambda])d\lambda&=\int_0^1\l[u_D(\lambda) - u_D(\lambda-\omega)\r]d\lambda\\
&= u_C(1/2) - u_C(1/2-\omega)= \omega.
\eeqs
Therefore, $\muinfty[\cdot]$ is indeed mean-parameterized. Recall the discussion in the beginning of Section~\ref{smooth-LE-def}. We considered the prototypical smooth LE state $\vN$,  in which $\Law(v_i)=\mu[\E v_i]$ for some mean-parameterized family $\mu[\cdot]$. In this case, $\E f(v_i) = \mu[\E v_i](f) = \hat f(\E v_i)$, where $\hat f(v) = \mu[v](f)$. Comparing to~\eqref{rough-fundamental}, we can now summarize a key difference between prototypical smooth LE states and the rough LE state of $\wN$:\\

For smooth LE states $\vN$, $\Law(v_i)\approx \mu[\E v_i]$ for some mean-parameterized family $\mu[\cdot]$, and therefore \eqref{Efp} holds. In our rough LE state $\wN$, $\Law(w_i)$ are \emph{not} mean-parameterized, but their mesoscopic averages \emph{are}: $(1/2N\epsilon)\sum_{i\in\loc}\Law(w_i)\approx\mu[\E\ol{\bf w}_\loc]$ for a mean-parameterized family $\muinfty[\cdot]$ and therefore,~\eqref{Ef} holds. \\

Before stating our result formally, we confirm numerically that $\hat f(\omega)=\muinfty[\omega](f)$ is the correct function mapping $\E\ol{\bf w}_\loc$ to $\E \bar f({\bf w}_\loc)$ in the $N\to\infty$, $\epsilon\to0$ limit. 
See Figure~\ref{fig:Ef-rough-fhat} for confirmation.
\begin{figure}
\centering
\includegraphics[width=\textwidth]{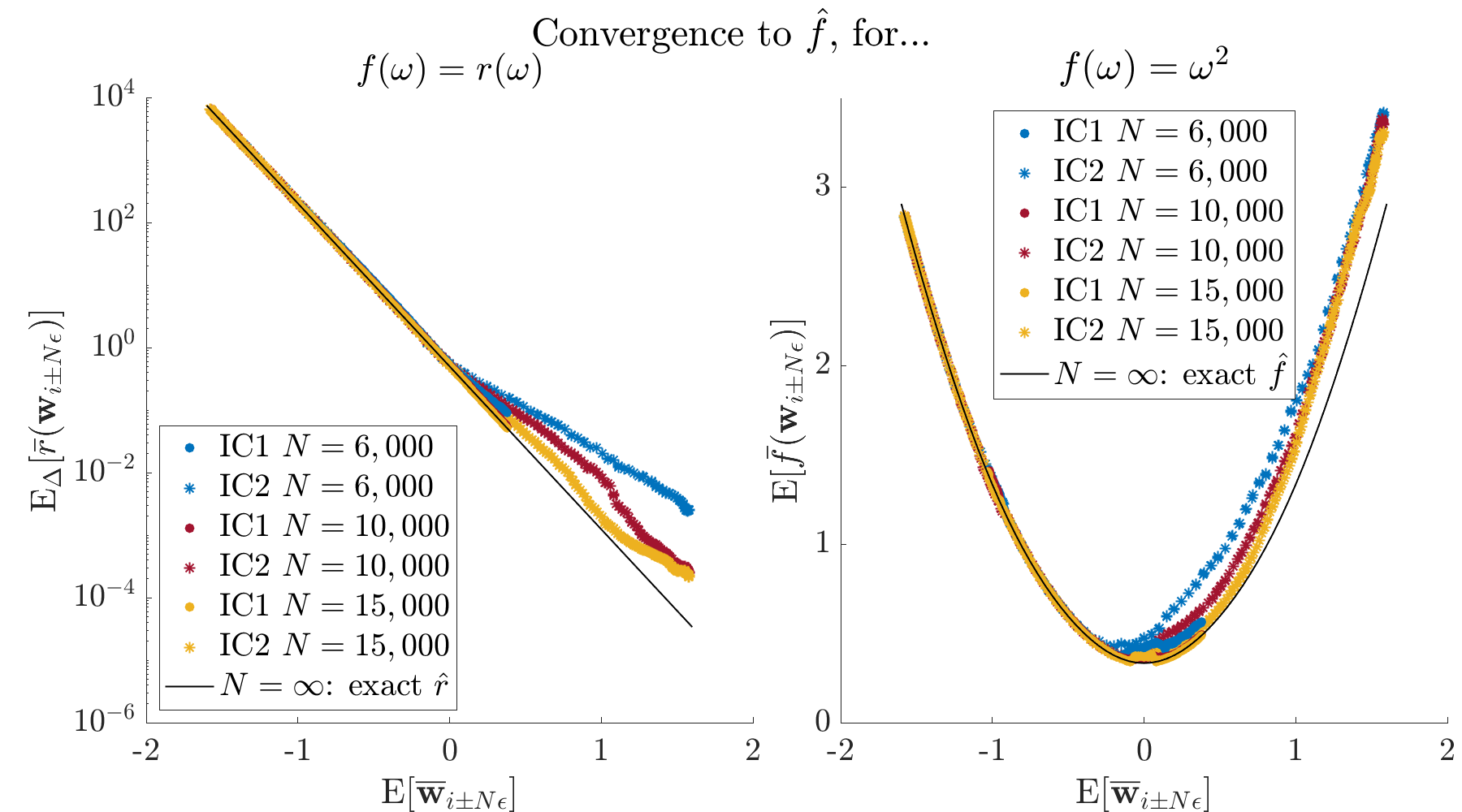}
\caption{These plots are the same as in Figure~\ref{fig:Ef-rough}, except that we now additionally plot the predicted curve $(\omega, \hat f(\omega))$. We see that $(\E\ol{\bf w}_{i\pm N\epsilon},\; \E\bar f({\bf w}_{i\pm N\epsilon}))$  converge to the curve as $N$ increases. The function $\hat f$ is given by $\hat f(\omega)=\muinfty[\omega](f)$, where $\muinfty[\omega]$ is defined in~\eqref{muinfty-def}.}
\vspace{-0.4cm}
\label{fig:Ef-rough-fhat}
\end{figure}

\begin{proposition}\label{prop:aud} Define the pmf $\muinfty[\omega]$ by
\beq\label{muinfty-def}\muinfty[\omega](n) = \int_0^1\mu_K[\omega, \lambda](n)d\lambda = \e\l(-\frac K2(n-\omega)^2\r)\int_0^1Q(n,\omega,\lambda)d\lambda.\eeq If $\w(t,x)$ is irrational, then
\beq\label{aud-pts-prob-bd-foreal}
\bigg|\frac{1}{2N\epsilon}\sum_{i\in\loc}\mu_K[\omega_i,\lambda_i](n) - \muinfty\l[\E\ol{\bf w}_\loc\r](n)\bigg| \leq C_{N,\epsilon}\e\l(-Kn^2/3\r)\eeq for all $n\in\Z$, where $C_{N,\epsilon} \to 0$ as $\Nepslim$. Letting $p(n)\propto\e(-Kn^2/3)$ be a pmf on $\Z$, it follows that~\eqref{Ef} holds for any $f\in L^1(p)$, with $\hat f(\omega) = \muinfty[\omega](f)$.
\end{proposition}
Let us give an outline for the proof of Proposition~\ref{prop:aud}. First, we justify replacing both $\omega_i$ and $\E\bar{\bf w}_\loc$ with $\w:=\w(t,x)$ in~\eqref{aud-pts-prob-bd-foreal}. This is justified since the $\omega_i$ converge to $w$ by Lemma~\ref{dlam-converge}, and $\E\bar{\bf w}_\loc$ converges to $w$ by~\eqref{E}. It then remains to show
\begin{lemma}\label{lma:aud}Let $\lambda_i$, $i\in \idxsetx$ be any numbers such that $\omega_i=\lambda_i-\lambda_{i-1}$ satisfy the smoothness Assumption 2 and converge uniformly to $\w=\w(t,x)$, as in Lemma~\ref{dlam-converge}. If $\w$ is irrational, then 
\beq\label{aud-pts-prob-bd}
\l|\frac{1}{2N\epsilon}\sum_{i\in\idxsetx}\mu_K[\w, \lambda_i](n) - \int_{0}^{1}\mu_K[\w, \lambda](n)d\lambda\r| \leq C_{N,\epsilon}\e\l(-Kn^2/3\r)\eeq for all $n\in\Z$, where $C_{N,\epsilon} \to 0$ as $\Nepslim$.  
\end{lemma}
The rigorous proof of Lemma~\ref{lma:aud} and Proposition~\ref{prop:aud} is given in Appendix~\ref{app:proofs}. We briefly sketch the proof of Lemma~\ref{lma:aud} here, to explain the irrationality constraint. To prove the sum in~\eqref{aud-pts-prob-bd} converges to the integral, we show that the points $\lambda_i\bmod 1$ are asymptotically equidistributed in the unit interval. In other words, we show that the empirical measure $P^{N,\epsilon}$ given by the average of the point masses at $\lambda_i\bmod1$ converges to the Lebesgue measure on the unit interval. To do so, we essentially make the approximation $\lambda_{Nx+i} \approx \lambda_{Nx} + i\w$. Clearly, $\w$ must be irrational for $\lambda_i\bmod1$ to be asymptotically uniformly distributed. Since there is an error incurred by this approximation, we require a quantitative bound on the distance between $P^{N,\epsilon}$ and Lebesgue measure. For this, we use the Erd\H{o}s-Tur\'{a}n inequality~\cite{erdos}, which says that for a measure $\nu$ on the unit interval, we have
\beq\label{erdos-turan}
\sup_{(a,b)\subset [0,1]}|\nu((a,b)) -(b-a)| \leq C\l(\frac1n + \sum_{m=1}^n\frac{|\hat\nu(m)|}{m}\r)
\eeq
for an absolute constant $C$ and any $n=1,2,\dots$, where $\hat\nu(m) = \int \e\l(2\pi i mx\r)\nu(dx).$

Let us consider the limitations due to the constraint $\w(t,x)\notin\mathbb Q$. Consider the set of points $x$ for which $\w(t,x)$ is rational. For this set to have positive Lebesgue measure, there must exist $q\in\mathbb Q$ such that $\w(t,x)=q$ for uncountably many $x$. The case in which these points do not form an interval is pathological and unlikely to appear in practice. Thus, the only kind of non-pathological functions $\w(t,x)$ for which this set has positive Lebesgue measure, are the functions for which $\w(t,x)\equiv q\in\mathbb Q$ in some interval. The analysis of this case is beyond the scope of the paper.


\subsection{Discussion}\label{subsec:arr-discush}
We conclude with a few remarks on $\wN$'s rough LE.

\noin\textbf{Why $\wN$'s LE is not smooth.} The contributing factors are (1) the rough scaling, (2) the discreteness of the microscopic system, and (3) that the distributions $\Law( z_{i-1}, z_i)$ are Gibbs measures. The contribution of these three factors is evident in the equation 
\beqs\label{Etildew-discush}
\E w_i &= u_D(\lambda_i) - u_D(\lambda_{i-1}) \\&=\big[\lambda_i - \lambda_{i-1}\big]+ \bigg[u_o(\lambda_i\bmod 1) - u_o(\lambda_{i-1}\bmod 1)\bigg] \\&= \omega_i+u_o(\lambda_i\bmod 1) - u_o(\lambda_{i-1}\bmod 1).\eeqs Due to (1) the rough scaling, the $\lambda_i$ are $O(N)$. Due to (2) the discreteness, the expectations $\E w_i$ depend on $\lambda_i$ through the \emph{nonlinear} function $u_D(\lambda)=u_C(\lambda) + u_o(\lambda)$ rather than through the linear $u_C(\lambda)=\lambda$. 
As a result of the nonlinearity, $\E w_i$ depends not only on the smoothly varying finite differences $\omega_i=\lambda_i-\lambda_{i-1}$, but also on the roughly varying $\lambda_i$ themselves. This makes the LE state of $ \wN$ rough. The specific form of the nonlinearity comes from (3) the Gibbs distribution, but this is perhaps less important.

\noin\textbf{Why Law$(\wN)$ is a rough LE, and an LE at all.} The above point explains why the $\E w_i$ are roughly varying, implying that the LE state cannot be smooth. But to show $\Law(\wN)$ is rough, we needed to show~\eqref{Ef} is satisfied. We did this using (1) $\Law(w_i)$ belongs to a two-parameter family with parameters $\omega_i$, $\lambda_i$, (2) the $\omega_i$ converge uniformly over $i\in\idxsetx$ to $\w(t,x)$ and (3) the empirical distribution defined by the points $\lambda_i\bmod1$, $i\in\idxsetx$ converges to Lebesgue measure. To prove (3), we used that $\lambda_i-\lambda_{i-1}=\omega_i$. However, for~\eqref{Ef} to hold, any other asymptotically uniformly distributed points $\lambda_i\bmod1$ would do. This is an important point because it allows us to think of $\Law(\wN)$ not as derived from $\Law(\zN)$ (from which the $\lambda_i$ originated), but as a self-standing distribution with certain properties giving rise to a rough LE. This viewpoint can help us identify rough LE states in other contexts.

But does $\Law(\wN)$ constitute an LE state at all? Applying Definition~\ref{def:LE} of an LE state, the joint distribution $\Law(\{w_i\}_{i\in\idxsetx})$ should be asymptotically fully determined, through some parameterization, by $w(x)$ for a continuous function $w:\unit\to\R$. This is not true for our process $\wN$, since we need all the numbers $\lambda_i$ to specify the full joint distribution $\Law(\{w_i\}_{i\in\idxsetx})$. However, we have seen that the smoothly varying $\omega_i$ are the relevant parameters, since the $\lambda_i$ integrate out upon mesoscopic averaging. Since $\omega_i\approx w(t,x)$ for $i\in\idxsetx$, we can still argue that $\Law(\{w_i\}_{i\in\idxsetx})$ is in some sense fully determined by $w(t,x)$.\\

\noindent\textbf{Impact of Rough LE on the PDE.} 
Recall that the PDE governing the hydrodynamic limit $w$ in the rough scaling regime is 
\beq\label{rough-w-discush}\partial_t w = \partial_{xxxx}\e(-2Kw).\eeq For simplicity, we will take $2K=1$. We call~\eqref{rough-w-discush} the ``rough $w$ PDE".  

We claim that the rough LE state of $\wN$ causes $u_D$ to ``average out", and as a result, the argument to the exponential in~\eqref{rough-w-discush} involves $w=\lambda_C(w)=u_C^{-1}(w)$ rather than $\lambda_D(w)=u_D^{-1}(w)$. Our argument will also serve as a correction to the argument given in~\cite{mw-krug} for the appearance of $\lambda_C$. 

To make our point, it will be useful to compare a PDE obtained in the rough scaling limit to a corresponding PDE obtained in the smooth-scaling limit. We cannot compare $w$ PDEs since the smooth PDE is ill-defined for $w=z_x$. So instead, let us integrate the rough $w$ PDE once with respect to $x$, and then compare the rough $z$ PDE to the smooth $z$ PDE. We get
\beq\label{rough-discush}\partial_tz = \partial_{xxx}\e(-\lambda_C(z_x))\eeq
for the rough $z$ PDE, where we have written $\lambda_C(w)$ for $w$ (and recall $2K=1$.) Meanwhile, the smooth $z$ PDE can be written in the form
\beqs\label{smooth-discush}\partial_t z= \partial_{xxx}\frac{d}{d\gamma}\e\l(-\gamma\partial_x\lambda_D(z)\r)\big\vert_{\gamma=0}.\eeqs

The fact that the exponential gets differentiated at zero in~\eqref{smooth-discush} has to do with the smooth scaling. For us, the more important difference between the two PDEs is that the smooth $z$ PDE contains $\lambda_D$ while the rough $z$ PDE contains $\lambda_C$. The function $\lambda_D$ reflects the discreteness of the microscopic system, so its presence in the PDE~\eqref{smooth-discush} governing the non-discrete (i.e. $\R$-valued) limit $z(t,x)$ is an interesting feature of the smooth scaling regime.

Now, both the smooth and rough $z$ PDEs follow from the microscopic dynamics via $\LN\pi_i$. We have $(\LN\pi_i)(\mbf z) = r(w_{i+1})-3r(w_i) + 3r(w_{i-1})+r(w_{i-2})$, where $w_i=z_i-z_{i-1}$. We then move the three finite differences onto a test function, and replace $r(w_i)$ with the average of $r(w_j)$, $j\in i\pm N\epsilon$. This average should be approximately deterministic, so we replace it with its expectation. E.g. for $i=\lfloor Nx\rfloor$, we consider
\beq\label{rate-expect-avg}\E\bigg[\frac{1}{2N\epsilon}\sum_{i\in \idxsetx }r(w_i)\bigg].\eeq
We must then express this average of expectations as a function of the converging microscopic process, and from here the PDE essentially follows. In both the smooth and rough scaling regime, one possible framework for doing so, which centers around $\Law(\zN)$, is the following:
\begin{enumerate}
\item Introduce parameters $\lambda_i$ such that 
$$\mathbb P(\zN=\mbf z) \approx \prod_{i=1}^Ne^{-Kz_i^2+2K\lambda_iz_i}/\calZ_{\bml^N}.$$ This is the Gibbs product measure with parameters $\lambda_i$, $i=1,\dots, N$. 
\item Express~\eqref{rate-expect-avg} in terms of the $\lambda_i$: using the rate expectation formula~\eqref{r-exp}, we get
\beq\label{discush:rate-av}\E\bigg[\frac{1}{2N\epsilon}\sum_{i\in \idxsetx }r( z_i-z_{i-1})\bigg]\approx \frac{1}{2N\epsilon}\sum_{i\in \idxsetx }\e(\lambda_{i-1}- \lambda_i).\eeq 
\item \textbf{Write~\eqref{discush:rate-av} in terms of the appropriate, \emph{converging} microscopic process.}
\end{enumerate}
We now review the argument in~\cite{mw-krug} for why $\lambda_C$ appears in the rough $z$ PDE. 

Namely, Marzuola and Weare assume $\lambda_i=\lambda_D(\E z_i)$, and $\E z_i\approx Nz(i/N)$, where $z$ is the hydrodynamic limit of $\zN$ under amplitude scaling $N^1$ (this corresponds to taking the limit of $\hN$ in the rough scaling regime). We have omitted the $t$ variable in $z$ for brevity.
They then observe that $\lim_{\kappa\to\infty}\kappa^{-1}\lambda_D(\kappa u) = \lambda_C(u) = u$. Using the decomposition $\lambda_D =\mathrm{id.} + \lambda_o$ introduced in this paper, this observation is equivalent to saying $\kappa^{-1}\lambda_D(\kappa u)= u + \kappa^{-1}\lambda_o(\kappa u) \to u$ as $\kappa\to\infty$. One then has \beqs\label{mw-lambda}
\lambda_i \approx \lambda_D(\E z_i)\approx \lambda_D\l(Nz\l(i/N\r)\r) \approx Nz\l(i/N\r),\eeqs where in the last approximation we have discarded $\lambda_o(Nz(i/N))$. Substituting~\eqref{mw-lambda} into~\eqref{discush:rate-av} gives $\e\l(-\partial_xz(x)\r)$ for the average rate expectation. 
Thus, Marzuola and Weare claim that the reason $\lambda_C$ arises in the rough PDE is that $\lambda_D(\kappa u)\approx \kappa\lambda_C(u)=\kappa u$ for large $\kappa$, and that in the rough scaling regime, the argument to $\lambda_D$ is large (in contrast to the smooth scaling regime, where the argument to $\lambda_D$ is $O(1)$). 

This sequence of approximations is flawed for two reasons, both of which have to do with the fact that we cannot discard order $O(1)$ terms when estimating $\lambda_i$, since we ultimately need to estimate the finite difference $\lambda_i-\lambda_{i-1}$. First, as we explained in Section~\ref{subsec:omega-tilde-w}, the approximation $\E z_i = Nz(i/N) + o(1)$ would imply that $\E w_i$ are smoothly varying and $\omega_i=\lambda_i-\lambda_{i-1}$ are roughly varying, when in fact the opposite is true. Therefore, the second approximation in~\eqref{mw-lambda} is off by order $O(1)$. The third approximation is also off by $O(1)$, since it discards the non-vanishing $\lambda_o(Nz(i/N))$. 
 
A more fundamental reason for why the argument is flawed is that the expression 
\beq\label{inferno}\lambda_i - \lambda_{i-1} = \lambda_D(\E z_i) - \lambda_D(\E z_{i-1})\eeq is not written in terms of a converging process. Indeed, while we expect $N^{-1}\E z_i\stackrel{N}{\approx}z(i/N)$, the expression~\eqref{inferno} depends on the non-converging $\E z_{i}\bmod 1$. But the third step above requires us to write the average of $\e(\lambda_{i-1}-\lambda_i)$ in terms of a converging process, and it is not at all clear how to do this using~\eqref{inferno}. This is where the $\wN$ rough LE comes in. But before we continue, note that in the smooth scaling regime~\eqref{inferno} \emph{is} written in terms of a converging process: from numerical simulations (see Figure~\ref{fig:ASEP-crystal}), we should have $\E z_i\stackrel{N}{\approx} z(i/N)$ and $N(\E z_i - \E z_{i-1})\stackrel{N}{\approx}z_x(i/N).$ Using an extra power of $N$ multiplying the rate expectation average~\eqref{discush:rate-av}, we get the smooth $z$ PDE~\eqref{smooth-discush}, which involves $\lambda_D$. 

Returning to the rough scaling regime, another framework for deriving the PDE parallels 1-3 above, but instead uses the parameterization $\Law(w_i)\approx \mu_K[\omega_i,\lambda_i]$. Since we know the $\omega_i$ are smoothly varying, uniformly over $i\in\idxsetx$, let us for simplicity assume $\omega_i\equiv\omega$ for $i\in\idxsetx$. We then have
\beq\E\bigg[\frac{1}{2N\epsilon}\sum_{i\in \idxsetx }r( w_i)\bigg]\approx \e(-\omega).\eeq It remains to write $\omega$ in terms of a converging process. Of course, we already know from Lemma~\ref{dlam-converge} that $\omega\approx \E\ol{\bf w}_\loc$, which is our desired process converging to $w=z_x$ (we think of $z$ as the antiderivative of $w$). But we got to that lemma by assuming the Gibbs distribution on $\zN$. Using only the perspective of the $\wN$ distribution, and the insight that the $\lambda_i\bmod 1$ are asymptotically uniformly distributed, we have
$$
\frac1{2N\epsilon}\sum_{i\in\loc}\Law(w_i)\approx \frac1{2N\epsilon}\sum_{i\in\loc}\mu_K[\omega, \lambda_i]\approx \muinfty[\omega].$$ Thus, we should think of $\omega$ as the unique parameter such that $$\E\ol{\bf w}_\loc=m_1\bigg(\!\sum_{i\in\loc}\Law(w_i)/2N\epsilon\bigg) = m_1(\muinfty[\omega]).$$ Now, recall our computation~\eqref{mu-infty-mean} of $m_1(\muinfty[\omega])$, which finally explains our claim about $u_D$ averaging out:
\beq\label{mu-infty-mean-discush}m_1(\muinfty[\omega]) = \int_0^1u_D(\lambda) -\int_0^1 u_D(\lambda -\omega) =u_C(\omega)=\omega.\eeq Therefore, $\omega = u_c^{-1}(\E\ol{\bf w}_\loc) = \lambda_C(\E\ol{\bf w}_\loc)$. 

To summarize, we end up integrating $u_D$ over $\lambda\in[0,1]$ because the $\lambda_i\bmod 1$ are asymptotically uniformly distributed (and roughly varying). Therefore the rough LE state of $\wN$ led to the integration~\eqref{mu-infty-mean-discush}  which erased the discreteness in $u_D$, so we ended up not with $\lambda_D =u_D^{-1}$ but with $\lambda_C$. Interestingly enough, $ \wN$ owes its rough LE state to the discreteness of the microscopic system. And this rough LE state, in turn, caused the discreteness to be washed out in the continuum limit!

\section{Conclusion}\label{sec:conclude}
In this paper, we have discovered a new, \emph{rough} local equilibrium (LE) state, and studied its properties through the lens of the three key limits~\eqref{V},~\eqref{E}, and~\eqref{Ef}. These limits do hold under a rough LE,  but the stronger pointwise limits \eqref{Ep}, \eqref{Efp} do not. This is because mesoscopic window averaging is essential to smooth out the roughness of ensemble averages. Indeed, unlike in smooth LE states, most observable \emph{expectations} $\E[f( w_i)]$ vary roughly with $i$ (that $f( w_i)$ vary roughly before taking expectations is unsurprising). In Section~\ref{sec:theory}, we uncovered the mechanism underlying the convergence~\eqref{Ef}: mesoscopic averages of the single site marginals $\Law( w_i)$ belong to a mean-parameterized measure family in the $N\to\infty$, $\epsilon\to0$ limit. Finally, we explained in Section~\ref{subsec:arr-discush} why $ \wN$'s rough LE state is a product of the rough scaling regime and discreteness. In turn, the rough LE led to an integration which erased the discreteness of $u_D$, leaving $\lambda_C$ rather than $\lambda_D =u_D^{-1}$ in the PDE. 

The rough LE state we have discovered is not isolated to the $\wN$ process studied here. In~\cite{metrop_pde}, we show that a crystal surface process with Metropolis-type jump rates also has a rough LE state. Specifically, we study $\vN$ given by a third order finite difference of $\hN$, which is scaled as $O(N^3)$. For the Metropolis rate process, this is the scaling regime leading to a nontrivial PDE with exponential nonlinearity. (See~\cite{gao2020} for an informal derivation of the rough scaling PDE limit of the Metropolis height process). Even though $\Law(\zN)$ is not local Gibbs, the LE state of $\vN$ is still rough. This reinforces the centrality of (1) the rough scaling and (2) the discreteness of the microscopic system in contributing to the rough LE state. Interestingly, the LE state of $\vN$ is qualitatively very similar to the LE state of $\wN$ for the Arrhenius process. This similarity could be related to the fact that the invariant measure for the $\zN$ processes are the same: the $\lambda_i\equiv 0$ standard Gibbs measure $\Phi(z)\propto \e\l(-KH(z)\r)$. However, this merits further investigation, which we leave for future work.

\appendix
\section{Local Equilibrium of Zero Range Process}\label{app:ZR-LE} This section summarizes results from Chapter 2.3 in~\cite{kipnisbook}. We stated in the main text that the local equilibrium measure of the zero range process is given by a product of mean-parameterized measures $\mu[\cdot]$, of the form $\Law( \vN)\approx \otimes_{i=1}^N\mu\l[v\l(t,i/N\r)\r].$ This is true, but it is slightly more convenient to parameterize these measures using another parameter which is not the mean. To that end, define a family $\{\nu[\phi]\mid 0\leq\phi<\infty\}$ of measure on $\{0,1,2,\dots\}$, via
$$\nu[\phi](n) = \frac{1}{Z(\phi)}\frac{\phi^n}{g(n)!},\quad \text{where}\quad g(n)!=\prod_{k=1}^ng(k),$$ and $g(0)!=1$ by convention. Here, $Z(\phi)$ is the normalization constant $Z(\phi) = \sum_{n=0}^\infty\phi^n/g(n)!$. We assume that the rates $g(n)$ are such that $Z(\phi)<\infty$ for all $\phi\geq0$, and that $\lim_{\phi\to\infty}Z(\phi)=\infty$. Define
$$v(\phi) = m_1(\nu[\phi]) = \sum_{n=0}^\infty n\frac{1}{Z(\phi)}\frac{\phi^n}{g(n)!}.$$ The assumptions on $Z(\phi)$ guarantee that $v(\phi)$ is strictly increasing and that $v:\R_+\to\R_+$ is bijective. Therefore, our desired mean-parameterized LE measure family is $\{\mu[v]\mid v\geq 0\}$, where $\mu[v]=\nu[\phi(v)]$ and $\phi=v^{-1}$. 

Recall that given a mean-parameterized measure family $\mu[\cdot]$ and an observable $f$, we define $\hat f(v) = \mu[v](f)$. Therefore, for the zero range LE, we compute $\hat f$ to be
\beq\label{hatf-ZR}
\hat f(v) = \sum_{n=0}^\infty f(n)\frac{1}{Z(\phi(v))}\frac{\phi(v)^n}{g(n)!}.
\eeq
We use Equation~\eqref{hatf-ZR} to compute the $(v,\hat f(v))$ curves shown in Figure~\ref{fig:Efprime}.

\section{Numerics}\label{app:num}
\subsection{Kinetic Monte Carlo}\label{app:KMC}
Each of the interacting particle systems discussed in this paper is a Markov jump process $\{\vNtild(t)\}_{t\geq0}$, the paths of which  are step functions, with $\vNtild(t) = \vN^k$ when $t\in [t_k, t_{k+1})$. Therefore, simulating the process in a time interval $[0, T(N)]$ amounts to drawing pairs $(\vNtild^k, t_k)$ according to the law of the process until $t_k$ first exceeds $T(N)$. We do so using the Kinetic Monte Carlo algorithm (KMC)~\cite{kmc}. 

To simulate processes $\vN(t)=\vNtild(N^\alpha t)$, $N=1,2,\dots$  we do the following. First, we fix a macroscopic initial profile $\v_0:\unit\to\R$ and a macroscopic time $t$. Next, we construct a sequence of distributions $\muN_0$, $N=1,2,\dots$ associated to $\v_0$ as in Definition~\ref{def:init}. For simplicity, we take $\muN_0$ such that $\E_{\muN_0}[ v_i(0)] = \v(i/N)$. For a given $N$, we draw $n$ initial profiles $\vN^{(k)}(0)$, $k=1,\dots, n$ independently from $\muN_0$. We then use KMC to generate independent processes $\vN^{(k)}$ corresponding to the initial profiles, evolving each of them forward until time $N^\alpha t$. To estimate expectations $\E[f( \vN(t))]$, we then take
\beq\label{app:usual-av} \E[f( \vN(t))]\approx \E^n[f( \vN(t))]:=\frac 1n\sum_{k=1}^n f(\vNtild^{(k)}(N^\alpha t)).\eeq

For the crystal processes simulated in this paper --- the slope process $\zN$ in the smooth scaling regime, and the curvature process $ \wN$ in the rough scaling regime --- it is simpler to evolve forward the associated height process. We therefore first fix an initial height profile $\h_0$, and draw $h_i(0) = \lfloor N^\beta \h_0(i/N)\rfloor + \xi_i$, where $\beta=1,2$ for the smooth and rough scaling regimes, respectively, and $\xi_i$ are independent Bernoullis. One can check that the induced distribution $\muN_0$ of initial slope or curvature profiles is associated to $h_0'$ or $h_0''$ as in Definition~\ref{def:init}, with $\beta=0$. We then evolve $\hN(t)$ forward, and recover the process of interest by taking one or two finite differences. 

%

\subsection{Time Averaging}\label{app:time-av}
The remainder of Appendix~\ref{app:num} pertains to numerical computations for the rough scaling curvature process $ \wN(t) = \wNtild(N^4t)$. As discussed in the main text, the spatially averaged rate observable $f( \wN)=\bar r({\bf w}_{i\pm N\epsilon})$ needed to check~\eqref{Ef} has a very large variance. To get much lower variance, and low bias, estimates of this expectation, we will show that for $\Delta$ chosen appropriately, we have
\beq\label{Delt-approx}\E[f( \wN(t))]\approx \E_\Delta[f( \wN(t))]:=\frac1\Delta\int_{I_t}\E[f( \wN(s))]ds,\eeq where $I_t$ is any length $\Delta$ interval containing $t$. We estimate the time averaged expectation using the significantly lower variance estimator below:
\beq\label{app:time-av-eq}\E_\Delta[f( \wN(t))]\approx \E^n_\Delta[f( \wN(t))]:=\frac 1n\sum_{k=1}^n \frac{1}{N^4\Delta}\int_{N^4I_t}f(\wNtild^{(k)}(s))ds.\eeq The time integrals for each path $\wN^{(k)}$ can be computed exactly because $ \wN(t)$ is a step function in time. We will show that there is a $\Delta$ such that~\eqref{Delt-approx} holds for each rate function $f( \wN) = r( w_i)$, $i=1,\dots, N$. This same $\Delta$ will work to estimate $\E\bar r({\bf w}_{\idxsetx})$. First, we choose $\Delta$ small enough that decreasing $\Delta$ further has no effect on the estimator $\E^n_\Delta[r( w_i)]$, except perhaps to increase its variance. This is shown in the left panel of Figure~\ref{fig:sample-v-time}, with $N=6000$. We take $\Delta = 10^{-14}$, and note that the estimators $\E^n_\Delta[r( w_i)]$, $\E^n_{\Delta/2}[r( w_i)]$, and $\E^n_{\Delta/4}[r( w_i)]$ are all extremely close to one another (with $n=1000$). Next, we compare $\E^n[r( w_i)]$ to $\E_\Delta[r( w_i)]$ for increasing $n$. This is shown in the right panel of Figure~\ref{fig:sample-v-time}. It is clear that with increasing $n$, the instantaneous-time rate expectation estimator $\E^n[r( w_i)]$ approaches $\E_\Delta[r( w_i)]$, although even for the largest $n=10^6$, $\E^n[r( w_i)]$ has not yet converged. Increasing $n$ further is computationally intractable. 
\begin{figure}
\centering
\vspace{-1.6cm}
\includegraphics[width=0.8\textwidth]{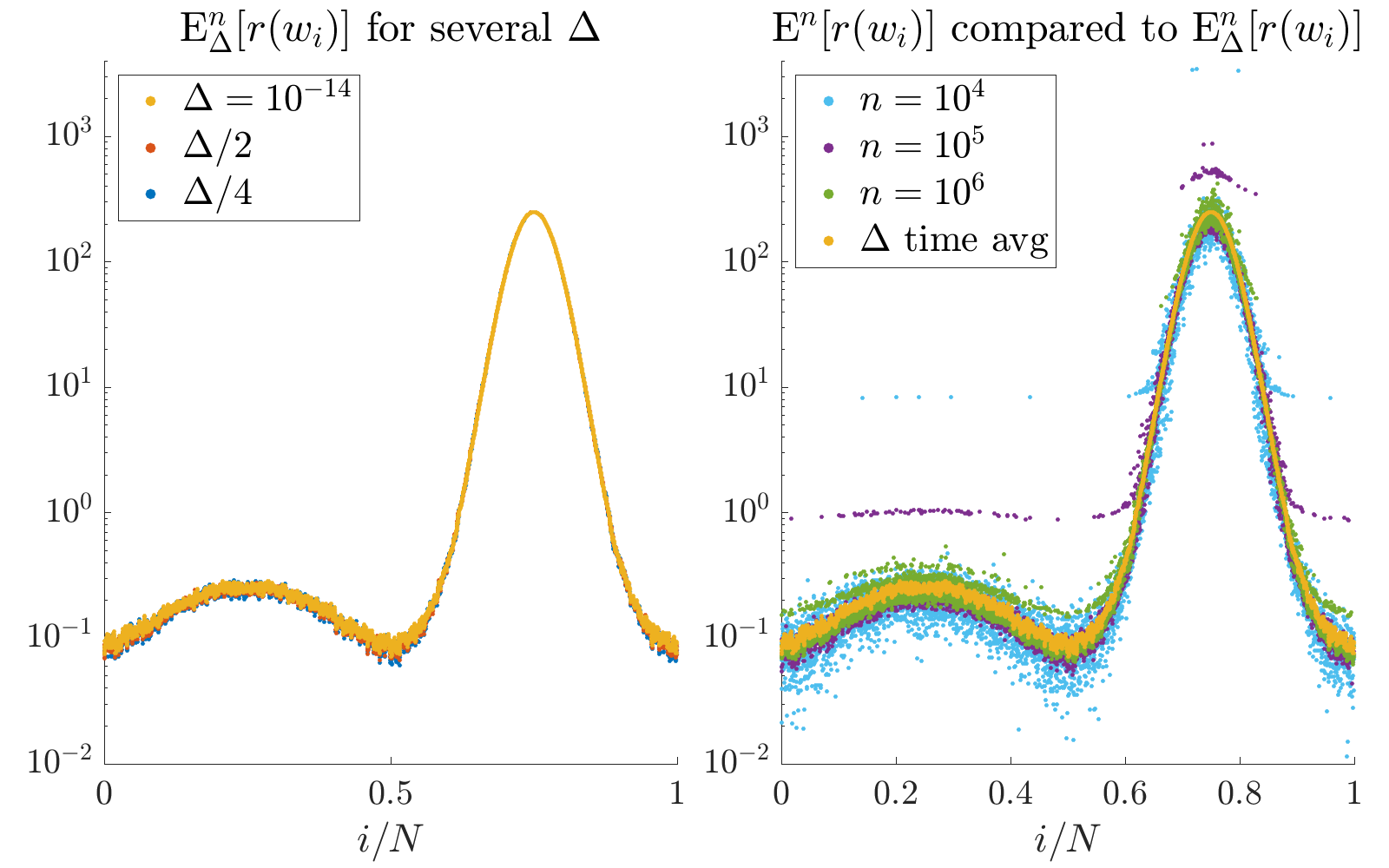}
\caption{Left: for a very small $\Delta$, we show that the effect of decreasing $\Delta$ on the estimator $\E^n_\Delta[r( w_i)]$ is negligible. Right: As $n$ increases, the instantaneous-time rate expectation estimator $\E^n[r( w_i)]$ approaches $\E_\Delta[r( w_i)]$.}
\vspace{-0.5cm}
\label{fig:sample-v-time}
\end{figure}
\subsection{Choice of $\epsilon(N)$ to verify~\eqref{E}}\label{app:epsilon-of-N}
Recall~\eqref{E}: $\frac{1}{2N\epsilon}\sum_{\argidxsetx i} \E w_i(t)\to\w(t,x)$ as $\Nepslim$.
\begin{figure}
\centering
\vspace{-0.3cm}
\includegraphics[width=0.85\textwidth]{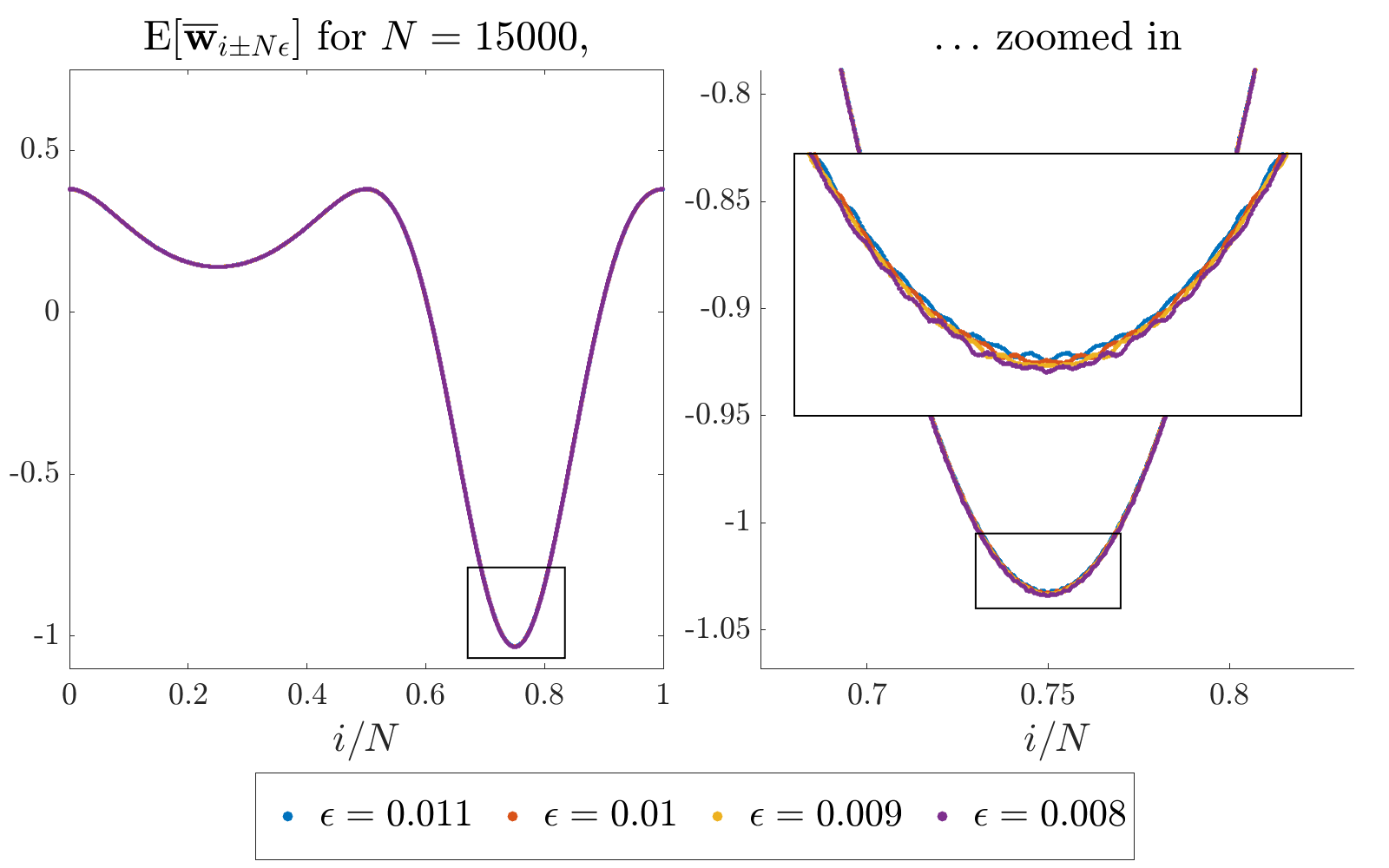}
\vspace{-0.25cm}
\caption{$\E[\overline{\bf w}_{i\pm N\epsilon}]$ for fixed $N$ and decreasing $\epsilon$, on three scales. From the perspective of the large and medium scales, $\E[\overline{\bf w}_{i\pm N\epsilon}]$ has nearly converged for $\epsilon\leq0.011$.}
\vspace{-0.8cm}
\label{fig:bar-omega-converge}
\end{figure} We use the following proxy for $N\to\infty,\epsilon\to0$ convergence. For each $N$, we choose a ``good'' $\epsilon(N)$: for $\epsilon>\epsilon(N)$, $\E[\overline{\bf w}_{i\pm N\epsilon}]$ is smooth but biased, and for $\epsilon<\epsilon(N)$, it is unbiased but rough. We then check that $\E[\overline{\bf w}_{N(x\pm\epsilon(N))}]$ is converging as $N\to\infty$. Figure~\ref{fig:bar-omega-converge} shows $\E[\ol{\bf w}_{i\pm N\epsilon}]$ for $N=15000$ and several values of $\epsilon$, on three scales. The largest scale is shown in the lefthand plot, with the medium scale region indicated by the rectangle on the bottom right. The medium and small scales are shown in the righthand plot. Relative to the variation in the range on the large and medium scales, $\E[\ol{\bf w}_{i\pm N\epsilon}]$ has nearly converged for $\epsilon\leq0.011$. But it has not converged on the small scale, which we use to pick $\epsilon(N)$. For $N=15000$ shown here, we take $\epsilon(N)=0.008$. The process is similar for other $N$'s. 

\subsection{Local Gibbs Computations}\label{app:loc-gibbs}
\begin{wrapfigure}{r}{0.4\textwidth}
\centering
\vspace{-0.6cm}
\includegraphics[width=0.4\textwidth]{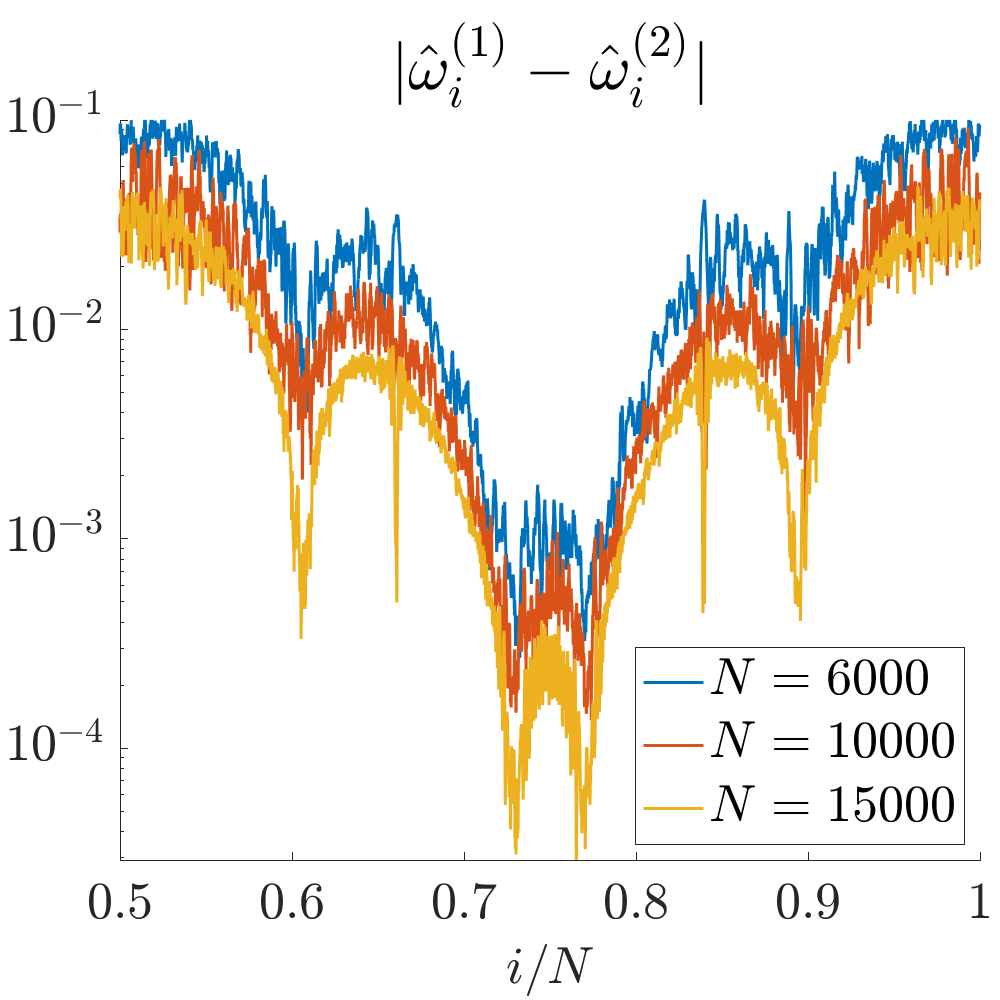}
\caption{$\hat\omega_i^{(1)}$, $\hat\omega_i^{(2)}$ are the estimators used to verify Props 1.1, 1.2. They become closer for larger $N$.}
\vspace{-0.4cm}
\label{fig:diff-omega}
\end{wrapfigure}

We use two different estimates of $\omega_i$ in the relative entropy calculation $H(\mathrm{Law}( w_i)\,\vert\, \mu_K[\omega_i, \lambda_i])$ needed to check Property 1.1, and for the finite difference computation $N|\omega_{i+1}-\omega_{i}|$ in Property 1.2. We call these estimates $\hat\omega_i^{(1)}$ and $\hat\omega_i^{(2)}$, respectively.

For the relative entropy computation, we find $\lambda_i$ as $\lambda_i = \lambda_D(\E_\Delta[z_i])$, where $\E_\Delta[z_i]$ is computed by taking cumulative sums of the $\E_\Delta[ w_j]$ and then subtracting off a constant to ensure $\sum_i\E_\Delta[ z_i ]= 0$. Then 
\beq\label{omega-via-lamD}\hspace{-0.5cm}\hat\omega_i^{(1)} = \lambda_D(\E_\Delta[ z_i]) - \lambda_D(\E_\Delta[ z_{i-1}]).\eeq This determines the measure $\mu_K[\hat\omega_i^{(1)}, \lambda_i]$ to which we compare $\Law( w_i(t))$. To estimate the distributions $\Law( w_i(t))$, we compute empirical probability mass functions based on $1000$ independent samples of $ w_i(t)$.

To test the smoothness of $\omega_i$, we compute $\omega_i$ as 
\beq\label{omega-via-rate}\hat\omega_i^{(2)} = -\log\l(\E_\Delta [r( w_i)]\r)/(2K).\eeq Recall that $\mu_K[\omega_i,\lambda_i](r) = e^{-2K\omega_i}$. The estimates~\eqref{omega-via-lamD} and~\eqref{omega-via-rate} would be very close if $\Delta\ll1$ and $\Law( w_i) = \mu_K[\omega_i,\lambda_i]$ exactly. But for finite $N$, we find that $\hat\omega_i^{(2)}$ is a significantly smoother estimate of $\omega_i$. Figure~\ref{fig:diff-omega} confirms that the two estimates become closer as $N$ increases.

\section{Proofs from Section~\ref{sec:theory}}\label{app:proofs}Here we take Assumptions~\ref{gibbs-theory} to be true.
\begin{proposition}
If $\w(t,x)\in\Z$, then $\l(\E w_j\r)_{\argidxsetx j}$ is locally smoothly varying as $N\to\infty$, $\epsilon\to0$. In other words, 
$$\ol{\lim_{\epsilon\to0}}\;\ol{\lim_{N\to\infty}}\max_{\argidxsetx j}|\E w_j - \E w_{j-1}| = 0.$$  
\end{proposition}
\begin{proof} Assume for simplicity that $x=0$. We have
$\max_{|j|\leq N\epsilon}|\omega_j - \omega_{j-1}| \to0$ as $\Nepslim$. Let $|j|\leq N\epsilon$. Recall that $\E w_j= \omega_j +  u_o(\lambda_j) - u_o(\lambda_{j-1}),$ where $u_o$ is periodic and smooth. Taking finite differences,
\begin{equation}\begin{split}\label{del-omeg}
\l|\E w_j - \E w_{j-1}\r| \leq |\omega_j - \omega_{j-1}| +2|u_o(\lambda_j) - u_o(\lambda_{j-1})|.\end{split}\end{equation}
Now,  using that $\w:=\w(t,x)\in\Z$, we have
$\lambda_{j} = \lambda_0 + \sum_{\ell=1}^j\omega_\ell =\lambda_0+j\w + \epsilon_j\equiv \lambda_0 + \epsilon_j\mod 1,$ where $\epsilon_j = \sum_{\ell=1}^j\l(\omega_\ell-\w\r)$. Hence for all $|j|\leq N\epsilon$ we have
\begin{equation}\begin{split}
\l|u_o(\lambda_{j}) -u_o(\lambda_{j-1})\r|  &= \l|u_o(\lambda_0 + \epsilon_{j}) - u_o(\lambda_0 + \epsilon_{j-1})\r|\\
&\leq \|u_o'\|_\infty\l|\epsilon_{j}-\epsilon_{j-1}\r| =\|u_o'\|_\infty  |\omega_j-\w|.\end{split}\end{equation}Substituting this bound into~\eqref{del-omeg}, we get
\begin{equation}\l|\E w_j - \E w_{j-1}\r|\leq  |\omega_j-\omega_{j-1}|+2\|u_o'\|_\infty\|\omega_j-\w|.\end{equation} Both summands on the right go to zero uniformly over $|j|\leq N\epsilon$ as $N\to\infty$, $\epsilon\to0$.
\end{proof}

We now turn to the proof of Proposition~\ref{prop:aud}. Recall that $$\mu_K[\omega,\lambda](n) =\e\l(-\frac K2(n-\omega)^2\r)Q(n,\omega,\lambda),$$ where $Q$ is a function bounded away from zero such that $Q(n,\omega,\lambda) = Q(n\bmod2,\omega\bmod2,\lambda\bmod1)$, smooth in the second two arguments. Recall
\beq
\muinfty[\omega](n) =\int_{-\frac12}^{\frac12}\mu_K[\omega, \lambda](n)d\lambda.
\eeq  
\begin{lemma}
Let $v_{N,\epsilon}$ converge to $v$ as $N\to\infty$ and then $\epsilon\to0$. There exists $C_{N,\epsilon}=C_{N,\epsilon}(v)>0$ such that
\beq\label{vNeps}\l| \muinfty[v_{N,\epsilon}](n) - \muinfty[v](n)\r| \leq C_{N,\epsilon}e^{-Kn^2/3},\quad\forall n\in\Z,\eeq where 
$\limNeps C_{N,\epsilon}=0$.
\end{lemma}
\begin{proof}
We use the following bound: if $|w|,|w'|<R$, then  
\beq\label{app:ineq}\l|\mu[w,\lambda](n) - \mu[w',\lambda](n)\r| \leq C(R)|w-w'|e^{-\frac K3n^2}\eeq for some $C(R)$ and for all $n\in\Z$. Indeed, using that $Q$ is periodic in all variables and smooth in the second and third variable, we have that $\|Q\|_\infty$ and $\|\partial_2Q\|_\infty$ are finite. Hence
\beqs\label{temp}
\l|\mu[w,\lambda](n) - \mu[w',\lambda](n)\r| \leq &\|Q\|_\infty\l|e^{-\frac K2(n-w)^2}-e^{-\frac K2(n-w')^2}\r| \\
&+ e^{-\frac K2(n-w)^2}|w-w'| \|\partial_2Q\|_\infty
\eeqs From here, basic calculus gives~\eqref{app:ineq}. Integrating this bound over $\lambda\in[-1/2,1/2]$ and substituting $w=v_{N,\epsilon}$, $w'=v$, gives~\eqref{vNeps} with $C_{N,\epsilon} = C(2v) |v_{N,\epsilon}-v|.$ \end{proof}
Using this lemma with $v_{N,\epsilon}=\E\ol{\bf w}_{\idxsetx }$ and $v=\w$,  Proposition~\ref{prop:aud}  simplifies to showing the following:
\begin{proposition}\label{app:prop:aud} Let $\omega_i=\omega_i^N, \,|i|\leq N\epsilon$ satisfy 
\begin{alignat}{2}\label{omega-bds}
\ol{\lim_{\epsilon\to0}}\;\ol{\lim_{N\to\infty}}&\max_{|i|\leq N\epsilon}N|\omega_i-\omega_{i-1}| < C<\infty,\quad&\text{(Assumption 2.2)}& \nonumber\\
\ol{\lim_{\epsilon\to0}}\;\ol{\lim_{N\to\infty}}&\max_{|i|\leq N\epsilon}|\omega_i-\w|=0,\quad&\text{(Lemma~\ref{dlam-converge})},&
\end{alignat} for some $\w$ irrational. Let $\lambda_i$ be some numbers such that $\omega_i = \lambda_i-\lambda_{i-1}$. Then  
\beq\label{app:prop:bd}
\l|\frac{1}{2N\epsilon}\sum_{|i|\leq N\epsilon}\mu_K[\omega_i,\lambda_i](n) - \muinfty[\w](n)\r| \leq C_{N,\epsilon}\e\l(-Kn^2/3\r)\eeq for all $n\in\Z$, where $C_{N,\epsilon}\to0$ as $\Nepslim$.  
\end{proposition}
\begin{proof} Using the bound~\eqref{app:ineq}, we get that
\beq
\l|\frac{1}{2N\epsilon}\sum_{|i|\leq N\epsilon}\mu_K[\omega_i,\lambda_i](n) - \frac{1}{2N\epsilon}\sum_{|i|\leq N\epsilon}\mu_K[\w,\lambda_i](n)\r| \leq C_{N,\epsilon}\e\l(-Kn^2/3\r),\eeq where $C_{N,\epsilon}=C(\w)\max_{|i|\leq N\epsilon}|\omega_i-\w|$. The proof then follows from
\end{proof}
\begin{lemma}\label{app:lma:aud}
Let $\omega_i$, $\lambda_i$, and $\w$ be as in Proposition~\ref{app:prop:aud}. Then
\beq\label{app:lma:bd}
\l|\frac{1}{2N\epsilon}\sum_{|i|\leq N\epsilon}\mu_K[\w,\lambda_i](n) - \muinfty[\w](n)\r|\leq C_{N,\epsilon}\e\l(-Kn^2/3\r),\quad\forall n\in\Z,
\eeq where $C_{N,\epsilon}\to0$ as $\Nepslim$.  
\end{lemma}
\begin{proof}[Proof of Lemma~\ref{app:lma:aud}] 
First, we have
\beqs\bigg|\frac{1}{2N\epsilon}\sum_{|i|\leq N\epsilon}&\mu_K[\w,\lambda_i](n) - \muinfty[\w](n)\bigg| \\&= e^{-\frac K2(n-\w)^2}\l|\frac{1}{2N\epsilon}\sum_{|i|\leq N\epsilon}Q(n,\w,\lambda_i) - \int_{0}^{1}Q(n,\w,\lambda)d\lambda\r|,\eeqs
and there exists some $C=C(\w)$ such that $e^{-\frac K2(n-\w)^2} \leq Ce^{-\frac K3n^2}$ for all $n\in\Z$. Hence it remains to show the term inside the absolute values goes to zero as $\Nepslim$ for $n=0$ and $n=1$ (since $Q$ depends on $n$ through its parity only).
Define the measure 
$$P^{N,\epsilon} = \frac{1}{2N\epsilon}\sum_{|i|\leq N\epsilon}\delta_{\lambda_i\bmod1}.$$ Let $Q(\lambda) = Q(n,\omega,\lambda)/\|Q\|_\infty$, $n=0,1$. (The proof does not depend on $n$ or $\w$.) We have
$$\l|\sum_{|i|\leq N\epsilon}\frac{Q(n,\omega,\lambda_i)}{2N\epsilon} - \int_0^1Q(n,\omega,\lambda)d\lambda\r|=\|Q\|_\infty\l|\int Q(\lambda)P^{N,\epsilon}(d\lambda) - \int Q(\lambda)d\lambda\r|.$$ 
We will show that for each step function $g$ on $[0,1]$, with $\|g\|_\infty\leq 1$, we have \beq\label{step-fxn}\ol{\lim_{\epsilon\to0}}\;\ol{\lim_{N\to\infty}}\l|\int g(\lambda)P^{N,\epsilon}(d\lambda) - \int g(\lambda)d\lambda\r|=0.\eeq Since $Q$ is continuous, for each $\delta>0$ we can find a step function $g$ with $\|g\|_\infty\leq 1$ such that $\|g-Q\|_\infty \leq\delta$, and hence 
\beqs
\bigg|\l|\int g(\lambda)P^{N,\epsilon}(d\lambda) - \int g(\lambda)d\lambda\r| &- \l|\int Q(\lambda)P^{N,\epsilon}(d\lambda) - \int Q(\lambda)d\lambda\r|\bigg|\\&\leq \int|Q-g|P^{N,\epsilon}(d\lambda) + \int|Q-g|d\lambda \leq 2\delta. \eeqs Then using~\eqref{step-fxn}, it follows that
$$\ol{\lim_{\epsilon\to0}}\;\ol{\lim_{N\to\infty}}\l|\int Q(\lambda)P^{N,\epsilon}(d\lambda) - \int Q(\lambda)d\lambda\r|\leq 2\delta,$$ and since $\delta>0$ is arbitrary, the limit is zero. 

Returning to~\eqref{step-fxn}, let $g(\lambda) = \sum_{k=1}^mc_k\mathbbm{1}_{I_k}(\lambda)$ for intervals $I_k$. We have
\beqs
\l|\int g(\lambda)P^{N,\epsilon}(d\lambda) - \int g(\lambda)d\lambda\r| &\leq m\sup_{k=1,\dots, m}|P^{N,\epsilon}(I_k) - |I_k|| \\&\leq m\sup_A|P^{N,\epsilon}(A) - |A||,\eeqs where the supremum is over all intervals $A\subset [0,1]$. Thus it suffices to show
$$\ol{\lim_{\epsilon\to0}}\;\ol{\lim_{N\to\infty}}\sup_A|P^{N,\epsilon}(A) - |A||=0.$$ For the proof of this statement, see Lemma~\ref{lma:erdos}
\end{proof}

We use the Erd\H{o}s-Tur\'{a}n inequality~\cite{erdos}:
\beq\label{app:erdos-turan}
\sup_A|\nu(A) - |A|| \leq C\l(\frac1n + \sum_{m=1}^n\frac{|\hat\nu(m)|}{m}\r)
\eeq
for an absolute constant $C$ and any $n=1,2,\dots$, where $\nu$ is a measure on $[0,1]$ and $\hat\nu(m) = \int \e\l(2\pi i mx\r)\nu(dx).$
\begin{lemma}\label{lma:erdos}
If $\omega\in\mathbb Q^C$, then $\sup_A|P^{N,\epsilon}(A) - |A||\to0$ as $N\to\infty,\epsilon\to0$.
\end{lemma}
\begin{proof}
We assume for simplicity that $N\epsilon$, $N^{1/3}\epsilon$, and $N^{2/3}$ are all integers. We also assume that $P^{N,\epsilon}$ is actually the average of the point masses $\lambda_i\bmod 1$ over $i=\{1-N\epsilon, \dots, N\epsilon\}$ (omitting $i=-N\epsilon$). The error these assumptions occur is negligible in the limit. Now, break up the indices $\{1-N\epsilon, \dots, N\epsilon\}$ into subsets of $N^{1/3}\epsilon$ consecutive indices.  We label the subsets $I_k$, $|k|\leq N^{2/3}$, with
$$I_k = \{i_k+1,\dots, i_k+N^{1/3}\epsilon\},\quad i_k=kN^{1/3}\epsilon.$$
We use the subintervals $I_k$ to write $\hat P^{N,\epsilon}(m)$. 
\beq\label{Phat}
\hat P^{N,\epsilon}(m) = \frac{1}{2N\epsilon}\sum_{|k|\leq N^{2/3}}\sum_{j=1}^{N^{1/3}\epsilon}\e\l(2\pi im\lambda_{i_k+j}\r)
\eeq
Now, define $\|\delta(\bm{\omega})\|_\infty = \max_{|i|\leq N\epsilon}|\omega_i - \omega_{i-1}|$. We have
$$ \l|\omega_{i_k+\ell} - \omega_{i_k}\r|\leq \ell\|\delta(\bm{\omega})\|_\infty,\quad |k|\leq N^{2/3},\;\ell=1,\dots, N^{1/3}\epsilon,\vspace{-0.2cm}$$ so that
\beq\label{barz-barw-eps}
\lambda_{i_k+j}= \lambda_{i_k} + \sum_{\ell=1}^j\omega_{i_k+\ell} = \lambda_{i_k} + j\omega_{i_k} + \kappa_{j,k},\eeq where
\beq\label{eps-jk}|\kappa_{j,k}|\leq \sum_{\ell=1}^j\ell\|\delta(\bm{\omega})\|_\infty  \leq B_{N,\epsilon}:=N^{2/3}\epsilon^2\|\delta(\bm{\omega})\|_\infty\eeq for all $|k|\leq N^{2/3}$ and $j=1,\dots, N^{1/3}\epsilon$. 
Let \beq\label{delta}\kappa_{m,N,\epsilon} =  \frac{1}{2N\epsilon}\sum_{|k|\leq N^\frac23}\sum_{j=1}^{N^{1/3}\epsilon}\l|\e\l(2\pi im\kappa_{j,k}\r)-1\r|.\eeq Using~\eqref{eps-jk}, we have $|\kappa_{m,N,\epsilon}| \leq \e\l(2\pi mB_{N,\epsilon}\r)-1.$ Now, substituting~\eqref{barz-barw-eps} into~\eqref{Phat} and using the definition~\eqref{delta} of $\kappa_{m,N,\epsilon}$, we get
\beqs\label{P-est}
|\hat P^{N,\epsilon}(m)| &\leq \frac{1}{2N\epsilon}\sum_{|k|\leq N^{2/3}}\l|\sum_{j=1}^{N^{1/3}\epsilon}\e\l(2\pi im\l[\lambda_{i_k} + j\omega_{i_k}\r]\r)\r| + \kappa_{m,N,\epsilon}\\
&=  \frac{1}{2N\epsilon}\sum_{|k|\leq N^{2/3}}\l|\sum_{j=1}^{N^{1/3}\epsilon}\e\l(2\pi ijm\omega_{i_k}\r)\r| + \kappa_{m,N,\epsilon}.
\eeqs
Now, fix $n\in\mathbb N$, and define
$$\delta(n) = \min_{q=1,\dots,n}\;\min_{p=0,\dots,q-1}\l|\l(\w\bmod1\r)-\frac pq\r|.$$ Note that $\delta(n)$ is at most $1/n$. We know $\delta(n)>0$ because $\w$ is irrational. By the second bound in~\eqref{omega-bds}, there exists $\epsilon_1$ such that for all $\epsilon<\epsilon_1$, $N>N(\epsilon)$, we have
\beq
|\omega_{i_k}-\w|=|\omega_{i_k}\bmod1-\w\bmod1|< \delta(n)/2,\quad\forall |k|\leq N^{\frac23}.
\eeq 
For such $\epsilon,N$, and all $|k|\leq N^\frac23$, we have
\beq\label{min-qp}\min_{q=1,\dots,n}\;\min_{p=0,\dots,q-1}\l|\l(\omega_{i_k}\bmod1\r)-\frac pq \r|\geq \delta(n) - \delta(n)/2 = \delta(n)/2.\eeq We show that~\eqref{min-qp} implies that for all $m=1,\dots,n$, we have $|(m\omega_{i_k})\bmod1|\geq m\delta(n)/2\geq \delta(n)/2$. First, assume without loss of generality that $\omega_{i_k}\in [0,1)$. Then since $m\omega_{i_k}\in [0,m)$, we can write 
$$(m\omega_{i_k})\bmod 1 = \min_{j=0,\dots,m-1}|m\omega_{i_k}-j|.$$ Using~\eqref{min-qp}, we then have
\beqs\min_{j=0,\dots,m-1}|m\omega_{i_k}-j|&= m\min_{j=0,\dots,m-1}\l|\omega_{i_k}-j/m\r|\\&\geq m\min_{m=1,\dots,n}\min_{j=0,\dots,m-1}\l|\omega_{i_k}-j/m\r|\geq m\delta(n)/2.\eeqs
Now we return to~\eqref{P-est}. We can use the geometric sum formula for the sum inside the absolute value, since $m\omega_{i_k}$ is not an integer. Let $C(n) = 2/\sqrt{2-2\cos(\pi\delta(n))}$. We have
\beqs\label{P-est-II}
|\hat P^{N,\epsilon}(m)| &\leq \frac{1}{2N\epsilon}\sum_{|k|\leq N^{2/3}}\frac{2}{\l|1-\e\l(2\pi im\omega_{i_k}\r)\r|}+\kappa_{m,N,\epsilon}\\
&=  \frac{1}{2N\epsilon}\sum_{|k|\leq N^{2/3}}\frac{2}{\sqrt{2-2\cos(2\pi m\omega_{i_k})}}+ \kappa_{m,N,\epsilon}\\
&\leq  C(n)N^{-1/3}\epsilon^{-1} + \kappa_{m,N,\epsilon},\quad m=1,\dots, n.
\eeqs
Finally, we apply this bound, together with $\kappa_{m,N,\epsilon} \leq \e\l(2\pi nB_{N,\epsilon}\r)-1$, in~\eqref{app:erdos-turan}. We get that for all $\epsilon<\epsilon_1$ and $N>N(\epsilon)$,
$$
\sup_A|P^{N,\epsilon}(A) - |A|| \leq C\l(\frac1n + nC(n)N^{-1/3}\epsilon^{-1}+ n\l[\e\l(2\pi nB_{N,\epsilon}\r)-1\r]\r).$$
Since we take $N\to\infty$ before taking $\epsilon\to0$, the second summand above will converge to zero in the limit. Further, recall that $B_{N,\epsilon}=\epsilon^2N^{2/3}\|\delta(\bm{\omega})\|_\infty$. By the first bound in~\eqref{omega-bds}, $B_{N,\epsilon}$ goes to zero as $\Nepslim$, so the third summand also goes to zero. We conclude that
$$\ol{\lim_{\epsilon\to0}}\;\ol{\lim_{N\to\infty}}\sup_A|P^{N,\epsilon}(A) - |A||\leq\frac Cn.$$ But $n$ was arbitrary, so we take $n\to\infty$ to conclude.
\end{proof}

\bibliographystyle{siamplain}
\bibliography{bibliogr}

\vspace{0.4cm}
\tiny{Disclaimer: This report was prepared as an account of work sponsored by an agency of the United States Government. Neither the United States Government nor any agency thereof, nor any of their employees, makes any warranty, express or implied, or assumes any legal liability or responsibility for the accuracy, completeness, or usefulness of any information, apparatus, product, or process disclosed, or represents that its use would not infringe privately owned rights. Reference herein to any specific commercial product, process, or service by trade name, trademark, manufacturer, or otherwise does not necessarily constitute or imply its endorsement, recommendation, or favoring by the United States Government or any agency thereof. The views and opinions of authors expressed herein do not necessarily state or reflect those of the United States Government or any agency thereof.}
\end{document}